\documentclass{amsart}

\usepackage{amssymb}
\usepackage{color}
\usepackage{tikz-cd}
\usepackage{stmaryrd}
\usepackage{hyperref}
\hypersetup{
    colorlinks=true,
    linkcolor=black,
    filecolor=blue,      
    urlcolor=blue,
    citecolor=blue,
    pdfpagemode=FullScreen,
    }
\usepackage[margin=1.6 in, top=1.4in, bottom=1.4 in]{geometry}
\usepackage{enumitem}

\newtheorem{thm}{Theorem}[section]
\newtheorem{lem}[thm]{Lemma}

\newtheorem{prop}[thm]{Proposition}
\newtheorem{claim}[thm]{Claim}
\newtheorem{conj}[thm]{Conjecture}
\newtheorem{thm*}{Theorem}

\theoremstyle{definition}
\newtheorem{definition}[thm]{Definition}

\theoremstyle{remark}
\newtheorem{remark}[thm]{Remark}

\numberwithin{equation}{section}

\def\mrm#1{{\mathrm{#1}}}

\newcommand{\brat}[1]{{\left< #1 \right>}}

\newcommand{\R}{{\mathbb{R}}}
\newcommand{\Z}{{\mathbb{Z}}}
\newcommand{\C}{{\mathbb{C}}}
\newcommand{\Q}{{\mathbb{Q}}}

\newcommand{\D}{{\mathbb{D}}}
\newcommand{\bK}{{\mathbb{K}}}
\newcommand{\bL}{{\mathbb{L}}}
\newcommand{\bF}{{\mathbb{F}}}

\newcommand{\rH}{{\mathrm{H}}}

\newcommand{\ra}{\rightarrow}
\newcommand{\del}{\partial}

\newcommand{\cL}{\mathcal{L}}

\newcommand{\til}[1]{\widetilde{#1}}
\newcommand{\wh}[1]{\widehat{#1}}
\newcommand{\ol}[1]{\overline{#1}}

\newcommand{\lamzero}{\Lambda^0}

\newcommand{\overbar}{\overline}
\newcommand{\om}{\omega}

\newcommand{\ga}{\gamma}

\newcommand{\eps}{\epsilon}

\newcommand{\cA}{\mathcal{A}}
\newcommand{\cB}{\mathcal{B}}
\newcommand{\cC}{\mathcal{C}}
\newcommand{\cD}{\mathcal{D}}

\newcommand{\cF}{\mathcal{F}}
\newcommand{\cG}{\mathcal{G}}
\newcommand{\cH}{\mathcal{H}}

\newcommand{\cJ}{\mathcal{J}}
\newcommand{\cK}{\mathcal{K}}

\newcommand{\cP}{\mathcal{P}}

\newcommand{\cM}{\mathcal{M}}

\newcommand{\rP}{\mathrm{P}}

\newcommand{\rS}{\mathrm{S}}
\newcommand{\fix}{\mathrm{Fix}}

\renewcommand{\hat}{\wh}

\DeclareMathOperator{\rank}{\mathrm{rank}}

\DeclareMathOperator{\Ham}{\mathrm{Ham}}

\DeclareMathOperator{\im}{\mathrm{im}}

\DeclareMathOperator{\id}{\mathrm{id}}
\DeclareMathOperator{\spec}{\mathrm{Spec}}
\DeclareMathOperator{\Spec}{\mathrm{Spec}}
\DeclareMathOperator{\loc}{\mathrm{loc}}

\DeclareMathOperator{\cf}{\mathrm{CF}}
\DeclareMathOperator{\hf}{\mathrm{HF}}
\DeclareMathOperator{\h}{\mathrm{H}}
\DeclareMathOperator{\qh}{\mathrm{QH}}

\DeclareMathOperator{\cz}{\mathrm{CZ}}

\DeclareMathOperator{\pss}{\mathrm{PSS}}

\newcommand{\hlemail}{han.lou@uga.edu}
\newcommand{\msaemail}{marcelo.sarkis.atallah@umontreal.ca}

\begin{document}

\author{Marcelo S. Atallah}
\address{Marcelo S. Atallah, Department of Mathematics and Statistics,
	University of Montreal, C.P. 6128 Succ.  Centre-Ville Montreal, QC
	H3C 3J7, Canada}
\email{\msaemail}

	\author{Han Lou}
\address{Han Lou, Department of Mathematics,
	University of Georgia, Athens, GA 30602, USA}
\email{\hlemail}

\title{On the Hofer-Zehnder conjecture for semipositive symplectic manifolds}

\begin{abstract}
We show that, on a closed semipositive symplectic manifold with semisimple quantum homology, any Hamiltonian diffeomorphism possessing more contractible fixed points, counted homologically, than the total Betti number of the manifold, must have infinitely many periodic points. This generalizes to the semipositive setting the beautiful result of Shelukhin on the Hofer-Zehnder conjecture. The key component of the proof is a new study of the effect of reduction modulo a prime on the bounds on filtered Floer homology that arise from semisimplicity. This relies on the theory of algebraic extensions of non-Archimedean normed fields.

\end{abstract}

\maketitle
\setcounter{tocdepth}{1}
\tableofcontents

\section{Introduction}
\subsection{Introduction and Main Results}\label{sec: intro}
The Hofer-Zehnder conjecture, concerning the existence of infinitely many periodic points in Hamiltonian systems, originated from studies in celestial mechanics. In one of Poincaré's remarkable contributions to the three-body problem, he proved the existence of infinitely many periodic solutions provided the planetary masses are sufficiently small. In attempting to generalize such a result for larger planetary masses, he was led to conjecture the existence of at least two fixed points of any time-one map of an area-preserving isotopy of the annulus satisfying a twist condition on the boundary \cite{poincare1912theoreme}. He proved some particular cases of the conjecture, while Birkhoff established it in general; this is the content of the celebrated Poincaré-Birkhoff theorem \cite{birkhoff1913proof, birkhoff1926extension}. It was later shown \cite{brown1977proof} that a similar proof yields the existence of infinitely many periodic points\footnote{For $\phi:M\ra M$, a point $x\in M$ is periodic if there is an integer $k\geq 1$ such that $\phi^{k}(x)=x$.}. In a similar vein, Franks proved that, without the twisting condition, the existence of an interior fixed point is sufficient to guarantee the existence of infinitely many periodic points \cite{franks1992geodesics, franks1996area}. Furthermore, in what is one of the first results in the direction of the Hofer-Zehnder conjecture, and probably central to its motivation, Franks showed that any time-one map of an area-preserving isotopy of the sphere with at least three fixed points must have infinitely many periodic points. Note that having at least three fixed points is essential since any rotation of the sphere by an irrational fraction of $\pi$ has only two fixed points, which are also the only periodic points, namely, the North and South poles.

A conjecture by Arnol'd \cite{arnold2014stability, arnold1965proprietes}, relating Hamiltonian dynamics to the topology of the ambient symplectic manifold, has been a driving force of symplectic topology. A homological version of the conjecture states that any non-degenerate\footnote{Non-degeneracy means that for every fixed point $x$, the linearization $d\phi_{x}$ of $\phi$ at $x$ does not have $1$ as an eigenvalue. Equivalently, the graph of $\phi$ intersects the diagonal $\Delta$ transversely in $M\times M$.} Hamiltonian diffeomorphism $\phi$ of a closed symplectic manifold must have at least $\dim\h_{*}(M;\bK)$ fixed points, where $\bK$ is a choice of coefficient field. After some particular cases of the conjecture were established \cite{eliashberg1979estimates, fortune_weinstein1985symplectic, fortune1985symplectic, conley1983birkhoff, gromov1985pseudo}, Floer proved it for all symplectically aspherical and spherically monotone symplectic manifolds \cite{Floer1, Floer2, Floer3}. To this end, Floer developed a groundbreaking homology theory generated by the contractible orbits of a Hamiltonian diffeomorphism, inspired by the seminal work of Gromov, who introduced the study of pseudo-holomorphic curves to symplectic topology \cite{gromov1985pseudo}. With the advent of Floer theory the homological lower bound is achieved for the number of contractible fixed points\footnote{A fixed point $x$ of a Hamiltonian diffeomorphism $\phi$, generated by a Hamiltonian function $H$, is said to be contractible whenever the loop $\{\phi_{H}^t(x)\}_{t\in[0,1]}$ is contractible in $M$. This notion is independent of the choice of $H$ since loops of Hamiltonian diffeomorphisms have contractible orbits.} of a Hamiltonian diffeomorphism. Proving the conjecture in full generality, not only in terms of weakening the topological hypothesis on the symplectic manifold but also allowing for more general coefficient rings, is still the subject of ongoing research \cite{Hofer-Salamon, liu1998floer, fukaya1999arnold, ruan1999virtual, piunikhin1996symplectic, abouzaid2021arnold, bai2022arnold, rezchikov2022integral}; all of the proofs use the machinery invented by Floer. From a viewpoint, Franks's result on $\rS^{2}$ indicates that having more fixed points than what is required by the Arnol'd conjecture is sufficient to guarantee infinitely many periodic points. Indeed, any time-one map of an area-preserving isotopy on $\rS^2$  is Hamiltonian and, thus, must have at least two fixed points. Hofer and Zehnder conjectured the following for closed symplectic manifolds.
\begin{conj}[Hofer-Zehnder conjecture]\label{conj: HZ_1}
``One is tempted to conjecture that every Hamiltonian map on a compact symplectic manifold $(M,\om)$ possessing more fixed points than necessarily required by the V. Arnold conjecture possesses always infinitely many periodic orbits..."
\end{conj}
In the non-degenerate setting, a natural interpretation of Conjecture \ref{conj: HZ_1} leads to the statement that any Hamiltonian diffeomorphism $\phi$ of a closed symplectic manifold $(M, \om)$ possessing more contractible fixed points than the total Betti number  of $M$ must have infinitely many periodic points. Shelukhin suggested the following interpretation of Conjecture \ref{conj: HZ_1}.

\begin{conj}\label{conj: HZ_2}
Let $(M,\om)$ be a closed symplectic manifold and $\bK$ a choice of ground field. If $\phi$ is a (possibly degenerate) Hamiltonian diffeomorphism with finitely many fixed points satisfying
	\begin{equation*}
	N(\phi,\bK)=\displaystyle\sum_{x\in\fix(\phi)}\dim_{\bK}\hf^{\loc}(\phi,x) > \dim_{\bK}\h_{*}(M;\bK),	
	\end{equation*}
then, it must have infinitely many periodic points. 
 \end{conj}
Here, and throughout the text $\fix(\phi)$ denotes the contractible fixed points of $\phi$. Furthermore, $\hf^{\loc}(\phi,x)$ denotes a local version of Floer homology for a contractible fixed point $x$ of $\phi$. When $x$ is a non-degenerate fixed point, we have $\hf^{\loc}(\phi,x)\cong\bK$, recovering the non-degenerate interpretation of Conjecture \ref{conj: HZ_1}. In essence, one can think of $N(\phi,\bK)$ as a weighted sum of the number of contractible fixed points of $\phi$. In a striking result \cite[Theorem A]{shelukhin2022hofer}, Shelukhin proved Conjecture \ref{conj: HZ_2} for spherically monotone symplectic manifolds with semisimple even quantum homology; a class of manifolds that includes complex projective spaces, complex Grassmannians, and their products. Furthermore, if $\bK$ has characteristic $0$, then there is a simple $p$-periodic point for each sufficiently large prime $p$. This article aims to generalize Shelukhin's theorem to the semipositive setting. We obtain the following result:

\begin{thm}\label{main}
Let $(M,\om)$ be a closed semipositive symplectic manifold with semisimple even quantum homology $\qh_{ev}(M;\Lambda_{\bK})$ for a ground field $\bK$. Then, any Hamiltonian diffeomorphism $\phi$ with finitely many contractible fixed points such that 
	\begin{equation*}
	N(\phi,\bK)=\displaystyle\sum_{x\in\fix(\phi)}\dim_{\bK}\hf^{\loc}(\phi,x)  > \dim_{\bK}\h_{*}(M;\bK)	
	\end{equation*}
must have infinitely many periodic points. If $\bK$ has characteristic zero, then $\phi$ has a simple\footnote{A fixed point of $\phi^k$ is said to be \textit{simple} if it is not a fixed point of $\phi^l$ for any proper divisor $l$ of $k$.} contractible $p$-periodic point for each sufficiently large prime $p$. 
\end{thm}
\begin{remark}
The main difficulty that is circumvented to generalize Shelukhin's theorem to the non-monotone setting is to establish an upper bound of the boundary depth with $\bF_p$ coefficients (for a sufficiently large $p$) of a Hamiltonian diffeomorphism of a symplectic manifold satisfying the hypothesis of Theorem \ref{main}. Moreover, in order to obtain the existence of a contractible $p$-periodic point for each sufficiently large prime $p$ the bound obtained has to be independent of the prime $p$. In the monotone setting, Shelukhin exploits the relation between action and index of contractible periodic orbits to obtain such a bound. This is not possible in our setting. To overcome this issue, we work with quantum homology with coefficients in the algebraic closure of the universal Novikov field. This allow us to extract a good relation between the idempotents, before and after reducing mod p, and an upper bound for their valuation which is independent of the prime p provided it is sufficiently large. As detailed in Sections \ref{sec: semi} and \ref{sec: upper}, this is sufficient to obtain the desired bound. 
\end{remark}

\begin{remark}
There are known examples of closed semipositive symplectic manifolds that are not monotone and that have semisimple even quantum homology. Indeed, following \cite{ostrover2006calabi}, $(\rS^2\times\rS^2,\om\oplus\lambda\om)$, for $\lambda>1$, and the symplectic one point blow-up $(\C P^{2}\#\ol{\C P^{2}},\om_{\mu})$ of $\C P^{2}$, where $\om_{\mu}$ integrates to $\mu\in(0,1)\setminus\{1/3\}$ on the exceptional divisor and to $1$ on $[\C P^1]$, are non-monotone, semipositive, and have semisimple even quantum homology. Furthermore, any toric Fano manifold has semisimple even quantum homology for a generic toric symplectic form \cite{iritani2007convergence, fukaya2010lagrangian, ostrover2009quantum}, and if the even quantum homology of a toric Fano manifold $X$ is semisimple for the distinguished monotone symplectic form, then it is also semisimple for any toric symplectic form on $X$. If in addition $X$ is at most six-dimensional, it is also semipositive and, therefore, within the scope of Theorem \ref{main}. Recent work of Bai and Xu \cite{bai2023hofer} confirms Conjecture \ref{conj: HZ_2} for all compact toric symplectic manifolds. There are non-toric, non-monotone symplectic manifolds that satisfy the conditions of Theorem \ref{main}. For example, $(Q_3\times\C P^2, \om_{FS}\oplus \lambda\cdot\om_{FS})$, where $\lambda>0$ and $Q_3$ is the quadric in $\C P^{4}$. Indeed, it follows from the arguments in the proof of \cite[Proposition 1.8]{mcduff2011topology}, that if $Q_3\times\C P^2$ were symplectomorphic to a toric manifold, then it would have the toric structure of $\C P^3\times \C P^2$, which contradicts the fact that their Chern numbers are distinct. Moreover a $k$-point blow-up of size $\eps$ of $\C P^{2}$, is non-toric if $1/\eps$ is an integer and $k\geq 4$; see \cite{bayer2004semisimple, usher2011deformed, karshon2007circle}.
\end{remark}

The Hofer-Zehnder conjecture is related to a conjecture by Conley \cite{conley1983birkhoff, conley1984subharmonic}, which postulates that for a broad class of symplectic manifolds, any Hamiltonian diffeomorphism must have infinitely many periodic points. Therefore, the Hofer-Zehnder conjecture is automatically satisfied when the Conley conjecture holds. Following plenty of previous works that confirmed the conjecture in several cases \cite{salamon1992morse, conley1984subharmonic, conley1986symplectic, franks2003periodic, le2006periodic, hingston2009subharmonic, ginzburg2010conley, ginzburg2009action, hein2012conley, ginzburg2012conley}, Ginzburg and G\"{u}rel \cite{ginzburg2019conley} proved the most general statement known to hold. They showed that the existence of a Hamiltonian diffeomorphism with finitely many periodic points implies that there is a spherical homology class $A$ with $\brat{[\om],A}>0,\brat{c_{1}(M), A}>0$, in particular, the Conley conjecture holds for closed symplectically aspherical, negative-monotone and Calabi-Yau symplectic manifolds. However, the simple example of irrational rotation of the two-sphere shows that the Conley conjecture does not hold in general, and those are the cases where the Hofer-Zehnder conjecture is interesting. Furthermore, the Conley conjecture also fails for $\C\rP^{n}$ and complex Grassmannians, all of which have semisimple quantum homology \cite{entov2003calabi}.

Lastly, we point to other results in this direction. For instance, it was shown in \cite{ginzburg2014hyperbolic, ginzburg2018hamiltonian} that the presence of a hyperbolic fixed point prompts the existence of infinitely many periodic points. Furthermore, \cite{gurel2013non, ginzburg2016non, orita2017non, orita2019existence} showed that the existence of a non-contractible fixed point, which can also be considered ``unnecessary" from the Floer theoretic point of view, implies that there are infinitely many periodic points.

\subsection{Setup}
We recall some notions about Hamiltonian diffeomorphisms and symplectic topology required to follow an overview of the main result. 

\begin{definition}\label{def: semi-positive}
A closed symplectic manifold $(M,\om)$ is called \textit{semipositive} if for every sphere class $A$ in the image $\h^S_{2}(M;\Z)$ of the Hurewicz map $\pi_{2}(M)\ra\h_{2}(M;\Z)$ \[\brat{c_{1}(M),A}\geq 3-n,\quad \brat{[\om],A}>0 \quad\Rightarrow \quad\brat{c_{1}(M),A}\geq0.\] Where $c_{1}(M)\in\h^{2}(M;\Z)$ is the first Chern class\footnote{An \textit{almost complex structure} $J$ on $M$ is an automorphism of the tangent bundle $TM$ satisfying $J^{2}=-\id_{TM}$. It is called \textit{$\om$-compatible} if $\om(-,J-)$ is a Riemannian metric. The space $\cJ(M,\om)$ of $\om$-compatible almost complex structures is contractible, therefore, the first Chern class $c_{1}(M)=c_{1}(M,\om)$ of $(TM,J)$ is independent of $J\in\cJ(M,\om)$.} associated with the symplectic manifold.
\end{definition}
From this point forward, $(M,\om)$ denotes a closed semipositive symplectic manifold of dimension $2n$. Let $\phi=\phi_H$ be a Hamiltonian diffeomorphism generated by a Hamiltonian function\footnote{We usually consider normalized Hamiltonians $H$, i.e. $H(t,-)$ has zero $\om^n$ mean for all $t\in[0,1]$.} $H\in C^{\infty}(\R/\Z\times M,\R)$. We say that a fixed point $x$ of $\phi$ is contractible when the loop $\{\phi_{H}^{t}(x)\}_{t\in[0,1]}$ is contractible in $M$. It is a deep fact of symplectic topology that the class this loop represents in $\pi_{1}(M)$ is independent of the choice of $H$ generating $\phi$, see \cite[Remark 11.4.2]{McDuffSalamonIntro3}. We denote by $\fix(\phi)$ the collection of contractible fixed points. Observe that $\fix(\phi)$ is naturally included in $\fix(\phi^{k})$ for all $k\geq1$. For a fixed point $x$, we denote by $x^{(k)}$ its image under this inclusion. We call a fixed point $x$ of $\phi^k$ \textit{simple} if it is not a fixed point of $\phi^l$ for any proper divisor $l$ of $k$.

\begin{definition}
A fixed point $x$ of $\phi$ is called \textit{non-degenerate} if $1$ is not an eigenvalue of the linearized map $d\phi_{x}$. A Hamiltonian function $H$ is called non-degenerate if all contractible fixed points of $\phi_{H}$ are non-degenerate.
\end{definition}
To a Hamiltonian diffeomorphism $\phi\in\Ham(M,\om)$, one can associate a filtered Floer homology theory \cite{Floer1, Floer2, Floer3} with coefficients over a base field $\bK$ whose information, in the case where $\phi$ has finitely many fixed points, is partially captured by a finite set of strictly positive real numbers\footnote{The strict positivity follows from the fact that Hamiltonian Floer complex is strict in the sense of \cite{usher2016persistent}. } \[\beta_1(\phi,\bK)\leq\dots\leq\beta_{K(\phi,\bK)}(\phi,\bK),\] depending only on $\phi$ and $\bK$, called the \textit{bar-length spectrum of $\phi$}. Here, $K(\phi,\bK)$ denotes the number of bars in the bar-length spectrum of $\phi$. We describe these notions, which were first introduced in symplectic topology by Polterovich and Shelukhin \cite{polterovich2016autonomous} (see also \cite{polterovich2017persistence, usher2016persistent}) in Section \ref{sec: persistence}. The length $\beta(\phi,\bK)=\beta_{K(\phi,\bK)}(\phi,\bK)$ of the largest bar is called the \textit{boundary depth} and was introduced by Usher \cite{usher2011boundary, usher2013hofer}. By definition, the boundary depth $\beta(\phi,\bK)$ is zero when $K(\phi,\bK)$ is zero. We denote by \[\beta_{tot}(\phi,\bK)=\beta_1(\phi,\bK)+\dots+\beta_{K(\phi,\bK)}(\phi,\bK)\] the \textit{total bar-length}. Finally, Shelukhin \cite{shelukhin2022hofer} showed that
\begin{equation}\label{eq: nbk}
	N(\phi,\bK) = \dim_{\bK}\h_{*}(M;\bK) + 2K(\phi,\bK),
\end{equation}
where $N(\phi,\bK)$ is as in the statement of Theorem \ref{main}. This equality also holds in the semipositive setting, we summarize the details of the underlying theorem in Section \ref{sec: LFH}. In particular, the condition $N(\phi,\bK)>\dim_{\bK}\h_{*}(M;\bK)$ in the statement of the theorem implies that $\beta(\phi,\bK)$ is positive.
\begin{remark}\label{rmk: bar_field_ext}
We note that for sufficiently large primes $p$ we have by \cite[Lemma 16]{shelukhin2022hofer} the following equalities $N(\phi,\Q)=N(\phi,\bF_p), \dim_{\Q}\h_{*}(M;\Q)=\dim_{\bF_p}\h(M;\bF_p)$, and $\beta(\phi,\Q)=\beta(\phi,\bF_p)$ for any Hamiltonian diffeomorphism $\phi$. We note that how large the prime needs to be taken depends on the $\phi$. If $\bK$ has characteristic $0$, it is a field extension of $\Q$, which by \cite[Section 4.4.4]{shelukhin2022hofer} implies that $N(\phi,\Q)=N(\phi,\bK), \dim_{\Q}\h_{*}(M;\Q)=\dim_{\bK}(M;\bK)$, and $\beta(\phi,\Q)=\beta(\phi,\bK)$.
\end{remark}

\subsection{Overview}\label{sec: overview}
We summarize the proof of Theorem \ref{main} while pointing to the technical developments required to obtain the result in the semipositive setting. The full details of the proof are in Section \ref{sec:proof_main}.

There are two main components to the proof. One is a generalization to the semipositive setting, stated as Theorem \ref{thm: smith}, of the Smith type inequality
\begin{equation}\label{eq: smith_type}
	p\cdot\beta_{tot}(\phi,\bF_p)\leq\beta_{tot}(\phi^{p},\bF_p),
\end{equation}
shown to hold by Shelukhin for spherically monotone symplectic manifolds under the assumption that $\phi^p$ has finitely many fixed points. The main technical difficulty can be overcome using the recent work of Sugimoto \cite{sugimoto2021hofer}, generalizing the $\Z/(p)$-equivariant product-isomorphism to the semipositive setting. The Smith inequality is then obtained by following Shelukhin's proof \cite{shelukhin2022hofer} in the spherically monotone setting. The equivariant Floer theory required to obtain inequality (\ref{eq: smith_type}) is detailed in Section \ref{sec: smith}.

The other component of the proof, proven as Theorem \ref{thm:upper_bound}, is the main novelty of this paper. We uniformly bound the boundary-depth 
\begin{equation}\label{eq: depth}
	\beta(\phi^k,\bF_p)\leq C
\end{equation}
for sufficiently large iterations of $\phi$. When $(M,\om)$ is monotone and the even part of quantum homology is semisimple, Shelukhin \cite{shelukhin2022hofer} uses the relationship between indices and actions of contractible fixed points to obtain an upper bound for the boundary-depth that only depends on the dimension of the manifold. In the semipositive case, there is no such uniform relation between the actions and indices.
 Therefore, we work with quantum homology with coefficient field the algebraic closure of the universal Novikov field to have a good relation, before and after reducing coefficients, between the idempotents generating the even quantum homology of $M$. We are then able to bound the boundary-depth by a constant, independent of $p$, depending on the idempotents of the even quantum homology whose coefficient has characteristic $0$.

With the Smith-type inequality and the uniform bound on the boundary-depth established, we now summarize Shelukhin's argument to prove Theorem \ref{main}. First, we consider the case where $\bK$ has characteristic $0$ and the Hamiltonian diffeomorphism $\phi$ and all of its iterates are non-degenerate. By inequality (\ref{eq: smith_type}) and the simple observation that $\beta_{tot}(\psi,\bK)\leq K(\psi,\bK)\cdot\beta(\psi,\bK)$ for any Hamiltonian diffeomorphism $\psi$ and base field $\bK$, we have
\begin{equation*}\label{eq: main}
    p\cdot\beta_{\mathrm{tot}}(\phi, \bF_p)\le\beta_{\mathrm{tot}}(\phi^p, \bF_p)\le K(\phi^p, \bF_p)\cdot\beta(\phi^p, \bF_p).
\end{equation*}
The assumption that $N(\phi, \bK)>\dim_{\bK}H_*(M; \bK)$ and Remark \ref{rmk: bar_field_ext} imply that the total bar-length $\beta_{\mathrm{tot}}(\phi, \bF_p)$ is positive for a sufficiently large prime $p$. Furthermore, inequality (\ref{eq: depth}) yields
\begin{equation*}
    p\cdot\beta_{\mathrm{tot}}(\phi, \bF_p)\le C\cdot K(\phi^p, \bF_p),
\end{equation*}
which means that $K(\phi^p,\bK)$ grows at least linearly with respect to $p$. We now observe that in the non-degenerate setting Equation (\ref{eq: nbk}) yields
\begin{equation*}
	\fix(\phi^p) = \dim_{\bK}\h_{*}(M;\bK) + 2K(\phi^p,\bK),
\end{equation*}
which implies that $\phi$ must have infinitely many contractible periodic points. In general, when the Hamiltonian diffeomorphism $\phi$ is possibly degenerate, we can again achieve inequality (\ref{eq: main}) by the local equivariant Floer homology argument \cite[Section 7.4]{shelukhin2022hofer}. Furthermore, by the canonical complex whose properties are listed in Theorem \ref{thm: perturbation}, the upper bound for the boundary-depth, which is also independent of $p$, continues to hold. Therefore, we can use the same argument as in the non-degenerate case to obtain the linear growth of $K(\phi^p,\bK)$ and, thus, of $N(\phi^p,\bK)$. To conclude the argument, we assume that $p$ is large enough to guarantee $\phi^p$ is an admissible iteration in the sense of Definition \ref{def: admissible}, it then follows by Theorem \ref{A}, \cite[Theorem 1.1]{ginzburg2010local},\cite[Theorem C]{shelukhin2022hofer}, that $\hf^{\loc}(\phi^{p_1},x)\cong\hf^{\loc}(\phi^{p_2},x)$ for all $x\in\fix(\phi^{p_1})$ for any two primes $p_2\geq p_1\geq p$. In particular, there must be a new simple $p'$-periodic point for each prime $p'>p$. In fact, if $\fix(\phi^{p_1})=\fix(\phi^{p_2})$ for $p_2\geq p_1\geq p$, then $N(\phi^{p_1},\bK)=N(\phi^{p_2},\bK)$ contradicting the linear growth of $N(\phi^{p'},\bK)$ for $p'\geq p$. A similar argument works when $\bK$ has characteristic $p$ the details of which can be found in Section \ref{sec:proof_main}.

\section*{Acknowledgements}
We thank Egor Shelukhin and Michael Usher for bringing us together to work on this project, for their support, and for the numerous helpful discussions. We thank Guangbo Xu for pointing out a problem in Proposition \ref{prop:interleaving}, which has since been addressed. We also thank the anonymous reviewer for the helpful comments and suggestions. H.L. thanks Shengzhen Ning for the example of the four-point blowup of $\C P^2$ and for pointing out the work \cite{li2022enumerative} which was helpful to understand toric structures under blowups. This work is part of both authors' Ph.D. theses. H.L.'s Ph.D is taking place at the University of Georgia under the supervision of Michael Usher and M.S.A's Ph.D is being carried out at the Université de Montréal under the supervision of Egor Shelukhin. M.S.A  was partially supported by Fondation Courtois, by the ISM’s excellence scholarship, and by the J.Armand Bombardier’s excellence scholarhip. This material is based upon work supported by the National Science Foundation under Grant No. DMS-1928930, while M.S.A was in residence at the Simons Laufer Mathematical Sciences Institute (previously known as MSRI) Berkeley, California during the Fall 2022 semester.
\section{Preliminaries}
\subsection{Basic setup}
\subsubsection{Fixed points of Hamiltonian diffeomorphisms}
Let $(M,\om)$ be a closed symplectic manifold of dimension $2n$. Denote by $\cH$ the space of $1$-periodic Hamiltonian functions $H\in C^{\infty}(\R/\Z\times M,\R)$ normalized so that $H(t,-)$ has zero $\om^n$ mean for all $t\in[0,1]$. We denote by $X_{H}$ the time-dependent vector-field induced by Hamilton's equation \[\iota_{X_{H}^{t}}\om=-dH_{t},\] and by $\{\phi_{H}^{t}\}_{t\in[0,1]}$ the induced symplectic isotopy i.e. $\phi_{H}^{t}$ satisfies \[\frac{d\phi^{t}_{H}}{dt}=X_{H}^{t}\circ\phi_{H}^{t}\] with initial condition $\phi_{H}^{0}=\id$. We denote by $\phi_{H}=\phi_{H}^{1}$ its time-one map. Diffeomorphisms obtained in this manner are called \textit{Hamiltonian diffeomorphisms}. There is a bijective correspondence between the contractible $1$-periodic orbits of $X_{H}^{t}$ and the contractible fixed points $\fix(\phi_{H})$ of $\phi_{H}$, therefore, periodic points of $\phi_{H}$ correspond to \[\mrm{Per}(\phi_{H})=\bigcup_{k\geq1}\fix(\phi_{H}^{k}).\] 

For $F,G\in\cH$, we denote by $F\#G$ the Hamiltonian function \[F\#G(t,x)=F(t,x) + G(t,(\phi_{F}^{t})^{-1}(x))\] that induces the isotopy $\{\phi_{F}^{t}\phi_{G}^{t}\}$, in particular $H\#\cdots\#H$ ($k$-times) generates $\{(\phi_{H}^{t})^k\}$. Note that $H^{(k)}(t,x)=kH(kt,x)$ generates a homotopic path (rel. ends), therefore $\phi_{H^{(k)}}=\phi_{H}^{k}$. We denote by $\ol{H}$ the Hamiltonian function $\ol{H}(t,x)=-H(t,\phi_{H}^{t}(x))$, which generates $\{(\phi_{H}^{t})^{-1}\}$.

\subsubsection{The Hamiltonian action funcitonal}
Let $\cL M$ denote the component of contractible loops of the free loop-space of $M$. For a pair $(x,\ol{x})$ consisting of a loop $x\in\cL M$ and a smooth \textit{capping} $\ol{x}:\D\ra M$, with $\ol{x}|_{\del \D}=x$, consider the following equivalence relation: $(x,\ol{x})\sim(y,\ol{y})$ if and only if,
\begin{align*}
	x=y,\quad\text{and}\quad [\ol{x}\#(-\ol{y})]\in\ker[\om]\cap\ker c_{1}.
\end{align*}
Here, $\ol{x}\#(-\ol{y})$ stands for gluing the disks along their boundaries with the orientation of $\ol{y}$ reversed, and $[\ol{x}\#(-\ol{y})]$ its class in $\h^{S}_{2}(M;\Z)$. Denote by $\til\cL M$ the covering space of $\cL M$ given by the collection of such pairs modulo the equivalence relation $\sim$. For the sake of brevity, we shall write $\ol{x}$ instead of $(x,\ol{x})$. Note that, the group of deck transformations of $\til\cL M$ is isomorphic to \[G_{\om}=\pi_{2}(M)/\ker[\om]\cap\ker c_{1},\] where the transformation associated with $A\in\pi_{2}(M)$ is given by sending $\ol{x}$ to $\ol{x}\#A$. To a Hamiltonian function $H$, we associate an action functional $\cA_{H}:\til\cL M:\ra\R$ defined by
\begin{equation*}
	\cA_{H}(\ol{x})=\int^{1}_{0}H(t,x(t))dt - \int_{\ol{x}}\om.
\end{equation*}
The critical points of $\cA_{H}$ are the lifts $\til\cP(H)$ of the contractible $1$-periodic orbits $\cP(H)$ satisfying the equation $x'(t)=X_{H}^{t}(x(t))$. The \textit{action spectrum} of $H$ is defined as the subset of $\R$ given by the critical values $\Spec(H)=\cA_{H}(\til\cP(H))$ of the action functional. If $A\in G_{\om}$ then, \[\cA_{H}(\ol{x}\#A)=\cA_{H}(\ol{x})-\brat{[\om],A}.\] Furthermore, the action functional behaves well with iterations in the sense that \[\cA_{H^{(k)}}(\ol{x}^{(k)})=k\cA_{H}(\ol{x}),\] where $\ol{x}^{(k)}$ inherits the natural capping induced by $\ol{x}$.

\subsection{Filtered Hamiltonian Floer homology}\label{sec: Filtered_FH}
Floer theory was first developed by A. Floer in \cite{Floer1, Floer2, Floer3} as a generalization of Morse-Novikov homology, to prove the non-degenerate Arnold conjecture. We refer to \cite{Hofer-Salamon, McDuffSalamonIntro3, OhBook} and references therein for the details of the construction and to \cite{AbouzaidBook, SeidelMCG, Zap:Orient} for in-depth discussions of canonical orientations. We shall consider the construction of Hamiltonian Floer homology in the {semipositive} setting.

Let $(M,\om)$ be a closed symplectic manifold and $\bK$ a choice of base field. Let $H$ be a non-degenerate Hamiltonian function on $M$ and $J=\{J_{t}\}$ be a $1$-periodic family of $\om$-compatible almost complex structures. For $a\in\R\setminus\spec(H)$ the \textit{Floer chain complex at filtration level $a$} is defined by
\begin{equation*}
	\cf_{*}(H;J)^{<a}=\Big\{\sum a_{i}\ol{x}_{i}\,\Big|\,a_{i}\in\bK,\,\ol{x}_{i}\in\til\cP(H),\, \cA_{H}(\ol{x}_{i})<a\Big\},
\end{equation*}
where every summation satisfies the condition that the set $\{i\,|\,a_{i}\neq0, \cA_{H}(\ol{x}_{i})>c\}$ is finite for all $c\in\R$. The complex is graded by the Conley-Zehnder index, which assigns an integer $\cz(\ol{x})$ to each $\ol{x}\in\til\cP(H)$ of $\cA_{H}$ and satisfies $\cz(\ol{x}\#A)=\cz(\ol{x})-\brat{c_{1}(M),A}$. Note that $\cf_{*}(H;J)=\cf_{*}(H;J)^{+\infty}$ is naturally a finitely generated module over the Novikov ring
\begin{equation*}
	\Lambda_{\om,\bK} = \Bigg\{\sum_{A\in G_{\om}} a_{A}T^{\om(A)}\,\Big|\, a_{A}\in\bK,\,\#\{A\,|\,a_{A}\neq0, \om(A)<c\}<\infty,\,\forall c\in\R\Bigg\} .
\end{equation*}
A \textit{Floer trajectory} between capped orbits $\ol{x}_{-},\ol{x}_{+}$, is a smooth map $u:\R\times\rS^{1}\ra M$ satisfying the \textit{Floer equation}
\begin{equation*}
	\frac{\del u}{\del s} + J_{t}(u)\Bigg(\frac{\del u}{\del t} - X_{H}^{t}(u)\Bigg) = 0,
\end{equation*}
with asymptotics, \[\lim_{s\ra\pm\infty}u(s,t)=x_{\pm}(t),\] such that the capping $u\#\ol{x}_{+}$ is equivalent to $\ol{x}_{-}$, and $\cz(\ol{x}_{-})-\cz(\ol{x}_{+})=1$. In the semipositive setting, the compactified moduli space $\cM(\ol{x},\ol{y};J)$ of Floer-trajectories from $\ol{x}$ to $\ol{y}$ (modulo the natural $\R$-action) is a manifold of dimension $\cz(\ol{x})-\cz(\ol{y})-1$. The \textit{Floer differential} \[d_{F}:\cf_{k}(H;J)^{<a}\ra\cf_{k-1}(H;J)^{<a}\] is defined by 
\begin{equation*}
	d_{F}(\ol{x})=\sum_{\cz(\ol{x})-\cz(\ol{y})=1}\big(\#\cM(\ol{x},\ol{y};J)\big)\cdot\ol{y}.
\end{equation*}
It squares to zero and preserves the filtration induced by $\cA_{H}$. For an interval $I=(a,b)$, $a<b$, where $a,b\in\R\setminus\spec(H)$ we define the \textit{Floer complex in the action window $I$} as the quotient complex
\begin{equation*}
	\cf_{*}(H;J)^{I}=\cf_{*}(H;J)^{<b}/\cf_{*}(H;J)^{<a}.
\end{equation*}
We denote by $\hf_{*}(H)^{I}$ the resulting homology of this complex with the differential induced by $d_{F}$ and call it the \textit{Floer homology of $H$ in the action window $I$}. Note that it is independent of the generic choice of almost complex structure $J$. The (total) Floer homology $\hf_{*}(H)$ of $H$ is obtained by setting $a=-\infty$ and $b=+\infty$ and does not depend on the choice of Hamiltonian (by a standard continuation argument). Furthermore, for all $H\in\cH$, $\hf_{*}(H)^{I}$ depends only on the homotopy class of $\{\phi_{H}^{t}\}_{t\in[0,1]}$ in the universal cover $\widetilde\Ham(M,\om)$ of the group of Hamiltonian diffeomorphisms $\Ham(M,\om)$. For details; refer to \cite[Section 12.4]{mcduff2012j}.

When $H$ is degenerate, we have to consider perturbation data $\cD=(K^{H}, J^{H})$, where $K\in\cH$ is such that $H^{\cD}=H\#K^{H}$ is a non-degenerate Hamiltonian and $J^H$ is a choice of generic almost complex structure with respect to $H^{\cD}$. We take the action functional to be $\cA_{H;\cD}=\cA_{H^{\cD}}$. When $(M,\om)$ is rational, an admissible action window $I=(a,b)$ for $H$, i.e. $a,b\in\R\setminus\Spec(H)$, will remain so for $H^{\cD}$ for sufficiently $C^{2}$-small $K^H$, furthermore, the groups $\hf(H;\cD)^{I}$ are canonically isomorphic. Therefore, $\hf(H)^{I}$ is defined as the colimit of the induced directed system. The general case is dealt with by taking the colimit over partially ordered non-degenerate perturbations  whose action spectrums do not include $a$ or $b$;  refer to \cite{hein2012conley} or \cite[Section 2.2.2]{atallah2020hamiltonian} for a detailed exposition.

\subsection{Novikov ring and non-Archimedean valuation}
\subsubsection{Novikov ring and extending coefficients}
Let $R$ be a commutative unital ring. The \textit{universal Novikov ring over $R$} is defined as
\begin{equation}\label{eq: univ_nov_ring}
	\Lambda_{R} = \bigg\{\sum_{i=-K}^\infty a_iT^{\lambda_i}\,|\,a_i\in R,\,\lambda_i\nearrow+\infty\bigg\}.
\end{equation}
It follows from \cite[Theorem 4.2]{Hofer-Salamon} that $\Lambda_{{R}}$ is a principal ideal domain (resp. a field) whenever $R$ is a principal ideal domain (resp. a field). In particular, when $R$ is a principal ideal domain, every nonzero prime ideal in $\Lambda_{{R}}$ is maximal. We are particularly interested in the case $R=\Z$, when $R$ is not a field, and the case $R=\bF_{p}, \Q$, for $p$ prime, when $R$ is a field.

 For a base field $\bK$, denote by $\cf(H;\Lambda_{\bK,\om})$ the Floer chain complex defined in Section \ref{sec: Filtered_FH}, where we omit the choice of generic almost complex structure $J$. It will often be convenient to extend the coefficients to the universal Novikov field $\Lambda_{\bK}$ and to its algebraic closure $\ol\Lambda_{\bK}$. We define
\begin{equation*}
	\cf(H; \Lambda_{\bK}) = \cf(H; \Lambda_{\bK, \omega})\otimes_{\Lambda_{\bK, \omega}}\Lambda_{\bK}
\end{equation*}
and
\begin{equation*}
	\cf(H; \ol\Lambda_{\bK}) = \cf(H; \Lambda_{\bK})\otimes_{\Lambda_{\bK}}\ol\Lambda_{\bK},
\end{equation*}
the differentials are extended by linearity. In particular, if $\{x_1,\cdots,x_B\}$ is a $\Lambda_{\bK,\om}$-basis of $\cf(H;\Lambda_{\bK,\om})$, then the elements of $\cf(H; \Lambda_{\bK})$ (resp. $\cf(H; \ol\Lambda_{\bK})$) are of the form $\sum \lambda_ix_i$ where $\lambda_i\in\Lambda_{\bK}$ (resp. $\ol\Lambda_{\bK}$).

\subsubsection{Non-Archimedean valuation}\label{sec: non-Arch_val}
\begin{definition}\label{def: valuation}
A \textit{non-Archimedean valuation} on a field $\Lambda$ is a function \[\nu:\Lambda\ra\R\cup\{+\infty\}\] satisfying the following properties:
	\begin{enumerate}
	\item $\nu(x)=+\infty$ if and only if $x=0$,
	\item $\nu(xy)=\nu(x)+\nu(y)$ for all $x,y\in\Lambda$,
	\item $\nu(x+y)\geq \min\{\nu(x),\nu(y)\}$ for all $x,y\in\Lambda$. 
	\end{enumerate}
Furthermore, we set $\Lambda^0=\nu^{-1}([0,+\infty))$ to be the subring of elements of nonnegative valuation.
\end{definition}

It shall often be the case that $\Lambda = \Lambda_{\bK}$, where $\Lambda_{\bK}$ is the \textit{universal Novikov field over a ground field $\bK$},
\begin{equation}\label{eq: univ_nov}
	\Lambda_{\bK} = \bigg\{\sum_{i=-K}^\infty a_iT^{\lambda_i}\,|\,a_i\in\bK,\,\lambda_i\nearrow+\infty\bigg\}.
\end{equation}
In this case, $\Lambda$ can be endowed with a non-Archimedean valuation given by setting $\nu(0)=+\infty$ and
\begin{equation}\label{eq: valuation_1}
	\nu\bigg(\sum_{i=-K}^\infty a_iT^{\lambda_i}\bigg)=\lambda_{-K}
\end{equation}
on $\Lambda\setminus\{0\}$. The universal Novikov ring over $\Z$, denoted $\Lambda_{\Z}$, is defined just as in Equation (\ref{eq: univ_nov}), however, it is not a field. Its field of fractions $Q(\Lambda_{\Z})$ can be identified with a subfield of $\Lambda_{\Q}$, therefore, its elements can be as expressed as \[\sum_{i=-K}^\infty c_iT^{\lambda_i},\quad c_i\in\Q.\] We can hence define a valuation $\nu:Q(\Lambda_{\Z})\ra\R\cup\{+\infty\}$ as in Equation (\ref{eq: valuation_1}). 

\subsubsection{Field norms and extension of valuations}
The notion of valuation on a field $\Lambda$ is closely related to that of a field norm defined below.
\begin{definition}
A \textit{non-Archimedean norm on a field $\Lambda$} is a map $|\cdot|:\Lambda\ra\R_{\geq0}$ satisfying the following properties:
	\begin{enumerate}
	\item $|x|=0$ if and only if $x=0$,
	\item $|xy|=|x||y|$ for all $x,y\in\Lambda$,
	\item $|x+y|\leq\max\{|x|,|y|\}$ for all $x,y\in\Lambda$.
	\end{enumerate}
Furthermore, $\Lambda$ is said to be \textit{complete} with respect to $|\cdot|$ if it is a complete metric space with respect to the induced topology. 
\end{definition}
 Note that given a non-Archimedean valuation $\nu$, one can define a non-Archimedean norm by setting $|x|=e^{-\nu(x)}$. Conversely, if $|\cdot|$ is a non-Archimedean norm, a non-Archimedean valuation can be obtained by setting $\nu(x)=-\ln(|x|)$. The following results in \cite{cassels1986local} allow one to extend a given valuation to certain field extensions.

 \begin{prop}[Chapter 7, Theorem 1.1 in \cite{cassels1986local}]\label{prop: valuation_extension1}
Let $\bK$ be a field that is complete with respect to a norm $|\cdot|$ and let $\bL$ be a finite extension of degree $n$. Then, there is precisely one extension $\parallel\cdot\parallel$ of $|\cdot|$ to $\bL$. It is given by
\begin{equation*}
    \parallel A\parallel=|N_{\bL/\bK}(A)|^{1/n}
\end{equation*}
where $A\in\bL$ and $N_{\bL/\bK}(A)$ is the determinant of the map $B\mapsto AB$ for $B\in\bL$. Furthermore, $\bL$ is complete with respect to $\parallel\cdot\parallel$.
\end{prop}
 
 \begin{prop}[Chapter 9, Lemma 2.1 in \cite{cassels1986local}]\label{prop: valuation_extension2}
Let $\bL=\bK(A)$ be a separable extension and let $F(x)\in\bK[x]$ be the minimal polynomial for $A$. Let $\mathfrak{K}$ be the completion of $\bK$ with respect to a norm $|\cdot|$. Let $F(x)=\phi_1(x)\cdots\phi_J(x)$ be the decomposition of $F(x)$ into irreducibles in $\mathfrak{K}[x]$. Then the $\phi_j$ are distinct. Let $\bL_j=\mathfrak{K}(B_j)$ where $B_j$ is a root of $\phi_j(x)$. Then there is an injection
\begin{equation}\label{BB}
\bL\hookrightarrow\bL_j
\end{equation}
extending $\bK\hookrightarrow\mathfrak{K}$ under which $A\mapsto B_j$. Denote by $|\cdot|_j$ the norm on $\bL$ induced by equation \ref{BB} and the unique norm on $\bL_j$ extending $|\cdot|$. Then the $|\cdot|_j$ $(1\le j\le J)$ are precisely all the extensions of $|\cdot|$ from $\bK$ to $\bL$. Furthermore, $\bL_j$ is the completion of $\bL$ with respect to $|\cdot|_j$.
\end{prop}
\begin{remark}\label{rmk: valuation_extension}
We apply these extension propositions in the following two situations. First, we describe how to extend the valuation $\nu$ on $\Lambda_{\bF_p}$ to a finite field extension $\Lambda_{\bF_p}(\gamma)$ for an algebraic element $\gamma$ over $\Lambda_{\bF_p}$. In this case, we note that $\Lambda_{\bF_p}$ is complete with respect to the norm $|\cdot|=e^{-\nu(\cdot)}$, therefore, Proposition \ref{prop: valuation_extension1} gives us an extension of the norm and, hence, the valuation, to $\Lambda_{\bF_p}(\gamma)$. The same argument applies to the extension of $\nu$ to the algebraic closure $\ol\Lambda_{\bF_p}$. Similarly, we now explain how to extend the valuation $\nu$ on $Q(\Lambda_{\Z})$ to a valuation on a field extension $Q(\Lambda_{\Z})(\alpha)$, where $\alpha$ is algebraic over $Q(\Lambda_{\Z})$. We note that $Q(\Lambda_{\mathbb{Z}})$ has characteristic zero making it a perfect field, which implies that $Q(\Lambda_{\mathbb{Z}})(\alpha)$ is a separable extension. Therefore, by Proposition \ref{prop: valuation_extension2} one obtains a norm $|\cdot|$ and, thus, a valuation $\nu$ on $Q(\Lambda_{\mathbb{Z}})(\alpha)$. 
\end{remark}

\begin{prop}[Chapter 7, Corollary 1 in \cite{cassels1986local}]\label{prop: unique_extension}
There is a unique extension of $|\cdot|$ to the algebraic closure $\overline{\bK}$ of $\bK$
\end{prop}

\begin{remark}\label{rmk: unique_extension}
 Proposition \ref{prop: unique_extension} implies that the extensions of the valuations on $\Lambda_{\bF_p}$ and $\Lambda_{\bF_p}(\gamma)$ to a valuation on $\ol\Lambda_{\bF_p}$ coincide.
\end{remark}

The following proposition implies that for a sufficiently large prime $p$, elements in the field of fractions $Q(\Lambda_{\Z})$ of $\Lambda_{\Z}$ can be reduced modulo $p$. Its proof follows closely that of \cite[Theorem 4.1]{Hofer-Salamon} and \cite[Theorem~4.2]{Hofer-Salamon}.
\begin{prop}
\label{prop:finite_denominators}
         Let $Q(\Lambda_{\mathbb{Z}})$ be the field of fractions of $\Lambda_{\mathbb{Z}}$. If $f\in Q(\Lambda_{\mathbb{Z}})$, then it can be written as $\sum c_jT^{\mu_j}$, where $c_j\in\mathbb{Q}$ and only finitely many primes appear in the denominators of the coefficients of $f$.
\end{prop}
\begin{proof}
Note that if $$h=\displaystyle\sum_{i=-N}^\infty h_iT^{\nu_i}$$ is an element in $\Lambda_{\mathbb{Z}}$, then $T^{-\nu_{-N}}h$ only has nonnegative exponents. Suppose that $f\in Q(\Lambda_{\mathbb{Z}})$. We may assume without loss of generality that 
        \begin{align*}
            f=\displaystyle\frac{a_0+\displaystyle\sum_{i=1}^\infty a_iT^{\lambda_i}}{b_0+\displaystyle\sum_{j=1}^\infty b_jT^{\theta_j}},
        \end{align*}
        where $0<\lambda_1<\lambda_2<\cdots$ and $0<\theta_1<\theta_2<\cdots$. Set
        \begin{align*}
            A:= a_0+\displaystyle\sum_{i=1}^\infty a_iT^{\lambda_i} \quad \text{and}\quad B:= b_0+\displaystyle\sum_{j=1}^\infty b_jT^{\theta_j}
        \end{align*}
        and let $g_0=\displaystyle\frac{a_0}{b_0}$. Then
        \begin{align*}
            A-g_0B=\displaystyle\sum_{i=1}^\infty a_iT^{\lambda_i}-\sum_{j=1}^\infty\frac{a_0b_j}{b_0}T^{\theta_j},
        \end{align*}
        and the leading term has a nonnegative exponent. 

        \textbf{Case 1. $\lambda_1<\theta_1$} The leading term in $A-g_0B$ is $a_1T^{\lambda_1}$. Define $g_1=\displaystyle\frac{a_0}{b_0}+\frac{a_1}{b_0}T^{\lambda_1}$. Then 
        \begin{align*}
            A-g_1B=\displaystyle\sum_{i=2}^\infty a_iT^{\lambda_i}-\sum_{j=1}^\infty\frac{a_0b_j}{b_0}T^{\theta_j}-\sum_{j=1}^\infty\frac{a_1b_j}{b_0}T^{\lambda_1+\theta_j},
        \end{align*}
        which has the exponent of the leading term greater than $\lambda_1$.

        \textbf{Case 2. $\lambda_1=\theta_1$} The leading term in $A-g_0B$ is $\left(a_1-\displaystyle\frac{a_0b_1}{b_0}\right)T^{\lambda_1}$. Define $g_1=\displaystyle\frac{a_0}{b_0}+\left(\frac{a_1}{b_0}-\frac{a_0b_1}{b_0^2}\right)T^{\lambda_1}.$ Then
        \begin{align*}
            A-g_1B=\displaystyle\sum_{i=2}^\infty a_iT^{\lambda_i}-\sum_{j=2}^\infty\frac{a_0b_j}{b_0}T^{\theta_j}-\sum_{j=1}^\infty\left(\frac{a_1}{b_0}-\frac{a_0b_1}{b_0^2}\right)b_jT^{\lambda_1+\theta_j},
        \end{align*}
         which has the exponent of the leading term greater than $\lambda_1$.

         \textbf{Case 3. $\lambda_1>\theta_1$} The leading term in $A-g_0B$ is $-\displaystyle\frac{a_0b_1}{b_0}T^{\theta_1}$. Define $g_1=\displaystyle\frac{a_0}{b_0}-\frac{a_0b_1}{b_0^2}T^{\theta_1}$. Then 
         \begin{align*}
             A-g_1B=\displaystyle\sum_{i=1}^\infty a_iT^{\lambda_i}-\sum_{j=2}^\infty\frac{a_0b_j}{b_0}T^{\theta_j}+\sum_{j=1}^\infty\frac{a_0b_1b_j}{b_0^2}T^{\theta_1+\theta_j},
         \end{align*}
         which has the exponent of the leading term greater than $\theta_1$. 

One can repeat this process to get $g_n$ for $n\in\mathbb{N}$. It is easy to see that the exponent of the leading term of $A-g_nB$ strictly grows as $n$ tends to infinity. Next, we show that, in fact, the sequence of exponents of the leading term diverges. We then conclude that $f=\lim g_n$. 

Let $\Sigma(A-g_nB)$ be the set of finite linear combinations of the exponents of the terms in $A-g_nB$ over nonnegative integers. Then 
         \begin{align*}
             \cdots\subset\Sigma(A-g_{n+1}B)\subset\Sigma(A-g_nB)\subset\cdots\subset\Sigma(A-g_0B).
         \end{align*}
Thus, the exponent of the leading term of $A-g_nB$ strictly grows in $\Sigma(A-g_0B)$. Note that $\Sigma(A-g_0B)$ is a discrete subspace of $\R$, and, hence, the leading exponent of $A-g_nB$ must, in fact, diverge. Indeed, for any $C>0$ there exist positive integers $K$ and $m$ such that $\theta_i\geq C$ and $\lambda_i\geq C$ for all $i\geq K$, and such that $m\theta_i\geq C$ and $m\lambda_i\geq C$ for all $i\geq1$. In particular, if $E$ is the (finite) set of finite linear combinations of the form $\sum p_i\lambda_i + q_i\theta_i$, $i<K$, and $p_i<m$, $q_i<m$ for all $i$, then
\begin{equation*}
	\Sigma(A-g_0B)\cap(-\infty,C)\subset E.
\end{equation*}
Therefore, it is clear that any $\xi\in\Sigma(A-g_0B)\cap(-\infty,C)$ is isolated. Since, $C>0$ was chosen arbitrarily, it follows that $\Sigma(A-g_0B)$ is discrete. 

We now show that  $f = \lim g_n$. Recall $$f-g_n = \frac{A}{B}-g_n =\frac{A-g_nB}{B}.$$ It follows from the preceeding paragraph that the valuation of $A-g_nB$ diverges as $n\ra\infty$. Thus, the norm of $A-g_nB$ tends to $0$ as $n\ra\infty$ (because the norm is defined as $e^{-val}$). Consequently, $A-g_nB$ and, hence, $f-g_n$ converge to zero. Finally, the primes appearing in the denominators of the coefficients of $f$ are the primes dividing $b_0$, of which there are finitely many.

%Denote the exponent of the leading term of $A-g_nB$ by $\xi_n$. Suppose that $\xi_n$ is bounded and, hence, converges to $\xi=\sum p_i\lambda_i+q_i\theta_i$ where $i\geq1$, $p_i$ and $q_i$ are integers, and only finitely many of $p_i, q_i$ are nonzero. Take $\epsilon<\min\{{\lambda_i,\theta_i |p_i,q_i\neq 0}\}$. Then, the neighborhood $$B_{\epsilon}(\xi) = \{x \in\R\,|\,|x-\xi|< \epsilon\}$$ does not contain any finite linear combination of the exponents of the terms in $A-g_nB$ over nonnegative integers. In particular, $B_{\epsilon}(\xi)$ does not contain any $\xi_n$ since it is a finite linear combination of the exponents of the terms in $A-g_nB$ over nonnegative integers. This contradicts the assumption. Hence $\xi_n$ diverges.        
\end{proof}

\begin{remark}\label{rmk: reduction}
Note that reduction of coefficients is not possible for arbitrary finite sets $\{f_1,\cdots, f_k\}\subset\Lambda_{\Q}$. Indeed, one can have elements of the form \[\sum_{l=-K}^{\infty}\frac{1}{l!}T^{l},\] which has infinitely many primes in the denominators. In particular, one cannot reduce the coefficients. The previous proposition, shows that this does not happen in $Q(\Lambda_{\Z})$.
\end{remark}

\subsubsection{Non-Archimedean filtrations}\label{sec: Non-arch_filtration}
Let $\Lambda$ be a field with a non-Archimedean valuation $\nu$. Suppose $C$ is a finite dimensional module over $\Lambda$. 
\begin{definition}
A non-Archimedean filtration is a function $l:C\ra\R\cup\{-\infty\}$ satisfying the following properties:
	\begin{enumerate}
	\item $l(x)=-\infty$ if and only if $x=0$,
	\item $l(\lambda x)=l(x)-\nu(\lambda)$ for all $\lambda\in\Lambda, x\in C$,
	\item $l(x+y)\leq\max\{l(x),l(y)\}$\label{def: max_prop}
	\end{enumerate}
\end{definition}
It is not hard to check that the maximum property (\ref{def: max_prop}) implies that when $l(x)\neq l(y)$, we have that $l(x+y)=\max\{l(x),l(y)\}$, see \cite{entov2003calabi, usher2016persistent}. We call a $\Lambda$-basis $(x_1,\cdots,x_{B})$ of $(C, l)$ \textit{orthogonal} if 
\begin{equation*}
	l\Big(\sum\lambda_i x_i\Big) = \max\{l(x_{i}) - \nu(\lambda_i)\}
\end{equation*}
for all $\lambda_{i}\in\Lambda$. It is called \textit{orthonormal} if it also satisfies $l(x_i)=0$ for all $i$.

In order to define a non-Archimedean filtration $\cA:\cf(H; \Lambda_{\bK})\ra\R\cup\{-\infty\}$,
we choose the $\Lambda_{\bK,\om}$-basis $\{x_1,\cdots,x_N\}$ of $\cf(H;\Lambda_{\bK,\om})$ given by the set of contractible fixed points of $\phi_H$ equipped with a fixed choice of capping and set
\begin{equation}
	\cA\Big(\sum \lambda_{i}x_{i}\Big)=\max\{\cA_{H}(x_{i})-\nu(\lambda_{i})\}.
\end{equation}
Equivalently, we declare $\{x_1,\dots,x_N\}$ to be an orthogonal basis of $(\cf(H;\Lambda_{\bK}), \cA)$. We note that for any non-trivial $x\in\cf(H;\Lambda_{\bK})$ we have $\cA(d_{F}(x))<\cA(x)$. In this case, the filtered complex is said to be \textit{strict}. The basis given by $\{\til{x}_1,\cdots,\til{x}_N\}=\{T^{\cA(x_1)}x_1,\cdots,T^{\cA(x_N)}x_N\}$ is orthonormal.

 We now consider the case where we extend the coefficients of $\cf(H;\Lambda_{\bK})$ to $\ol\Lambda_{\bK}$. For an orthogonal basis $\{x_1,\dots,x_N\}$ of $(\cf(H;\Lambda_{\bK}),\cA)$, $\{x_1\otimes1,\dots, x_N\otimes1\}$ is an orthogonal basis of $\cf(H;\ol\Lambda_{\bK})$. We define a non-Archimedean filtration $\ol\cA$ on $\cf(H;\ol\Lambda_{\bK})$ by setting
\begin{equation*}
	\ol\cA\Big(\sum\ol\lambda_{i}x_{i}\otimes1\Big)=\max\{\cA_{H}(x_{i})-\ol\nu(\ol\lambda_{i})\},
\end{equation*}
where $\ol\nu$ is the non-Archimedean valuation on $\ol\Lambda_{\bK}$ described in Remark \ref{rmk: valuation_extension}. One verifies that $\{T^{\cA(x_1)}x_1\otimes1,\dots, T^{\cA(x_N)}x_N\otimes1\}$ is an orthonormal basis of  $(\cf(H;\ol\Lambda_{\bK}),\ol\cA)$. Any orthonormal basis $\{y_1,\dots,y_N\}$ of $(\cf(H;\Lambda_{\bK}),\cA)$ is related to the orthonormal basis $\{\til{x}_1,\dots, \til{x}_N\}$ by an invertible matrix $A\in \mrm{GL}(N,\lamzero_{\bK})$ in the sense that $A(T^{\cA(x_j)}x_j)=y_j$, furthermore, \[\{y_1\otimes1,\dots,y_N\otimes1\}=A\{x^\prime_1\otimes1,\dots,x^\prime_N\otimes1\}\] is an orthonormal basis of $(\cf(H;\ol\Lambda_{\bK}),\ol\cA)$.

\subsection{Quantum homology}
\subsubsection{Definition of quantum homology}
We follow \cite{mcduff2012j} for the definition of quantum homology. The quantum homology of $M$ is defined by $\qh_*(M)=\qh_*(M,\Lambda_{\Z})=\h_*(M)\otimes\Lambda_{\mathbb{Z}}$. There is a product on \[\qh_{ev}(M)=\bigoplus_{i} \qh_{2i}(M)\] defined as follows. Choose an integer basis $e_0, \dots, e_N$ of the free part of $\h_*(M;\mathbb{Z})$ such that $e_0=[M]\in \h_{2n}(M)$ and each basis element $e_\nu$ has pure degree. Define the integer matrix $g_{\nu\mu}$ by 
\begin{center}
    $g_{\nu\mu}:=\displaystyle\int_M PD(e_\nu)\smile PD(e_\mu),$
\end{center}
and let $g^{\nu\mu}$ denote the inverse matrix. Then the product of $a, b\in \h_{ev}(M)$ is defined by
\begin{center}
  $a*b:=\displaystyle\sum_{A}\sum_{\nu, \mu}GW^M_{A, 3}(a, b, e_\nu)g^{\nu\mu}e_\mu T^{\omega(A)}$.  
\end{center}
The product of $a$ and $b$ can also be expressed as 
\begin{center}
    $a*b=\displaystyle\sum_{A}(a*b)_AT^{\omega(A)}$
\end{center}
where
\begin{center}
    $(a*b)_A:=\displaystyle\sum_{\nu, \mu}GW^M_{A,3}(a, b, e_\nu)g^{\nu\mu}e_\mu\in\h_{deg(a)+deg(b)+2c_1(A)-2n}(M)$
\end{center}
which is characterized by the condition
\begin{center}
    $\displaystyle\int_MPD((a*b)_A)\smile c:=GW^M_{A, 3}(a, b, PD(c))$
\end{center}
for $c\in \h^*(M)$. For a base field, $\bK$ we define $\qh_*(M,\Lambda_{\bK})=\h_*(M)\otimes\Lambda_{\mathbb{\bK}}$, and the product is defined in the same way. Also, if $\bL$ is a field containing $\Lambda_{\Z}$ or if it is a field extension of $\Lambda_{\bK}$, we set $\qh(H,\bL)=\qh(M)\otimes\bL$ and extend the quantum product linearly.

\subsubsection{Semisimplicity of quantum homology}
We note that the quantum product is graded commutative, however, since we are considering only the even degrees, $\qh_{ev}(M,\Lambda_{\bK})$ has the structure of a commutative algebra over $\Lambda_{\bK}$. Hence, it is semiseimple if it splits as an algebra into a direct sum of fields $F_1\oplus\cdots\oplus F_m$. In particular $F_j$ is a finite dimension vector space over $\Lambda_{\bK}$ for each $j$. 

\begin{remark}\label{rmk: semiseimple_diff}
Other notions of semisimplicity have been considered in the non-monotone setting, for instance we can ask that $\qh_{2n}(M,\Lambda_{\omega,\bK})$ is a semisimple algebra over the field $\Lambda_{\omega,\bK;0}$, which is the degree $0$ component of $\Lambda_{\omega, \bK}$. We now show that this condition implies that $\qh_{ev}(M,\Lambda_{\bK})$ is semisimple. Indeed, by \cite[Proposition 2.1(A)]{entov2008symplectic} it follows that if $\qh_{2n}(M,\Lambda_{\omega,\bK})$ is semisimple over $\Lambda_{\omega,\bK;0}$, then $\qh_{2n}(M,\Lambda_{\omega,\bK})\otimes\Lambda_{\bK}$ is semisimple over $\Lambda_{\bK}$. Furthermore, denote by $\Lambda_{\omega,\bK;j}$ the degree $2j$ component of $\Lambda_{\omega,\bK}$ and note that,\[\qh_{2n}(M,\Lambda_{\omega,\bK})=\bigoplus\h_{2i}(H;\bK)\otimes\Lambda_{\omega,\bK;n-i}.\] Therefore, $\qh_{ev}(M,\Lambda_{\bK})=\qh_{2n}(M,\Lambda_{\omega,\bK})\otimes\Lambda_{\bK}$, which by applying \cite[Proposition 2.1(A)]{entov2008symplectic} once more, implies  that it is semisimple.
\end{remark}

\subsubsection{PSS isomorphism}
Piunikhin, Salamon and Schwarz \cite{piunikhin1996symplectic} defined the $PSS$ isomorphism between Hamiltonian Floer homology and quantum homology
\begin{center}
    $\pss_{H}:\qh(M)\to\hf(H)$.
\end{center}
On the chain level and for generic auxiliary data, the map is defined by counting certain isolated configurations consisting of negative gradient trajectories $\gamma:(-\infty,0]\ra M$ of a generic Morse-Smale pair\footnote{A Morse function $f$ and Riemannian metric $g$ on $M,$ satisfying the Morse-Smale condition.} incident at $\gamma(0)$ with the asymptotic $\lim_{s \to -\infty}u(s,\cdot)$ of a map $u:\R\times\rS^1\ra M$ of finite energy, satisfying the Floer equation
\begin{equation*}
\frac{\partial u}{\partial s} + J_t(u)\left(\frac{\partial u}{\partial t} - X_K^t(u)\right)=0,
\end{equation*}
for $(s,t)\in\R\times\rS^1$ and $K(s,t)\in C^\infty(M,\R)$ a small perturbation of $\beta(s)H_t$ such that $K(s,t) = \beta(s) H_t$ for $s\ll-1$ and for $s\gg+1$. Here $\beta:\R\ra[0,1]$ is a smooth function satisfying $\beta(s)=0$ for $s\ll-1$ and $\beta(s)=1$ for $s\gg +1$. This map produces an isomorphism of $\Lambda_{\bK,\om}$-modules, which intertwines the quantum product on $\qh(M)$ with the pair-of-pants product on Hamiltonian Floer homology. It is extended by linearity when we work over the universal Novikov ring $\Lambda_{\bK}$ and its algebraic closure $\ol\Lambda_{\bK}$.

\subsubsection{Filtration on quantum homology}
Consider the non-Archimedean valuation $\nu$ on $\Lambda_{\mathbb{Z}}$ defined in Section \ref{sec: non-Arch_val}. For each element $\sum f_i\alpha_i$ where $f_i\in\Lambda_{\mathbb{Z}}$ and $\alpha_i\in\h_*(M)$ define the filtration $l:\qh(M)\ra\R\cup\{-\infty\}$ to be $l(\sum f_i\alpha_i)=\max\{-\nu(f_i)\}$. Now, as in \cite{polterovich2017persistence, shelukhin2022hofer}, for each $\alpha\in \qh(M)$, we have a map 
\begin{align*}
	\alpha*: \hf(H)^{<a}&\longrightarrow\hf(H)^{<a+l(\alpha)}
\end{align*}
defined by counting negative $g$-gradient trajectories $\ga:(-\infty,0]\rightarrow M$ of a Morse function $f$ on $M$, for a Morse-Smale pair $(f,g)$, asymptotic to critical points of $f$ as $s\rightarrow-\infty$, and with $\gamma(0)$ incident to Floer cylinders $u:\R\times\rS^1\rightarrow M$ at $u(0,0)$. This construction is reminiscent of the quantum cap product as in \cite{piunikhin1996symplectic, schwarz2000action, SeidelMCG, Floer3}.

\subsection{Spectral invariants}
\subsubsection{Definitions and basic properties}
Let $H$ be a non-denegerate Hamiltonian function and consider the filtered complex $(\cf(H;\Lambda), \cA)$, where $\Lambda$ is one of the following $\Lambda_{\bK, \om}, \Lambda_{\bK}, \ol\Lambda_{\bK}$ and $\cA$ is as in Section \ref{sec: Non-arch_filtration}. Denote $\cf(H;\Lambda)^{<c}=\cA^{-1}(-\infty,c)$ and $\hf(H;\Lambda)^{<c}$ the filtered homology groups. The spectral invariant associated to a non-trivial $\alpha\in\qh(M)$ is defined as
\begin{equation*}
	c(\alpha, H) = \inf\{a\in\R\,|\,\pss_{H}(\alpha)\in\im(\hf(H;\Lambda)^{<a}\ra\hf(H;\Lambda))\}.
\end{equation*}
By \cite{biran2009rigidity}, spectral invariants do not change under extension of coefficients, in particular, we do not need to specify the $\Lambda$ in the notation. Spectral invariants enjoy a wealth of useful properties established by Schwarz \cite{schwarz2000action}, Viterbo \cite{viterbo1992symplectic}, Oh \cite{oh2005construction, oh2006lectures, mcduff2012j} and generalized by Usher \cite{usher2008spectral, usher2010duality}. We summarize some of their properties.
\begin{prop}
The spectral invariants satisfy the following:
	\begin{enumerate}
	\item{Stability: } for all $H,G\in\cH$ and $\alpha\in\qh(M)\setminus\{0\}$, \[\int^{1}_{0}\min(H_{t}-G_{t})dt\leq c(\alpha,H)-c(\alpha,G)\leq\int^{1}_{0}\max(H_{t}-G_{t})dt.\]
	\item{Triangle inequality:} for all $H,G\in\cH$ and $\alpha,\beta\in\qh(M)\setminus\{0\}$, \[c(\alpha*\beta,H\#G)\leq c(\alpha,H)+c(\beta,G).\]
	\item{Novikov action:} for all $H\in\cH,\alpha\in\qh(M)\setminus\{0\}$, and $\lambda\in\Lambda$,\[c(\lambda\alpha,H)=c(\alpha,H)-\nu(\lambda).\]
	\item{Non-Archimedean property:} for all $H\in\cH$ and $\alpha,\beta\in\qh(M)\setminus\{0\}$,\[c(\alpha+\beta,H)\leq\max\{c(\alpha,H),c(\beta,H)\}.\]
	\end{enumerate}
\end{prop}
We remark that by the stability property, the spectral invariants are defined for all $H\in\cH$ and all the above listed properties apply in this generality.

\subsubsection{Poincare Duality for Spectral Invariants}
Let $H$ be a non-degenerate Hamiltonian function and fix a choice of capping $\ol{x}_k$ for each $1$-periodic orbit $x_k$. Define a bilinear pairing
\begin{align*}
    \Delta: \cf_*(H, \overline{\Lambda}_{\bF_p})\times \cf_*(\ol{H}, \overline{\Lambda}_{\bF_p})&\longrightarrow\overline{\Lambda}_{\bF_p}\\
    \Big(\sum a_i\ol{x}_i, \sum b_j\ol{x}^\prime_j\Big)&\longmapsto\sum a_ib_i
\end{align*}
where the sums are finite and for all $k$, the capped orbit $\ol{x}^\prime_k$ is equal to $\ol{x}_k$ with reversed orientation, in particular, $x^\prime_k(t)=x_k(1-t)$.
\begin{lem}
The bilinear pairing $\Delta$ is non-degenerate.
\end{lem}
\begin{proof}
Suppose $\Delta(\sum a_i\ol{x}_i, \cdot)=0$. Then, for every $\ol{x}^\prime_i$, \[\Delta\Big(\sum a_i\ol{x}_i, \ol{x}^\prime_i\Big)=0.\] On the other hand, $\Delta(\sum a_i\ol{x}_i, \ol{x}^\prime_i)=a_i$. Thus $a_i=0$ for each $i$, i.e. $\sum a_i\ol{x}_i=0$. Similarly, if $\Delta(\cdot, \sum b_j\ol{x}^\prime_j)=0$, then $\sum b_j\ol{x}^\prime_j=0$.
\end{proof}

\begin{lem}\label{AAA}
For any real number $\alpha$ the composition $\nu\circ\Delta$ is positive on \[\cf_*(H, \overline{\Lambda}_{\bF_p})^{<\alpha}\times \cf_*(\ol{H}, \overline{\Lambda}_{\bF_p})^{<-\alpha}.\]

\end{lem}

\begin{proof}
Suppose that \[\Big(\sum a_i\ol{x}_i, \sum b_i\ol{x}^\prime_i\Big)\in\cf_*(H, \overline{\Lambda}_{\bF_p})^{<\alpha}\times \cf_*(\ol{H}, \overline{\Lambda}_{\bF_p})^{<-\alpha}.\] Then, \[\cA\big(\sum a_i\ol{x}_i\big)=\max\{\mathcal{A}_{H}(\ol{x}_i)-\nu(a_i)\}<\alpha.\] In particular, we have that $\nu(a_i)>\mathcal{A}_{H}(\ol{x}_i)-\alpha$ for all $i$. Following the same logic, $\nu(b_i)>-\mathcal{A}_{H}(\ol{x}_i)+\alpha$ for all $i$. Thus,
\begin{align*}
	\nu\Big(\Delta\Big(\sum a_i\ol{x}_i, \sum b_i\ol{x}^\prime_i\Big)\Big)&=\nu\Big(\sum a_ib_i\Big)\\
	&\ge \min\{\nu(a_ib_i)\}\\
	&=\min\{\nu(a_i)+\nu(b_i)\}\\
	&>\mathcal{A}_{H}(\ol{x}_i)-\alpha-\mathcal{A}_{H}(\ol{x}_i)+\alpha=0
\end{align*}
\end{proof}
\begin{lem}
Let $a\in \cf_*(H, \overline{\Lambda}_{\bF_p})$ and $b\in \cf_*(\ol{H}, \overline{\Lambda}_{\bF_p})$. Then, $\Delta(\partial(a), b)=\pm\Delta(a, \partial(b))$. In particular, there is an induced pairing on homology
\begin{equation*}
    \Delta: \hf_*(H, \overline{\Lambda}_{\bF_p})\times \hf_*(\ol{H}, \overline{\Lambda}_{\bF_p})\longrightarrow\overline{\Lambda}_{\bF_p}
\end{equation*}
\end{lem}

\begin{proof}
Choose a basis  $\{\ol{x}_i\}$ of $\cf_{k+1}(H, \overline{\Lambda}_{\bF_p})$ and a basis  $\{\ol{y}_j\}$ of $\cf_{k}(H, \overline{\Lambda}_{\bF_p})$. Then $\{\ol{x}^\prime_i\}$ is a basis of $\cf_{-k-1}(\ol{H}, \overline{\Lambda}_{\bF_p})$ and $\{\ol{y}^\prime_j\}$ is basis of $\cf_{-k}(\ol{H}, \overline{\Lambda}_{\bF_p})$. Suppose $\partial\ol{x}_i=\sum a_j\ol{y}_j$ and $\partial\ol{y}^\prime_j=\sum b_i\ol{x}^\prime_i$. Then, \[\Delta(\partial\ol{x}_i, \ol{y}^\prime_j)=\Delta\big(\sum a_j\ol{y}_j, \ol{y}^\prime_j\big)=a_j.\] Similarly, $\Delta(\ol{x}_i, \partial\overline{y}_j)=b_i$. By definition, $a_j$ is the number of Floer trajectories connecting $\ol{x}_i$ and $\ol{y}_j$ and $b_i$ is the number of the Floer trajectory connecting $\ol{y}^\prime_j$ and $\ol{x}^\prime_i$. Thus $a_j=b_i$, i.e. \[\Delta(\partial\ol{x}_i, \ol{y}_j)=\Delta(\ol{x}_i, \partial\ol{y}_j).\] Since $\Delta$ is bilinear, we have that $\Delta(\partial a, b)=\Delta(a, \partial b)$ for any $a\in \cf_*(H, \overline{\Lambda}_{\bF_p})$ and $b\in \cf_*(\ol{H}, \overline{\Lambda}_{\bF_p})$.
\end{proof}

 The following proposition is interpreted as Poincaré duality in Floer theory. The equality follows from standard arguments in \cite[Corollary 1.4]{usher2010duality} and \cite[Lemma 2.2]{entov2003calabi}.
\begin{prop}\label{CCC}
Let $a\in \qh_*(M, \overline{\Lambda}_{\bF_p})$ be non-trivial, then, \[c(a, H)=-\inf\{c(b, \ol{H}) \,|\, b\in \qh_*(M, \overline{\Lambda}_{\bF_p}),\,\nu(\Delta(\pss_{H}(a), \pss_{\ol{H}}(b)))\leq 0\}.\]
\end{prop}
\begin{proof}
First we show that 
\begin{align*}
-c(a, H)\ge\inf\{c(b, \ol{H})\mid b\in \qh(M, \overline{\Lambda_{\mathbb{K}}}), \nu(\Delta(\text{PSS}_H(a), \text{PSS}_{\ol{H}}(b)))\le 0\}.
\end{align*}
Suppose that $\alpha< c(a, H)$. We have a short exact sequence of chain complexes

\begin{equation*}
	   0\ra\cf_*(H, \overline{\Lambda}_{\bF_p})^{<\alpha}\ra\cf_*(H, \overline{\Lambda}_{\bF_p})\ra\cf_*(H, \overline{\Lambda}_{\bF_p})^{[\alpha, \infty)}\ra0,
\end{equation*}
inducing an exact sequence
\begin{equation*}
    \hf_*(H, \overline{\Lambda}_{\bF_p})^{<\alpha}\xrightarrow{i_\alpha} \hf_*(H, \overline{\Lambda}_{\bF_p})\xrightarrow{\pi_\alpha}\hf_*(H, \overline{\Lambda}_{\bF_p})^{[\alpha, \infty)}.
\end{equation*}

The fact that $\alpha<c(a, H)$ means that $\text{PSS}_H(a)$ is not represented by any chains of filtration level at most $\alpha$, so that $\text{PSS}_H(a)\notin \text{im}(i_\alpha)$. Thus $\pi_\alpha(\text{PSS}_H(a))\ne 0$. We now prove the following auxilary claim. 

\begin{claim}
    Fix a representative $\tilde{a}$ of $\text{PSS}_H(a)$. There is $b\in \qh(M, \overline{\Lambda}_{\bF_p})$ such that 
\begin{align*}
\nu(\Delta(\text{PSS}_H(a), \text{PSS}_{\ol{H}}(b)))\le 0
\end{align*}
and 
\begin{align*}
\mathcal{A}(\tilde{b})= -\mathcal{A}(\tilde{a})\le -\alpha
\end{align*}
where $\tilde{b}$ is a representative of $\text{PSS}_{\ol{H}}(b)$ respectively.
\end{claim}
\begin{proof}
Consider the dual vector space $\cf(H,\overline{\Lambda}_{\bF_p})^*$ of $\cf(H,\overline{\Lambda}_{\bF_p})$. Let $$\{\xi_1, \cdots, \xi_B, \eta_1, \cdots, \eta_K, \zeta_1,\cdots, \zeta_K\}$$ be a singular value decomposition for the complex $\cf(H,\overline{\Lambda}_{\bF_p})$. We recall from \cite[Proposition 2.20]{usher2016persistent} that there is an $\mathcal{A}^*$-orthogonal dual basis $$\{\xi_1^*, \cdots, \xi_B^*, \eta_1^*, \cdots, \eta_K^*, \zeta_1^*,\cdots, \zeta_K^*\}$$ of $\cf(H,\overline{\Lambda}_{\bF_p})^*$ such that $\mathcal{A}^*(\xi_i^*)=-\mathcal{A}(\xi_i)$, $\mathcal{A}^*(\eta_i^*)=-\mathcal{A}(\eta_i)$, and $\mathcal{A}^*(\zeta_i^*)=-\mathcal{A}(\zeta_i)$, where 
    \begin{align*}
        \mathcal{A}^*(f^*)=\sup\{-\mathcal{A}(\theta)-\nu(f^*(\theta))\mid 0\ne\theta\in\cf(H,\overline{\Lambda}_{\bF_p})\}.
    \end{align*}
If $f=\sum a_i\ol{x}_i$ is an element in $\cf(H,\overline{\Lambda}_{\bF_p})$, then we denote by ${f}'$ the element $\sum a_i\ol{x}'_i$ in $\cf(\ol{H},\overline{\Lambda}_{\bF_p})$ and $f^*$ the dual element of $f$ in $\cf(H,\overline{\Lambda}_{\bF_p})^*$. Then 
    \begin{align*}
        \mathcal{A}({f'}) &=\max\{\mathcal{A}_{\ol{H}}(\ol{x}'_i)-\nu(a_i)\}\\
        &=\max\{-\mathcal{A}_H(\ol{x}_i)-\nu(a_i)\}\\
        &=\max\{-\mathcal{A}_H(\ol{x}_i)-\nu(f^*(\ol{x}_i))\}\\
        &\le \mathcal{A}^*(f^*).
    \end{align*}
Next we check $d{\xi}'_i=0$. Assume
    \begin{align*}
        d{\xi}'_i=\sum a_{ij}{\xi}'_j+\sum b_{ij}{\eta}'_j+\sum c_{ij}{\zeta}'_j
    \end{align*}
    Then $a_{ij}=\Delta(\xi_j, d{\xi}'_i)=\Delta(d\xi_j, {\xi}'_i)=\Delta(0, {\xi}'_i)=0$. Similarly, $b_{ij}=c_{ij}=0$. Thus $d{\xi}'_i=0$.

    Assume $\tilde{a}=\sum a_i\xi_i$. Then $\mathcal{A}(\tilde{a})=\max\{\mathcal{A}(\xi_i)-\nu(a_i)\}=\mathcal{A}(\xi_k)-\nu(a_k)$ for some $k$. Define $\tilde{b}:=a_k^{-1}{\xi}'_k$. Then
    \begin{align*}
        \Delta(\tilde{a}, \tilde{b})=a_ka_k^{-1}=1
    \end{align*}
    We have $\nu(\Delta(\tilde{a}, \tilde{b}))=0$. 
    It also means that $[\tilde{b}]\ne 0$ in $\hf(\ol{H}, \overline{\Lambda}_{\bF_p})$. In addition, $\mathcal{A}(\tilde{b})=\mathcal{A}({\xi}'_k)-\nu(a_k^{-1})\le\mathcal{A}^*(\xi_k^*)+\nu(a_k)=-\mathcal{A}(\xi_k)+\nu(a_k)=-\mathcal{A}(\tilde{a})$. Thus we can take $b:=\text{PSS}_{\ol{H}}^{-1}([\tilde{b}])$
\end{proof}
Let $\alpha=c(a, H)-\epsilon$. Then 
\begin{align*}
    \mathcal{A}(\tilde{b})\le -c(a, H)+\epsilon
\end{align*}
Take $\epsilon\to 0$. Then 
\begin{align*}
    \mathcal{A}(\tilde{b})\le -c(a, H)
\end{align*}
We have
\begin{align*}
    \inf\{c(b, \ol{H})\mid b\in \qh(M, \overline{\Lambda}_{\bF_p}), \nu(\Delta(\text{PSS}_H(a), \text{PSS}_{\ol{H}}(b)))\le 0\}
    &\le \mathcal{A}(\tilde{b})\\
    &\le -c(a, H)
\end{align*}

Next, we show that 
\begin{align*}
-c(a, H)\le \inf\{c(b, \ol{H})\mid b\in \qh(M, \overline{\Lambda}_{\bF_p}), \nu(\Delta(\text{PSS}_H(a), \text{PSS}_{\ol{H}}(b)))\le 0 \}. 
\end{align*}
Suppose that $\alpha>c(a, H)$. Thus, there must be some cycle $c\in \cf(\phi, \overline{\Lambda}_{\bF_p})^{<\alpha}$ representing the class $\text{PSS}_H(a)$. If $b\in \qh(M, \overline{\Lambda}_{\bF_p})$ is an arbitrary class satisfying $\nu(\Delta(\text{PSS}_H(a), \text{PSS}_{\ol{H}}(b)))\le 0$, then by the definition of $\Delta$ it must hold that every representative $d\in \cf(\ol{H}, \overline{\Lambda}_{\bF_p})$ of the class $\text{PSS}_{\ol{H}}(b)$ satisfies $\nu(\Delta(c, d))\le 0$. By Lemma \ref{AAA}, this can only be true if no representative $d$ of $b$ belongs to $\cf(\ol{H}, \overline{\Lambda}_{\bF_p})^{-\alpha}$, which amounts to the statement that $c(b, \ol{H})\ge-\alpha$. Note that $b$ was an arbitrary class with $\nu(\Delta(\text{PSS}_H(a), \text{PSS}_{\ol{H}}(b)))\le 0$, while $\alpha$ was an arbitrary number exceeding $c(a, H)$, therefore, we obtain that 
\begin{align*}
-c(a, H)\le \inf\{c(b, \ol{H})\mid b\in \qh(M, \overline{\Lambda}_{\bF_p}), \nu(\Delta(\text{PSS}_H(a), \text{PSS}_{\ol{H}}(b)))\le 0 \}. 
\end{align*}
\end{proof}

On the other hand, we can define a pairing on the even quantum homology:
\begin{align*}
	\tilde{\Delta}:\qh_{ev}(M,\ol{\Lambda}_{\bF_p})\times\qh_{ev}(M,\ol{\Lambda}_{\bF_p})&\ra\ol{\Lambda}_{\bF_p}\\
	\Bigg(\sum_{i=0}^{2n}a_ih_{2i} , \sum_{j=0}^{2n}b_jh_{2j}\Bigg)&\mapsto\sum_{i+j=n}a_ib_j(h_{2i}\circ h_{2j})
\end{align*}
where $a_i,b_i\in\ol{\Lambda}_{\bF_p}$ and $h_{2i}\in\h_{2i}(M,\bF_p)$ for $i=0,\dots,n$. The following result is from \cite[Section 2.6.8]{entov2003calabi} and \cite[Proposition 7.7]{usher2023abstractinterlevelpersistencemorsenovikov}.
\begin{lem}
\label{lem:Pair_QH}
Let $a,b\in\qh_{ev}(M,\ol{\Lambda}_{\bF_p})$. Then, 
\begin{equation}
\label{eq:PD_QH}
\Delta(\pss_{H}(a),\pss_{\ol{H}}(b))=\tilde{\Delta}(a,b).
\end{equation}
\end{lem}

As in \cite[Section 2.3]{entov2003calabi}, we have the following lemma:

\begin{lem}
\label{Pair_QH2}
The pairing satisfies $\tilde{\Delta}(a,b)=\tilde{\Delta}(a*b,[M])$. In particular, $$\Delta(\pss_H(a),\pss_{\ol{H}}(b))=\Delta(\pss_H(a*b),\pss_{\ol{H}}([M])).$$
\end{lem}

\newpage
\subsection{Floer persistence}\label{sec: persistence}
\subsubsection{Overview of persistence modules}
The theory of persistence modules has origins in the field of topological data analysis. It was introduced by Carlsson and Zomordian \cite{zomorodian2004computing} as an algebraic tool whose purpose was to deal with persistence homology invented by Edelsbrunner, Letscher and Zomordian \cite{edelsbrunner2000topological} to study topological aspects of large data sets.  Persistence modules have since then proven useful in many disciplines of pure mathematics such as metric geometry and calculus of variations. Polterovich and Shelukhin \cite{polterovich2016autonomous} (see also \cite{polterovich2017persistence, usher2016persistent}) were the first to view filtered Floer homology as a persistence module in order to prove interesting results about autonomous Hamiltonian diffeomorphisms. Since then, this viewpoint has led to several applications such as Shelukhin's proof \cite{shelukhin2022hofer} of the Hofer-Zehnder conjecture in the monotone setting (under the semisimplicity assumption) and to obstructions to the existence of non-trivial finite subgroups of $\Ham(M,\om)$ \cite{atallah2020hamiltonian}. However, in order to view filtered Floer homology as a persistence module a certain finiteness assumption is required. In practice this means that going beyond the monotone setting requires some work. In \cite{usher2016persistent}, Usher and Zhang generalized the notion of a \textit{barcode} (the main invariant obtained from a persistence module) to the semipositive setting.

\subsubsection{Persistence modules} 
\label{sec:persistence}
In this section we follow \cite[Section 4.4.1]{shelukhin2022hofer} in order define persistence modules and their associated barcodes and discuss the relation between them.

Let $\bK$ be a field. Denote by $\mrm{Vect}_{\bK}$ the category of finite dimensional $\bK$-vector spaces and by $(\R,\leq)$ the poset category of $\R$. A \textit{persistence module} over $\bK$ is a functor \[V:(\R,\leq)\ra\mrm{Vect}_{\bK}.\] The collection of such functors together with their natural transformations form an abelian category $\mrm{Fun}((\R,\leq),\mrm{Vect}_{\bK})$. We consider a full abelian subcategory 
\begin{equation*}
	{\mathbf{pmod}}\subset\mrm{Fun}((\R,\leq),\mrm{Vect}_{\bK}),
\end{equation*}
which is defined by requiring that certain technical assumptions are satisfied. The following definition summarizes the data of such a persistence module.
\begin{definition}
A persistence module $V$ in $\mathbf{pmod}$ consists of a family \[\{V^{a}\in\mrm{Vect}_{\bK}\}_{a\in\R}\] of vector spaces and $\bK$-linear maps $\pi_{V}^{a,b}:V^a\ra V^b$ for each $a\leq b$ such that $\pi_{V}^{a,a}=\id_{V^a}$, and $\pi_{V}^{b,c}\circ\pi_{V}^{a,b}=\pi_{V}^{a,c}$ for all $a\leq b\leq c$. Furthermore, we require them to satisfy the following:
	\begin{enumerate}
	\item{\textit{Support:}} $V^a=0$ for all $a\ll0$.
	\item{\textit{Finiteness:}} there exists a finite subset $S\subset\R$ such that for all $a,b$ in the same connected components of $\R\setminus S$, the map $\pi_{V}^{a,b}$ is an isomorphism. 
	\item{\textit{Continuity:}} for every two consecutive elements $s<s'$ of $S$, and any $a\in (s,s')$, the map $\pi^{a,s'}_{V}$ is an isomorphism.
	\end{enumerate}
We define $V^\infty=\lim_{a\ra\infty}V^a$.
\end{definition}
The normal form theorem \cite{zomorodian2004computing, crawley2015decomposition} implies that the isomorphism classes of a persistence module $V\in\mathbf{pmod}$ is determined by its \textit{barcode}, that is, a multiset $\cB(V)=\{(I_{k}, m_{k})\}_{1\leq k\leq N}$ of intervals $I_k\subset\R$ with multiplicities $m_{k}\in\Z_{>0}$. The intervals are of two types, $K=K(V)$ of them are finite, $I_{k}=(a_{k},b_{k}]$, and $B=B(V)=N-K$ are infinite, $I_k=(a_k,\infty)$. The intervals are called \textit{bars} and the \textit{bar-lengths} are defined as $|I_k|=b_k-a_k$ in the finite case, and $|I_k|=+\infty$ otherwise.

The isometry theorem \cite{chazal2016structure, bauer2015induced, chazal2009proximity, cohen2005stability} shows that the barcode assignment map
\begin{align*}
	\cB:(\mathbf{pmod},d_{inter})&\ra (\mathbf{barcodes}, d_{bottle})\\
		V&\mapsto\cB(V)
\end{align*}
is an isometry. The \textit{interleaving distance} is defined on $\mrm{Fun}((\R,\leq),\mrm{Vect}_{\bK})$ by setting
\begin{equation*}
	d_{inter}(V,W)=\inf\{\delta\geq0\,|\,\exists\,\delta\text{-interleaving},\, f\in\mrm{hom}(V,W[\delta]), g\in\mrm{hom}(W,V[\delta])\},
\end{equation*}
where for $V\in\mrm{Fun}((\R,\leq),\mrm{Vect}_{\bK})$ and $c\in\R$, $V[c]\in\mrm{Fun}((\R,\leq),\mrm{Vect}_{\bK})$ is given by pre-composing with the functor $T_{c}:(\R,\leq)\ra(\R,\leq)$, $t\mapsto t+c$. We say that a pair $f\in\mrm{hom}(V,W[c]), g\in\mrm{hom}(W,V[c])$ is a $c$-interleaving if \[g[c]\circ f=\mrm{sh}_{2\delta,V},\quad f[c]\circ g=\mrm{sh}_{2\delta, W},\] where for $c\geq0$, $\mrm{sh}_{c,V}\in\mrm{hom}(V,V[c])$ is the natural transformation $\id_{(\R,\leq)}\ra T_{c}$. Note that, $d_{inter}(V,W)\in\R_{\geq0}\cup\{\infty\}$, and it is finite if and only if $V^{\infty}\cong W^{\infty}$.

The \textit{bottleneck distance} is defined as
\begin{equation*}
	d_{bottle}(\cB,\cC)=\inf\{\delta>0\,|\,\exists\,\delta\text{-matching between }\cB,\cC\},
\end{equation*}
where a $\delta$-\textit{matching} between $\cB,\cC$ is defined as bijection $\sigma:\cB^{2\delta}\ra\cC^{2\delta}$ between the sub-multisets $\cB^{2\delta}\subset\cB$, $\cC^{2\delta}\subset\cC$, which contain the bars of $\cB,\cC$, respectively, with bar-length greater than $2\delta$, such that if $\sigma((a,b])=(c,d]$, then $\max\{|a-c|,|b-d|\}\leq\delta$. We have that $d_{bottle}(\cB,\cC)\in\R_{\geq0}\cup\{\infty\}$, with it being finite if and only if $B(\cB)=B(\cC)$.

Note that for each $c\in\R$ there is an isometry given by sending a barcode $\cB=\{(I_k,m_k)\}$ to $\cB[c]=\{(I_k-c, m_k)\}$. We can therefore consider the quotient space $(\mathbf{barcodes'},{d}_{bottle}')$ by this isometric $\R$-action, where
\begin{equation*}
	d_{bottle}'([\cB],[\cC])=\inf_{c\in\R}d_{bottle}(\cB,\cC[c]).
\end{equation*}
We observe that bar-lengths are well-defined in the quotient. 

\subsubsection{Boundary depth of persistence modules with no finiteness conditions}
In this section we show that to $V\in\mrm{Fun}((\R,\leq),\mrm{Vect}_{\bK})$ one associates a real value $\beta(V)$, called \textit{boundary depth}, which is controlled by the interleaving distance in a precise sense. 

\begin{definition}
\label{def:boundary_depth}
Let $V\in\mrm{Fun}((\R,\leq),\mrm{Vect}_{\bK})$. The \textit{boundary depth} $\beta({V})$ of ${V}$ is defined as the infimum of all $\lambda\in(0, \infty)$ with the property that, for all $s\in\mathbb{R}$,
\begin{align*}
    \ker\left(V^s\to\lim_{\longrightarrow}V^t\right)=\ker\left(V^s\to V^{s+\lambda}\right)
\end{align*}
Note the set on the left hand side is the same as the ascending union $$\displaystyle\bigcup_{t\in[s, \infty)}\ker(V^s\to V^t).$$ Depending on $V$, there may be no $\lambda$ with this property, in which case $\beta({V})=\infty$. 
\end{definition}
A straightforward verification yields the following:
\begin{prop}
\label{prop:sum_boundary}
Suppose $V_1,\dots,V_m\in\mrm{Fun}((\R,\leq),\mrm{Vect}_{\bK})$ then $$\beta(\bigoplus_{j=1}^{m} V_j)=\max_{1\leq j\leq m}\beta(V_j).$$
\end{prop}

Furthermore, we have the following stability result:
\begin{prop}
\label{prop:boundary-depth_stability}
Suppose $V,W\in\mrm{Fun}((\R,\leq),\mrm{Vect}_{\bK})$ satisfy $d_{inter}(V,W)\leq\delta$ and $\beta(V)<\infty$, then $|\beta(V)-\beta(W)|\leq2\delta.$
\end{prop}
\begin{proof}
Suppose the interleaving is realized by $f:V\ra W[\delta]$ and $g:W\ra V[\delta]$. Let $\beta>\beta(V)$ be arbitrary and take $$x\in\ker\left(W^s\to\displaystyle\lim_{\longrightarrow}W^t\right).$$ Thus, there exists $\lambda$ such that $x\in\ker\left(W^s\to W^{s+\lambda}\right)$. Note that since $g$ is a morphism, we have that $$g(x)\in\ker\left(V^{s+\delta}\to V^{s+\delta+\lambda}\right)\subset\ker\left(V^{s+\delta}\to V^{s+\delta+\beta}\right).$$ The last inclusion follows from the fact that $\beta>\beta(V)$. Furthermore, since $f$ is a morphism and $f\circ g=\mrm{sh}_{2\delta, W}$, we have that $$\mrm{sh}_{2\delta, W}(x)=(f\circ g)(x)\in\ker\left(W^{s+2\delta}\to W^{s+2\delta+\beta}\right),$$ and, hence, $x\in\ker\left(W^{s}\to W^{s+2\delta+\beta}\right)$. Thus, $\beta(W)\leq \beta+2\delta$. Since $\beta$ was arbitraty we have that $\beta(W)\leq \beta(V)+2\delta$, in particular, $\beta(W)<\infty$. We conclude the proof by observing that once we know $\beta(W)$ is finite we can run the same argument to obtain $\beta(V)\leq \beta(W)+2\delta$.
\end{proof}

\subsubsection{Bar-length spectrum of Hamiltonian diffeomorphisms}
\label{sec:SVD}
In the symplectically aspherical and monotone settings, filtered Hamiltonian Floer homology of a non-degenerate Hamiltonian function $H$, together with the maps \[\hf_k(H)^{<a}\ra\hf_k(H)^{<b},\] for $a<b$, induced by the natural inclusions, were studied in \cite{polterovich2016autonomous, polterovich2017persistence, shelukhin2022hofer} from the viewpoint of persistence modules. In particular, one can associate a barcode to each Hamiltonian diffeomorphism with finitely many fixed points, and, hence, a bar-length spectrum. We refer to \cite{polterovich2016autonomous, polterovich2017persistence, usher2016persistent} for details of the construction and for first properties. Following the discussion in \cite[Section 4.4]{shelukhin2022hofer}, we next describe two equivalent descriptions of the bar-length spectrum in the semipositive setting; these are the relevant notions for what is done in this article. We remark that while the notion of a barcode described in Section \ref{sec:persistence} does not extend to the semipositive setting, all three descriptions of bar-length spectrum coincide in the monotone setting \cite[Lemma 9]{shelukhin2022hofer}.

When $(M,\om)$ is semipositive, consider the filtered Floer chain complex $(\cf(H;\Lambda), d, \cA)$, where $\Lambda$ is one of the following $\Lambda_{\bK}, \ol\Lambda_{\bK}$, the non-Archimidean filtration $\cA$ is as in Section \ref{sec: Non-arch_filtration}, and $d$ is the Floer differential. Then, by \cite{usher2016persistent}, when $H$ is non-degenerate, the complex $(\cf(H;\Lambda), d)$ admits an orthogonal basis
\begin{equation*}
	E=(\xi_1,\dots,\xi_B,\eta_1,\dots,\eta_K,\zeta_1,\dots,\zeta_K)
\end{equation*}
such that $d\xi_j=0$ for all $j\in\{1,\dots,B\}$, and $d\zeta_j=\eta_j$ for all $j\in\{1,\dots,K\}$. The lengths of the finite bars are given by \[\beta_j=\beta_j(\phi_H,\Lambda)=\cA(\zeta_j)-\cA(\eta_j),\] which we can assume to satisfy $\beta_1\leq\dots\leq\beta_K$. We call this ordered set of positive real numbers by the \textit{bar-length spectrum} of $\phi_H$. The length of the largest finite bar is the \textit{boundary-depth} introduced by Usher \cite{usher2011boundary, usher2013hofer}, and denoted by $\beta(\phi_H,\Lambda)$. There are $B$ infinite bar-lengths corresponding to each $\xi_j$. This description yields the identity $N=B+2K$, where $N, B$, and $K$ can be computed by $N=\dim_{\Lambda}\cf(H,\Lambda), B=\dim_{\Lambda}\hf(H,\Lambda)$, and $K=\dim_{\Lambda}\im(d)$. We denote by \[\beta_{tot}(\phi,\Lambda)=\beta_{1}(\phi,\Lambda)+\cdots+\beta_{K(\phi,\Lambda)}(\phi,\Lambda)\] the \textit{total bar-length} associated to the barcode. As a side comment which will be useful later, we note that for $V\in\mrm{Fun}((\R,\leq),\mrm{Vect}_{\Lambda})$, given by $V^s=\hf(H,\Lambda)^{<s}$, we have that $\beta(\phi_H,\Lambda)=\beta(V)$ as in Definition \ref{def:boundary_depth}.
 
Following \cite{fukaya2013displacement}, the Floer differential $d$ in the orthonormal basis described in Section \ref{sec: Non-arch_filtration} has coefficients in $\Lambda^0$. Therefore, one defines a Floer complex $(\cf(H,\Lambda^0),d)$ whose homology is a finitely generated $\Lambda^0$-module, and is therefore of the form $\cF\oplus\mathcal{T}$, where $\cF$ is a free $\Lambda^0$-module and $\mathcal{T}$ is a torsion $\Lambda^0$-module. The bar-lengths are given by the isomorphism
\begin{equation*}
	\mathcal{T}\cong\bigoplus_{1\leq j\leq K}\Lambda^{0}/(T^{\beta_j}).
\end{equation*}
We emphasize that $\beta_1\leq\cdots\leq\beta_K$ coincides with the bar-length spectrum of $\phi_H$. 

Combining the ideas in the proof of \cite[Lemma 16]{shelukhin2022hofer} and Proposition \ref{sec: reduction} we show that the bar-length spectrum over $\bF_p$ coincides with that over $\Q$ for a sufficiently large prime $p$.

\begin{remark}
We note that the bar-length spectrum is Hofer continuous since the quasiequivalences coming from Floer continuation maps are controlled by the Hofer norm; see \cite[Section 8]{usher2016persistent} for details. Therefore, the bar-length spectrum and, in particular, the boundary-depth can be extended to all Hamiltonian diffeomorphisms. Consequently, to bound the boundary-depth from above for all Hamiltonian diffeomorphisms, as in Theorem \ref{thm:upper_bound}, it is enough to do so for non-degenerate ones. 
\end{remark}

\begin{lem}
\label{lem: barcode_coeff}
Let $\phi$ be a non-degenerate Hamiltonian diffeormorphism. Then, for a sufficiently large prime $p$, the bar-length spectrum \[0<\beta_{1}(\phi,\Lambda_{\bF_p})\leq\cdots\leq\beta_{K(\phi,\Lambda_{\bF_p})}(\phi,\Lambda_{\bF_p})\] over $\bF_p$ coincides with the bar-length spectrum \[0<\beta_{1}(\phi,\Lambda_{\Q})\leq\cdots\leq\beta_{K(\phi,\Lambda_{\Q})}(\phi,\Lambda_{\Q})\] over $\Q$. In particular $\beta(\phi,\Lambda_{\bF_p})=\beta(\phi,\Lambda_{\Q})$ for sufficiently large $p$. 
\end{lem}
\begin{proof}
Let $\{\xi_1,\dots,\xi_B,\eta_1,\dots,\eta_K,\zeta_1,\dots,\zeta_K\}$ be an orthonormal singular value decomposition of $\cf(H,Q(\Lambda_{\Z}))$ satisfying $d\xi_i=0$ for all $i\in\{0,\dots,B\}$ and $d\zeta_j=T^{\beta_j}\eta_j$ for all $j\in\{1,\dots,K\}$, where $\beta_j$ is the $j$-th bar-length in the spectrum. Note that there is a canonical orthonormal basis $\{\ol{x}_1,\dots,\ol{x}_N\}$ where, $\ol{x}_i=T^\cA(x_i)$ for all $i$, and $\{x_i\}_{i=1}^{N}$ is the collection of contractible fixed points of $\phi$. Recall from the discussion in Section \ref{sec: Non-arch_filtration} that these two orthonormal basis will be related by an matrix $Q\in\mrm{GL}(N,Q(\Lambda_{\Z}))$ whose coefficients have nonnegative valuation, in particular, its filtration-preserving. Since $Q$ has only finitely many coefficients, Proposition \ref{sec: reduction} implies that for a sufficiently large prime $p$, it is possible to reduce $Q$ to a matrix $[Q]_p\in\mrm{GL}(N,\Lambda_{\bF_p})$, i.e. $[\det Q]_p\neq0$. We can then obtain a singular value decomposition $\{[\xi_1]_p,\dots,[\xi_B]_p,[\eta_1]_p,\dots,[\eta_K]_p,[\zeta_1]_p,\dots,[\zeta_K]_p\}$ of $\cf(H,\Lambda_{\bF_p})$, satisfying the same relations as before, by applying $[Q]_p$ to the canonical orthonormal basis given by the contractible fixed points of $\phi$. In particular, it will have the same bar-length spectrum. On the other hand, $\{\xi_1\otimes1,\dots,\xi_B\otimes1,\eta_1\otimes1,\dots,\eta_K\otimes1,\zeta_1\otimes1,\dots,\zeta_K\otimes1\}$ is an orthogonal singular value decomposition of $\cf(H,\Lambda_{\Q})$ with the same bar-length spectrum.
\end{proof}

\subsection{Local Floer homology}\label{sec: LFH}
The definition of local Floer homology and its properties can be found in \cite{ginzburg2010local}. We include them in this section for the convenience of the readers.

Let $x$ be an isolated fixed point of a Hamiltonian diffeomorphism $\phi$ and $\{\phi_{t}\}$ be a Hamiltonian isotopy with $\phi_{1}=\phi$. Then $x(t)=\phi_{t}(x)$ is an $1$-periodic orbit. Denote by $\tilde{x}: S^1\to S^1\times M$ the graph of $x$. Take $\tilde{U}$ to be a small enough neighborhood of $\tilde{x}$ and put $U=p_M(\tilde{U})$, where $p_M: S^1\times M\to M$ is the projection. When $x$ is a degenerate fixed point, we can take a sufficiently small non-degenerate perturbation $\phi'$ of $\phi$ with support in $U$ such that the Floer trajectories with sufficiently small energy connecting the $1$-periodic orbits of $\phi'$ in $U$ are contained in $U$. Thus, every broken trajectory is also contained in $U$. Denote by $\cf(\phi', x)$ the $\bK$-vector space generated by the $1$-periodic orbits of $\phi'$ in $U$. Then, we can define the Floer homology in $U$, which is independent of the choice of the perturbation and of the almost complex structure. We call this Floer homology in $U$ the \textit{local Floer homology at $x$} and denote it by $\hf^{loc}(\phi, x)$. By the definition, one can easily see that $\hf^{loc}(\phi, x)\cong\mathbb{K}$ whenever $x$ is a non-degenerate fixed point. Furthermore, if $(x, \ol{x}_1)$ and $(x, \ol{x}_2)$ are two distinct cappings of a $1$-periodic orbit $x$. Then, $\cz((x, \ol{x}_1))=\cz((x, \ol{x}_2))\mod 2$. Thus, there is a $\mathbb{Z}/2$-grading on $\hf^{loc}(\phi, x)$. When $x$ is non-degenerate, $\hf^{loc}(\phi, x)\otimes_{\mathbb{K}}\Lambda_{\mathbb{K}}$ contributes a copy of $\Lambda_{\mathbb{K}}$ to the Floer complex $\cf_k(\phi, \Lambda_{\mathbb{K}})$. Following the arguments in \cite{ginzburg2019conley, ginzburg2014hyperbolic, mclean2012local, shelukhin2021mathbb, sugimoto2021hofer, shelukhin_wilkins}, for any two distinct capped periodic orbits $\ol{x}$ and $\ol{y}$ of $\phi$, there exists a crossing energy $2\epsilon_0>0$ such that all Floer trajectories, or product structures with $\ol{x}$ and $\ol{y}$ among their asymptotes carry energy at least $2\epsilon_0$.

\begin{definition}\label{def: admissible}
An iteration $\phi^k$ of $\phi$ is \textit{admissible at a fixed point} $x$ \textit{of} $\phi$ if $\lambda^k\ne 1$ for all eigenvalues $\lambda\ne 1$ of $d\phi_x$.
\end{definition}
For example, when none of $\lambda\ne 1$ are roots of unity, $\phi^k$ is admissible for $k>0$. Otherwise, $\phi^{p^n}$ is admissible for sufficiently large $p$ and $n>0$. By \cite[Theorem 1.1, Remark 1.2 ]{ginzburg2010local}, we have the following theorem.
\begin{thm}\label{A}
Let $\phi^k$ be an admissible iteration of $\phi$ at an isolated fixed point $x$. Then, the $k$-iteration $x^{(k)}$ of $x$ is an isolated fixed point of $\phi^k$ and \[\hf^{loc}(\phi^k, x^k)\cong \hf^{loc}(\phi, x).\]
\end{thm}
\begin{remark}
Let $\phi^{k_1}$ and $\phi^{k_2}$ be admissible iterations of $\phi$ at an  isolated fixed point $x$ of $\phi$, then $\hf^{loc}(\phi^{k_1}, x^{k_1})\cong \hf^{loc}(\phi^{k_2}, x^{k_2})$ by Theorem \ref{A}.
\end{remark}
The following theorem is a summary of the canonical $\Lambda^0$-complex constructed in \cite[Section 4.4.7]{shelukhin2022hofer} and upgraded to our setting in \cite{sugimoto2021hofer} (see also \cite{shelukhin_wilkins}), we only change the coefficients to fit our setting. Let $\phi$ be a Hamiltonian diffeomorphism. For each isolated $1$-periodic orbit $x$ of $\phi$, there is an isolating neighborhood $U_x$ of $x$. Let $\phi_1$ be a sufficiently small non-degenerate perturbation of $\phi$. Because there are finitely many isolated $1$-periodic orbits, one can choose $\phi_1=\phi$ outside of $\bigcup_{x\in\fix(\phi)}U_x$. 
\begin{thm}
\label{thm: perturbation}
Suppose $\phi$ has finitely many fixed points and let $\phi_1$ be a sufficiently small non-degenerate perturbation as above. Then, there is a homotopically canonical $\Lambda_{\mathbb{K}}^0$-complex denoted by $\cf(\phi, \Lambda_{\mathbb{K}}^0)$ that satisfies the following properties:
\begin{enumerate}
    \item As a $\Lambda_{\mathbb{K}}^0$-module, \[\cf(\phi, \Lambda_{\mathbb{K}}^0)\cong\displaystyle\bigoplus_{x\in\fix(\phi)}\hf^{loc}(\phi, x)\otimes_\mathbb{K}\Lambda_{\mathbb{K}}^0.\]
    \item Its differential $d_{\phi;\phi_1}$ is defined over $\Lambda_{\mathbb{K}}^0$ and is strict.
    \item The homology of \[\cf(\phi, \Lambda_{\mathbb{K}})=\cf(\phi, \Lambda_{\mathbb{K}}^0)\otimes_{\Lambda_{\mathbb{K}}^0}\Lambda_{\mathbb{K}}\] is isomorphic to $\hf(\phi_1, \Lambda_{\mathbb{K}})$.
    \item The bar-length spectrum associated to $\cf(\phi, \Lambda_{\mathbb{K}})$, denoted by \[\beta^\prime_1(\phi, \Lambda_{\mathbb{K}})\le\dots\le\beta^\prime_{K(\phi, \Lambda_{\mathbb{K}})}(\phi, \Lambda_{\mathbb{K}}),\] satisfies $\beta^\prime_1(\phi, \Lambda_{\mathbb{K}})>\epsilon_0$ and is $2\delta_0$-close to the part \[\beta_{K^\prime+1}(\phi_1, \Lambda_{\mathbb{K}})\le\dots\le\beta_{K^\prime+K(\phi, \Lambda_{\mathbb{K}})}(\phi_1, \Lambda_{\mathbb{K}})\] of the bar-length spectrum of $\phi_1$ above $\epsilon_0$ where $\beta_{K^\prime}(\phi_1, \Lambda_{\mathbb{K}})<2\delta_0\ll\epsilon_0$ and $\delta_0\ll\epsilon_0$ is a small parameter converging to $0$ as $\phi_1$ converges to $\phi$ in the $C^2$-topology.
    \item The bar-lengths $\beta^\prime_j(\phi, \Lambda_{\mathbb{K}})$ for $1\le j\le K(\phi, \Lambda_{\mathbb{K}})$ have a limit $\beta_j(\phi, \Lambda_{\mathbb{K}})$ as the Hamiltonian perturbation tends to zero in the $C^2$-topology.  
\end{enumerate}
\end{thm}

\begin{remark}
\label{rmk:barcode_deg}
We emphasize that $(\cf(\phi, \Lambda_{\mathbb{K}}),d_{\phi;\phi_1})$ is a Floer-type complex, in the sense of \cite{usher2016persistent}, and hence admits an orthogonal basis as described above. This follows from a crossing energy estimate in the semipositive setting as in \cite[Lemma 5]{sugimoto2021hofer} and \cite{shelukhin_wilkins}. That is, the energy a Floer trajectory of a sufficiently $C^\infty$-close perturbation of $H$, which is not counted by a local differential (as in Section \ref{sec: LFH}), is bounded away from zero. In Theorem \ref{thm: perturbation}, \[\beta^\prime_1(\phi, \Lambda_{\mathbb{K}})\le\dots\le\beta^\prime_{K(\phi, \Lambda_{\mathbb{K}})}(\phi, \Lambda_{\mathbb{K}}),\] denotes the bar-legnth spectrum of this Floer-type complex, which depends on the choice of perturbation $\phi_1$ up to an error of $2\delta_0$, which converges to zero as the Hamiltonian perturbation term tends to zero. Finally, we denote by $\beta_j(\phi, \Lambda_{\mathbb{K}})$, the limit of $\beta^\prime_j(\phi, \Lambda_{\mathbb{K}})$ as the Hamiltonian perturbation term tends to zero. 
\end{remark}

\section{Semisimplicity of Quantum homology with different coefficients}\label{sec: semi}
\subsection{Reduction modulo $p$}\label{sec: reduction}
The aim of this section is to prove the following result:
\begin{thm}
\label{thm:QH_ss_alg}
 Let $\mathbb{K}$ be a field of characteristic zero.  If $\qh_{ev}(M, \Lambda_{\mathbb{K}})$ is semisimple, then $\qh_{ev}(M, \ol{Q(\Lambda_{\mathbb{Z}})})$ is semisimple. Moreover,  for all sufficiently large prime $p$, $\qh_{ev}(M, \bar{\Lambda}_{\mathbb{F}_p})$ is semisimple and is generated by a collection of idempotents $\{[\bar{e}_i]_p\}$, which are the reduction modulo $p$ of the idempotents generating $\qh_{ev}(M, \ol{Q(\Lambda_{\mathbb{Z}})})$. 
% 
% If $\qh_{ev}(M, \Lambda_{\mathbb{K}})$ is semisimple, then, for all sufficiently large prime $p$, $\qh_{ev}(M, \bar{\Lambda}_{\mathbb{F}_p})$ is semisimple and is generated by a collection of idempotents $\{[\bar{e}_i]_p\}$, which are the reduction modulo $p$ of the idempotents generating $\qh_{ev}(M, \ol{Q(\Lambda_{\mathbb{Z}})})$. 
 %where $\bar{\Lambda}_{\mathbb{F}_p}$ is the algebraic closure of $\Lambda_{\mathbb{F}_p}$.
\end{thm}
%\subsection{Semisimplicity and idempotents}\label{sec: idempotents}

The first part of the statement is a consquence of the following Proposition, which is a particular case of \cite[Proposition 2.1(A)]{entov2008symplectic}.

\begin{prop}\label{O}
Let $\mathbb{K}$ be a field of characteristic zero. If $\qh_{ev}(M, \Lambda_{\mathbb{K}})$ is semisimple, then $\qh_{ev}(M, Q(\Lambda_{\mathbb{Z}}))$ and $\qh_{ev}(M, \ol{Q(\Lambda_{\mathbb{Z}})})$ are semisimple.
\end{prop}

Under the assumption that $\qh_{ev}(M, \overline{Q(\Lambda_{\mathbb{Z}})})$ is semisimple, let $\{\overline{e}_i\}_{i=1}^m$ be a collection of idempotents such that \[\qh_{ev}(M, \overline{Q(\Lambda_{\mathbb{Z}})})=\displaystyle\bigoplus_{i=1}^{m}\overline{e}_i*\qh_{ev}(M, \overline{Q(\Lambda_{\mathbb{Z}})}).\] We now detail the process, for a suffeciently large prime $p$, of reducing the idempotents $\ol{e}_i$ to elements $[\ol{e}_i]_{p}$, in $\qh_{ev}(M,\ol\Lambda_{\bF_p})$, where $\ol\Lambda_{\bF_{p}}$ is the algebraic closure of $\Lambda_{\bF_{p}}$. First, each $\ol{e}_i$ is of the form \[\overline{e}_i=\displaystyle\sum_{j=0}^n\overline{k}_{ij}h_j\] where $\overline{k}_{ij}\in\overline{Q(\Lambda_{\mathbb{Z}})}$ and $h_j\in \h_{2j}(M)$. Then, there is a simple extension of $Q(\Lambda_{\mathbb{Z}})$ containing $\{\bar{k}_{ij}\}_{i, j}$. Let $p_{ij}(x)\in Q(\Lambda_{\mathbb{Z}})[x]$ be the minimial polynomial of $\overline{k}_{ij}$ and let $\{\alpha_{ij}^l\}$ be all of its roots. Since \[Q(\Lambda_{\mathbb{Z}})(\{\alpha_{ij}^l\}_{i, j, l})\] is a finite separable extension, the primitive element theorem implies there is an element $\alpha\in\overline{Q(\Lambda_{\mathbb{Z}})}$ such that \[Q(\Lambda_{\mathbb{Z}})(\{\alpha_{ij}^l\}_{i, j, l})=Q(\Lambda_{\mathbb{Z}})(\alpha).\] In particular, $\overline{k}_{ij}\in Q(\Lambda_{\mathbb{Z}})(\alpha)$ for all $i=1, \cdots, m$ and $j=1, \cdots, n$. Let $f(x)$ be the minimal polynomial of $\alpha$. Denote \[f(x)=a_rx^r+a_{r-1}x^{r-1}+\cdots a_1x+a_0\] where $a_i\in Q(\Lambda_{\mathbb{Z}})$. Note that there is a polynomial $f_1(x)\in\Lambda_{\mathbb{Z}}[x]$ such that 
    \begin{align*}
        Q(\Lambda_{\mathbb{Z}})(\alpha)\cong\dfrac{Q(\Lambda_{\mathbb{Z}})[x]}{(f(x))}=\dfrac{Q(\Lambda_{\mathbb{Z}})[x]}{(f_1(x))}
    \end{align*}
    In fact, assume $a_i=\dfrac{a'_i}{a^{''}_i}$ where $a'_i, a^{''}_i\in\Lambda_{\mathbb{Z}}$. Then 
    \begin{align*}
        f_1(x):=\prod_{i=1}^ra^{''}_i\cdot f(x)\in\Lambda_{\mathbb{Z}}[x]
    \end{align*}
    Denote 
    \begin{align*}
        f_1(x)=b_rx^r+b_{r-1}x^{r-1}+\cdots+b_1x+b_0
    \end{align*}
    where $b_i=\sum b_{ij}T^{\lambda_j}\in\Lambda_{\mathbb{Z}}$ and $b_{ij}\in \mathbb{Z}$. Denote by $[b_i]_p=\sum(b_{ij}\mod p)T^{\lambda_j}$. Then 
    \begin{align*}
        [f_1(x)]_p=[b_r]_px^r+[b_{r-1}]_px^{r-1}+\cdots+[b_1]_px+[b_0]_p\in \Lambda_{\mathbb{F}_p}[x]
    \end{align*}
    Note that $[f_1(x)]_p\ne 0$ since we can choose $p$ such that a single $b_{ij}\mod p\ne 0$.
    We can write 
    \begin{align*}
        [f_1(x)]_p=g_1(x)^{m_1}g_2(x)^{m_2}\cdots g_s(x)^{m_s}
    \end{align*}
in $\Lambda_{\bF_p}[x]$, where $g_i(x)$ are irreducible and distinct from each other. 

Next, we show the multiplicity of each $g_i(x)$ is $1$ for sufficiently large $p$.

\begin{claim}
$m_1=m_2=\cdots=m_s=1$ for sufficiently large $p$.
\end{claim}
\begin{proof}
Since $f_1(x)$ is an irreducible polynomial over $Q(\Lambda_{\mathbb{Z}})[x]$ we have that $$\gcd(f(x), f'(x))=1$$ where $f'(x)$ is the derivative of $f(x)$. So
    \begin{align*}
        r(x)f(x)+q(x)f'(x)=1
    \end{align*}
    for some $r(x), q(x)\in Q(\Lambda_{\mathbb{Z}})[x]$. Denote
    \begin{align*}
        r(x)=\sum\dfrac{r'_i}{r^{''}_i}x^i\quad\text{and}\quad q(x)=\sum\dfrac{q'_j}{q^{''}_j}x^j
    \end{align*}
    where $r'_i, r^{''}_i, q'_j, q^{''}_j\in\Lambda_{\mathbb{Z}}$. Let $\Theta:=\prod r^{''}_i\cdot\prod q^{''}_j\cdot \prod a^{''}_i\in\Lambda_{\mathbb{Z}}$. Then 
    \begin{align}
        \tilde{r}(x)f_1(x)+\tilde{q}(x)f'_1(x)=\Theta
    \end{align}
    where $\tilde{r}(x)=\prod r^{''}_i\cdot\prod q^{''}_j\cdot r(x)\in \Lambda_{\mathbb{Z}}$ and $\tilde{q}(x)=\prod r^{''}_i\cdot\prod q^{''}_j\cdot q(x)\in \Lambda_{\mathbb{Z}}$. Applying the same way as we get $[f_1(x)]_p$ to Equation (1), we get
    \begin{align*}
        [\tilde{r}(x)]_p[f_1(x)]_p+[\tilde{q}(x)]_p[f'_1(x)]_p=[\Theta]_p
    \end{align*}
    By choosing sufficiently large $p$, we can have $[\Theta]_p\ne 0$. Thus $\gcd([f_1(x)]_p, [f'_1(x)]_p)=1$ in  to $K_{ij}(x)$, we, hence, $m_1=m_2=\cdots=m_s=1$.
\end{proof}
Thus, for a sufficiently large prime $p$, $[f_1(x)]_p=g_1(x)g_2(x)\cdots g_s(x)$ where $g_i(x)$ are irreducible and distinct from each other.
%\subsection{Reduction of idempotents}\label{sec: reduction_idempotent}

Recall that we can write $\overline{e}_i=\sum\overline{k}_{ij}h_j$ where $\overline{k}_{ij}\in Q(\Lambda_{\mathbb{Z}})(\alpha)$ is non-zero and that \[Q(\Lambda_{\mathbb{Z}})(\alpha)\cong\frac{Q(\Lambda_{\mathbb{Z}})[x]}{(f_1(x))}.\] Therefore, we can write $\overline{k}_{ij}$ as $K_{ij}(x)+(f_1(x))$ for $K_{ij}(x)\in Q(\Lambda_{\Z})[x]$. 

We repeat the process used to reduce $f_1(x)$ modulo $p$, i.e. to obtain $[f_1(x)]_p$, in order to reduce  to $K_{ij}(x)$ modulo $p$, and obtain $[K_{ij}(x)]_{p}\in\Lambda_{\bF_p}[x]$.

Denote, $$K_{ij}(x)=\sum\dfrac{K'_{ijk}}{K^{''}_{ijk}}x^k,\quad\Upsilon_{K_{ij}}:=\displaystyle\prod_kK^{''}_{ijk}.$$ Note that $\widetilde{K}_{ij}(x):=\Upsilon_{K_{ij}}\cdot K_{ij}$ is an element in $\Lambda_{\Z}[x]$. Let $[\widetilde{K}_{ij}(x)]_p$ be its reduction modulo $p$.

\begin{claim}
        For a sufficiently large prime $p$ we have that $[\widetilde{K}_{ij}(x)]_p+([f_1(x)]_p)$ is non-zero in ${\Lambda_{\mathbb{F}_p}[x]}/{([f_1(x)]_p)}$.
\end{claim}

    \begin{proof}
    Since $\bar{k}_{ij}$ is invertible, there is an element $L_{ij}(x)\in Q(\Lambda_{\mathbb{Z}})[x]$ such that 
\begin{align*}
        K_{ij}(x)L_{ij}(x)+(f_1(x))=1+(f_1(x))
\end{align*}
    in particular,
\begin{align*}
        K_{ij}(x)L_{ij}(x)=1+M_{ij}(x)f_1(x)
\end{align*}
    for $M_{ij}(x)\in Q(\Lambda_{\mathbb{Z}})$. Denote 
    \begin{align*}
	 L_{ij}(x)=\sum\dfrac{L'_{ijk}}{L^{''}_{ijk}}x^k,\quad \text{and}\quad M_{ij}(x)=\sum\dfrac{M'_{ijk}}{M^{''}_{ijk}}x^k,
    \end{align*}
    Let  $\Upsilon_{L_{ij}}:=\displaystyle\prod_kL^{''}_{ijk}$, and $\Upsilon_{M_{ij}}:=\displaystyle\prod_kM^{''}_{ijk}$. Then,
    \begin{align}
        \widetilde{K}_{ij}(x)\widetilde{L}_{ij}(x)\Upsilon_{M_{ij}}=\Upsilon_{K_{ij}}\Upsilon_{L_{ij}}\Upsilon_{M_{ij}}+\widetilde{M}_{ij}(x)f_1(x)\Upsilon_{K_{ij}}\Upsilon_{L_{ij}}
    \end{align}
    where 
    \begin{align*}
        \widetilde{K}_{ij}(x)=\Upsilon_{K_{ij}}\cdot K_{ij}(x),\quad \widetilde{L}_{ij}(x)=\Upsilon_{L_{ij}}\cdot L_{ij}(x),\quad \text{and}\quad \widetilde{M}_{ij}(x)=\Upsilon_{M_{ij}}\cdot M_{ij}(x),
    \end{align*}
    all belong to $\Lambda_{\mathbb{Z}}$. Applying the same way as we get $[f_1(x)]_p$ to Equation (2), we have
    \begin{align*}
        [\widetilde{K}_{ij}(x)]_p[\widetilde{L}_{ij}(x)]_p[\Upsilon_{M_{ij}}]_p=[\Upsilon_{K_{ij}}]_p[\Upsilon_{L_{ij}}]_p[\Upsilon_{M_{ij}}]_p+[\widetilde{M}_{ij}(x)]_p[f_1(x)]_p[\Upsilon_{K_{ij}}]_p[\Upsilon_{L_{ij}}]_p
    \end{align*}
    By choosing sufficiently large $p$, we can have $[\Upsilon_{K_{ij}}]_p\ne 0$, $[\Upsilon_{L_{ij}}]_p\ne 0$, $[\Upsilon_{M_{ij}}]_p\ne 0$, and $[f_1(x)]_p\ne 0$. Thus,
    \begin{align*}
        [\widetilde{K}_{ij}(x)]_p[\widetilde{L}_{ij}(x)]_p+([f_1(x)]_p)\ne 0\: \text{in}\: \dfrac{\Lambda_{\mathbb{F}_p}[x]}{([f_1(x)]_p)}
    \end{align*}
    In particular, 
    \begin{align*}
        [\widetilde{K}_{ij}(x)]_p+([f_1(x)]_p)\ne 0\: \text{in}\: \dfrac{\Lambda_{\mathbb{F}_p}[x]}{([f_1(x)]_p)}
    \end{align*}
     \end{proof}
  Note that 
\begin{align*}
        \dfrac{\Lambda_{\mathbb{F}_p}[x]}{([f_1(x)]_p)}=\dfrac{\Lambda_{\mathbb{F}_p}[x]}{(g_1(x)g_2(x)\cdots g_s(x))}\cong\prod_{l=1}^s \dfrac{\Lambda_{\mathbb{F}_p}[x]}{(g_l(x))}
\end{align*}
    and each term in the product is a field because $g_l(x)$ is irreducible. Let
\begin{align*}
        P_l: \dfrac{\Lambda_{\mathbb{F}_p}[x]}{([f_1(x)]_p)}\to \dfrac{\Lambda_{\mathbb{F}_p}[x]}{(g_l(x))}
\end{align*}
    be the projection. Then, there exists at least one $P_l$ such that 
\begin{align*}
        P_l\left([\widetilde{K}_{ij}(x)]_p[\Upsilon_{K_{ij}}]_p^{-1}+([f_1(x)]_p)\right)\ne 0
\end{align*}
    Furthermore, let 
    \begin{align*}
        \iota_l: \dfrac{\Lambda_{\mathbb{F}_p}[x]}{(g_l(x))}\to \bar{\Lambda}_{\mathbb{F}_p}
    \end{align*}
    be the inclusion, then 
    \begin{align*}
        \iota_l\left(P_l\left([\widetilde{K}_{ij}(x)]_p[\Upsilon_{K_{ij}}]_p^{-1}+([f_1(x)]_p)\right)\right)\ne 0
    \end{align*}
    We define 
    \begin{align*}[\bar{k}_{ij}]_p:=\iota_l\left(P_l\left([\widetilde{K}_{ij}(x)]_p[\Upsilon_{K_{ij}}]_p^{-1}+([f_1(x)]_p)\right)\right)
    \end{align*}
    and 
    \begin{align*}
        [\bar{e}_i]_p:=\sum[\bar{k}_{ij}]_ph_j
    \end{align*}
    where $[\bar{e}_i]_p\ne 0$ for $i=1, \cdots, m$ when $p$ is sufficiently large.

\begin{remark}
In fact, for any $l=1, \cdots, s$ and sufficiently large $p$,
    \begin{align*}
        P_l\left([\widetilde{K}_{ij}(x)]_p[\Upsilon_{K_{ij}}]_p^{-1}+([f_1(x)]_p)\right)\ne 0
    \end{align*}
    Indeed, suppose otherwise then, there exists $\delta(x)\in\Lambda_{\mathbb{Z}}[x]$ such that 
    \begin{align*}
        [\widetilde{K}_{ij}(x)]_p[\Upsilon_{K_{ij}}]_p^{-1}=g_l(x)\delta(x)
    \end{align*}
    Now we can write the Equation (3.3) as:
    \begin{align*}
        g_l(x)\delta(x)[\Upsilon_{K_{ij}}]_p[\widetilde{L}_{ij}(x)]_p[\Upsilon_{M_{ij}}]_p=[\Upsilon_{K_{ij}}]_p[\Upsilon_{L_{ij}}]_p[\Upsilon_{M_{ij}}]+[\widetilde{M}_{ij}(x)]_p[f_1(x)]_p[\Upsilon_{K_{ij}}]_p[\Upsilon_{L_{ij}}]_p
    \end{align*}
    Let $\eta$ be a root of $g_l(x)$. By plugging $\eta$ into the equation, we obtain $[\Upsilon_{K_{ij}}]_p[\Upsilon_{L_{ij}}]_p[\Upsilon_{M_{ij}}]=0$, which contradicts our choice of $p$.
\end{remark}

\begin{prop}\label{AA}
The reductions $[\overline{e}_i]_p$ are idempotents in $\qh_{ev}(M, \overline{\Lambda}_{\bF_p})$ for all $i=1, \cdots, m$. In addition, $[\overline{e}_i]_p*[\overline{e}_j]_p=0$ for $i\ne j$, and $\sum_{i=1}^m[\overline{e}_i]_p=1$ where $1$ is the multiplicative identity in $\qh_{ev}(M, \overline{\Lambda}_{\bF_p})$.
\end{prop}

\begin{proof}
Since $\qh_{ev}(M, \overline{Q(\Lambda_{\mathbb{Z}})})$ is semisimple and $\{\overline{e}_i\}_{i=1}^m$ are idempotents such that \[\qh_{ev}(M, \overline{Q(\Lambda_{\mathbb{Z}})})=\displaystyle\bigoplus_{i=1}^m\overline{e}_i*\qh_{ev}(M, \overline{Q(\Lambda_{\mathbb{Z}})}),\] we have the following equations:
\begin{equation*}
    \overline{e}_i*\overline{e}_i=\overline{e}_i
\end{equation*}
for $i=1, \cdots, m$,
\begin{equation*}
    \overline{e}_i*\overline{e}_j=0, i\ne j
\end{equation*}
for $i,j=1, \cdots, m$,
\begin{equation*}
    \displaystyle\sum_{i=1}^m\overline{e}_i=1
\end{equation*}
By applying the above procedure for getting $[\bar{e}_i]$ to the above equations, we have the following equations in $\qh_{ev}(M, \bar{\Lambda}_{\mathbb{F}_p})$:
    \begin{align*}
        [\bar{e}_i]_p*[\bar{e}_i]_p=[\bar{e_i}]_p
    \end{align*}
    for $i=1, \cdots, m$
    \begin{align*}
        [\bar{e}_i]_p*[\bar{e}_j]_p=0, i\ne j
    \end{align*}
    for $i,j=1, \cdots, m$
    \begin{align*}
        \sum_{i=1}^m[\bar{e}_i]_p=1
    \end{align*} 
Thus, \[\qh_{ev}(M, \overline{\Lambda}_{\bF_p})=\displaystyle\bigoplus_{i=1}^m[\overline{e}_i]_p* \qh_{ev}(M, \overline{\Lambda}_{\bF_p}).\]
\end{proof}

\subsubsection*{Proof of Theorem \ref{thm:QH_ss_alg}}
Note that each $\overline{e}_i*\qh_{ev}(M, \overline{Q(\Lambda_{\mathbb{Z}})})$ is an algebraic field extension of $\overline{Q(\Lambda_{\mathbb{Z}})}$, which in turn, is the algebraic closure of $Q(\Lambda_{\mathbb{Z}})$. Thus, \[\overline{e}_i*\qh_{ev}(M, \overline{Q(\Lambda_{\mathbb{Z}})})\cong\overline{Q(\Lambda_{\mathbb{Z}})}.\] Also, \[\qh_{ev}(M, \overline{Q(\Lambda_{\mathbb{Z}})})=\displaystyle\bigoplus_{i=1}^m\overline{e}_i*\qh_{ev}(M, \overline{Q(\Lambda_{\mathbb{Z}})})\cong\bigoplus_{i=1}^m\overline{Q(\Lambda_{\mathbb{Z}})}.\] Thus $m=\rank(\h_{ev}(M))$. Since $\qh_{ev}(M, \overline{\Lambda}_{\bF_p})$ is a free $\overline{\Lambda}_{\bF_p}$-module and $[\overline{e}_i]_p*\qh_{ev}(M, \ol\Lambda_{\bF_p})$ is a submodule, $[\overline{e}_i]_p*\qh_{ev}(M, \ol\Lambda_{\bF_p})$ can be written as a direct sum of copies of $\overline{\Lambda}_{\bF_p}$. The facts that $m=\rank(\h_{ev}(M, \mathbb{Q}))=\rank(\h_{ev}(M, \bF_{p}))$ for sufficiently large $p$, $[\overline{e}_i]_p$ is nonzero and \[[\overline{e}_i]_p*\qh_{ev}(M,\ol\Lambda_{\bF_p})\cap [\overline{e}_j]_p*\qh_{ev}(M,\ol\Lambda_{\bF_p})=0\] for $i\neq j$, imply that \[[\overline{e}_i]_p*\qh_{ev}(M, \ol\Lambda_{\bF_p})\cong\overline{\Lambda}_{\bF_p}.\] Thus $\qh_{ev}(M, \overline{\Lambda}_{\bF_p})$ is semisimple and is generated by the idempotents $\{[\overline{e}_i]_p\}_{i=1}^m$.\qed

\section{Upper bound on the boundary depth}\label{sec: upper}
In this section we prove the following:

\begin{thm}
\label{thm:upper_bound}
Suppose that $\qh_{ev}(M,\Lambda_{\bK})$ is semisimple and $\bK$ is a ground field of characteristic zero. Then, the boundary depth of each $\psi\in\Ham(M,\om)$, satisfies $\beta(\psi, \Lambda_{\bF_p})\leq C$ for all sufficiently large prime $p$, where $C$ is a positive constant independent of $p$.
\end{thm}

First, we find an upper bound, independent of sufficiently large $p$, for the valuations of the idempotents $\{[\ol{e}_i]_p\}$ from Section \ref{sec: semi}.

Recall from Section \ref{sec: reduction} that $f$ is the minimal polynomial of $\alpha$ and, reducing to $\Lambda_{\bF_p}[x]$, we obtain $[f_1(x)]_p=g_1(x)g_2(x)\cdots g_s(x)$ where $g_i(x)$ are irreducible and distinct from each other for sufficiently large $p$. Therefore, \[|[f_1(0)]_p|=\displaystyle\prod_{i=1}^s|g_i(0)|.\] Thus, for some $i$, we have $|g_i(0)|\le|[f_1(0)]_p|^{1/s}$. Without loss of generality, assume that $|g_1(0)|\le|[f_1(0)]_p|^{1/s}$. There is an algebraic element $\gamma$ over $\Lambda_{\bF_p}$ such that \[\frac{\Lambda_{\bF_p}[x]}{(g_1)}=\Lambda_{\bF_p}(\gamma).\] Since $\Lambda_{\bF_p}$ is complete, by Remark \ref{rmk: valuation_extension}, one obtains a norm $|\cdot|$ and, hence, a valuation $-\ln(|\cdot|)$ on $\Lambda_{\bF_p}(\gamma)$.

\begin{lem}
Let $b\in Q(\Lambda_{\mathbb{Z}})$. Then $|b|\ge|[b]_p|$ for sufficiently large $p$.
\end{lem}
\begin{proof}
By Proposition \ref{prop:finite_denominators} we can write $b=\sum_{i=-K}^\infty b_iT^{\lambda_i}$ with $\lambda_{-K}<\lambda_{-K+1}<\cdots$. Then $\nu(b)=\lambda_{-K}$. Write $b_{-K}=\frac{b_{-K,0}}{b_{-K, 1}}$. For a sufficiently large prime $p$, we have that $[b_{-K,1}]_p$ is invertible, thus, $[b_{-K}]_p=[b_{-K, 0}]_p[b_{-K, 1}]_p^{-1}$. If $[b_{-K}]_p=0$, then $\nu(b)<\nu([b]_p)$. If $[b_{-K}]_p\ne 0$, then $\nu(b)=\nu([b]_p)$. Since $|\cdot|=e^{-\nu(\cdot)}$, then $|b|\ge|[b]_p|$.
\end{proof}

\begin{remark}
The prime $p$ in the above lemma depends on $b$ and there is not a uniform $p$ for all elements in $Q(\Lambda_{\mathbb{Z}})$.
\end{remark}

\begin{prop}
$|\gamma|\le\max\{1, |f_1(0)|\}$.
\end{prop}
\begin{proof}
One can see that $g_1$ is both the characteristic polynomial and the minimal polynomial of $\gamma$. Thus $N_{\Lambda_{\bF_p}(\gamma)/\Lambda_{\bF_p}}(\gamma)=(-1)^Mg_1(0)$ where $M$ is the degree of $g_1$. Then $|\gamma|=|N_{\Lambda_{\bF_p}(\gamma)/\Lambda_{\bF_p}}(\gamma)|^{1/M}=|g_1(0)|^{1/M}$. Furthermore $|\gamma|\le|f_1(0)|^{\frac{1}{sM}}$. Let $N$ be the degree of $f_1$. If $|f_1(0)|\le 1$, then $|f_1(0)|^{\frac{1}{sM}}\le 1$. If $|f_1(0)|>1$, $|f_1(0)|^{\frac{1}{sM}}\le|f_1(0)|$. Hence, $|\gamma|\leq\max\{1, |f_1(0)|\}$ as required. 
\end{proof}

\begin{prop}\label{BBB}
There is a number $\delta$, independent of $p$, such that $l([\overline{e}_i]_p)\le\delta$ for all $i=1, \cdots, m$ and all sufficiently large $p$.
\end{prop}
\begin{proof}
Recall from Section \ref{sec: semi} that $\{\overline{e}_i=\sum_{j=0}^n\overline{k}_{ij}h_j\}_{i=1}^m$ and $\{[\overline{e}_i]_p=\sum_{j=0}^n[\overline{k}_{ij}]_ph_j\}_{i=1}^m$ are the idempotents of $\qh_{ev}(M, \overline{Q(\Lambda_{\mathbb{Z}})})$ and  $\qh_{ev}(M, \overline{\Lambda}_{\bF_p})$, respectively, where $\overline{k}_{ij}\in Q(\Lambda_{\mathbb{Z}})(\alpha)$ and $[\overline{k}_{ij}]_p\in\Lambda_{\bF_p}[x]/(g_1(x))=\Lambda_{\bF_p}(\gamma)$. Suppose $\overline{k}_{ij}=\sum_{s=0}^N b_{ijs}\alpha^s$. Then, $[\overline{k}_{ij}]_p=\sum_{s=0}^N[b_{ijs}]_p\gamma^s$. Denote by $\Xi:=\displaystyle\max_{i, j, s}\{|b_{ijs}|\}$. Thus, \[|[\overline{k}_{ij}]_p|\le\displaystyle\max_{0\le s\le N}\{|[b_{ijs}]_p\gamma^s|\}=\displaystyle\max_{0\le s\le N}\{|[b_{ijs}]_p||\gamma^s|\}\le\Xi(\max\{1, |f_1(0)|\})^N.\] Thus, \[l([\overline{e}_i]_p)=\displaystyle\max_{0\le j\le n}\{-\nu([\overline{k}_{ij}]_p)\}\le\ln(\Xi(\max\{1, f_1(0)\})^N).\] 
\end{proof}

\begin{definition}
Suppose $\qh_{ev}(M, \mathbb{K})$ is semisimple and $E=\{e_1, \cdots, e_m\}$ are idempotents. Then, define
\begin{align*}
	 \gamma_{e_j}(H, \mathbb{K})&=c(e_j, H, \mathbb{K})+c(e_j, \overline{H}, \mathbb{K})\\
	  \gamma_{e_j}(\phi, \mathbb{K})&=\displaystyle\inf_{\phi_H^1=\phi}\gamma_{e_j}(H, \mathbb{K})
\end{align*}
and 
\begin{align*}
	\gamma_E(\phi, \mathbb{K})=\displaystyle\max_{1\le i\le m}\gamma_{e_j}(\phi, \mathbb{K}).
\end{align*}
\end{definition}

Next, it is shown that from the upper bound on the valuations of the idempotents, it is possible to obtain an upper bound for $\gamma_{E_p}$ where $E_p=\{[\ol{e}_1]_p,\dots,[\ol{e}_m]_p\}$.

\begin{lem}\label{DDD}
Let $H$ be a non-degenerate Hamiltonian function and let $\theta\in\qh_*(M, \overline{\Lambda}_{\bF_p})$. Then, we have $c(\theta, 0, \overline{\Lambda}_{\bF_p})\ge 0$ if $\nu(\Delta(\pss_H(\theta), \pss_{\ol{H}}([M]))\le 0$.
\end{lem}

\begin{proof}
Suppose that $\theta=\sum \theta_ih_i$ with $\theta_i\in\overline{\Lambda}_{\bF_p}$ and $h_i\in\h_*(M, \bF_{p})$. Then \[l(\theta)=\max\{-\nu(\theta_i)\}=-\min\{\nu(\theta_i)\}.\] Observe that $\Delta(\pss_H(\theta), \pss_{\ol{H}}([M]))=\theta_j$ with $h_j=[pt]$. Thus, \[\nu(\Delta(\pss_H(\theta), \pss_{\ol{H}}([M])))=\nu(\theta_j)\ge -l(\theta)=-c(\theta, 0, \overline{\Lambda}_{\bF_p}).\] Since $\nu(\Delta(\pss_{H}(\theta), \pss_{\ol{H}}([M])))\le 0$, we have that $c(\theta, 0, \overline{\Lambda}_{\bF_p})\ge 0$.
\end{proof}

\begin{lem}\label{EEE}
Let $\theta\in \qh_{ev}(M, \overline{\Lambda}_{\bF_p})$ and $[\ol{e}_{i}]_{p}*\theta\neq0$. Then we have that \[c(([\overline{e}_i]_p*\theta)^{-1}, 0, \overline{\Lambda}_{\bF_p})+c([\overline{e}_i]_p*\theta, 0, \overline{\Lambda}_{\bF_p})=2c([\overline{e}_i]_p, 0, \overline{\Lambda}_{\bF_p}),\] where the inversion is taken in the field $[\ol{e}_{i}]_{p}*\qh_{ev}(M,\ol\Lambda_{\bF_p})$.
\end{lem}

\begin{proof}
Note that $[\overline{e}_i]_p*\qh_{ev}(M, \overline{\Lambda}_{\bF_p})\cong\overline{\Lambda}_{\bF_p}$. Thus, there is a $\kappa\in\overline{\Lambda}_{\bF_p}$ such that $[\overline{e}_i]_p*\theta=\kappa[\overline{e}_i]_p$. Hence, $([\overline{e}_i]_p*\theta)^{-1}=\kappa^{-1}[\overline{e}_i]_p$. We then obtain
\begin{align*}
	&c(([\overline{e}_i]_p*\theta)^{-1}, 0, \overline{\Lambda}_{\bF_p})+c([\overline{e}_i]_p*\theta, 0, \overline{\Lambda}_{\bF_p})\\
	&=c(\kappa^{-1}[\overline{e}_i]_p, 0, \overline{\Lambda}_{\bF_p})+c(\kappa[\overline{e}_i]_p, 0, \overline{\Lambda}_{\bF_p})\\
	&=c([\overline{e}_i]_p, 0, \overline{\Lambda}_{\bF_p})-\nu(\kappa^{-1})+c([\overline{e}_i]_p, 0, \overline{\Lambda}_{\bF_p})-\nu(\kappa)\\
	&=2c([\overline{e}_i]_p, 0, \overline{\Lambda}_{\bF_p}).
\end{align*}
\end{proof}

\begin{prop}
\label{prop:gamma_adapted_bound}
There exists a constant $D$, independent of $p$ for sufficiently large $p$, such that $\gamma_{E_p}(\phi)\le D$ for each $\phi\in \Ham(M, \omega)$, where $E_p=\{[\bar{e}_1]_p, \cdots, [\bar{e}_m]_p\}$.
\end{prop}
\begin{proof}
We note that, by the Hofer continuity of spectral invariants, it is sufficient to prove the statement for non-degenerate Hamiltonians. 
First, by Proposition \ref{CCC}, 
\begin{align*}
    &c([\bar{e}_i]_p, H)+c([\bar{e}_i]_p, \ol{H})\\
    &=c([\bar{e}_i]_p, H)-\inf\{c(b, H)\mid b\in \qh(M, \bar{\Lambda}_{\mathbb{F}_p}), \nu(\Delta(\text{PSS}_H(b), \text{PSS}_{\ol{H}}([\bar{e}_i]_p)))\le 0\}.
\end{align*}
By Lemma \ref{Pair_QH2}
\begin{align*}
    &c([\bar{e}_i]_p, H)+c([\bar{e}_i]_p, \ol{H})\\
    &=c([\bar{e}_i]_p, H)-\inf\left\{c(b, H)\middle| 
    \begin{matrix}
    b\in \qh(M, \bar{\Lambda}_{\mathbb{F}_p}),\\
    \nu(\Delta(\text{PSS}_H([M]), \text{PSS}_{\ol{H}}([\bar{e}_i]_p*b)))\le 0
    \end{matrix}
    \right\}\\
    & =\sup\left\{c([\bar{e}_i]_p, H)-c(b, H)\middle| 
    \begin{matrix}
    b\in \qh(M, \bar{\Lambda}_{\mathbb{F}_p}),\\
    \nu(\Delta(\text{PSS}_H([M]), \text{PSS}_{\ol{H}}([\bar{e}_i]_p*b)))\le 0
    \end{matrix}
    \right\}.
\end{align*}
The triangle inequality of spectral invariants implies 
\begin{align*}
    c([\bar{e}_i]_p, H) & \le c([\bar{e}_i]_p*b, H)+c(([\bar{e}_i]_p*b)^{-1}, 0)\\
     &\le c([\bar{e}_i]_p, 0)+c(b, H)+c(([\bar{e}_i]_p*b)^{-1}, 0),
\end{align*}
where the inverse $([\bar{e}_i]_p*b)^{-1}$ is taken in the field $[\bar{e}_i]_p* \qh_{ev}(M, \bar{\Lambda}_{\mathbb{F}_p})$. Since $c([\bar{e}_i]_p, 0)=l([\bar{e}_i]_p)$, it follows from Proposition \ref{BBB},
\begin{align*}
    c([\bar{e}_i]_p, H)-c(b, H)\le c(([\bar{e}_i]_p*b)^{-1}, 0)+\delta.
\end{align*}
By Lemma \ref{DDD}, 
\begin{align*}
    c([\bar{e}_i]_p, H)-c(b, H)\le c(([\bar{e}_i]_p*b)^{-1}, 0)+c([\bar{e}_i]_p*b, 0)+\delta,
\end{align*}
and by Lemma \ref{EEE}
\begin{align*}
    c([\bar{e}_i]_p, H)-c(b, H) & \le 2c([\bar{e}_i]_p, 0)+\delta\\
    & =2l([\bar{e}_i]_p)+\delta\\
    & \le 3\delta.
\end{align*}
\end{proof}

The following proposition relates the boundary depth to $\gamma_{E_p}$.

\begin{prop}
\label{prop:interleaving}
For sufficiently large $p$, 
\begin{align*}
    |\beta(\phi, \bar{\Lambda}_{\mathbb{F}_p})-\beta(\psi, \bar{\Lambda}_{\mathbb{F}_p})|\le \gamma_{E_p}(\phi\psi^{-1}, \bar{\Lambda}_{\mathbb{F}_p})+\delta
\end{align*}
where $\beta(\phi, \bar{\Lambda}_{\mathbb{F}_p})$ and $\beta(\psi, \bar{\Lambda}_{\mathbb{F}_p})$ are boundary depth for $\phi, \psi\in \Ham(M, \omega)$ and $\delta$ is a constant independent of $p$ for sufficiently large $p$.
\end{prop}
\begin{proof}
We again note that by the Hofer continuity of spectral invariants and of the boundary depth, it is enough to prove the statement for non-degenerate Hamiltonians. 

Recall from Section \ref{sec:persistence} that a persistence module $V$ over a field $\mathbb{K}$ is a collection of functions $a\mapsto V^a$ from the poset category of $\mathbb{R}$ to the category of vector spaces over $\mathbb{K}$.

For each $[\bar{e}_i]_p$, there is a morphism given by taking product with $[\bar{e}_i]_p$
\begin{align*}
    [\bar{e}_i]_p*: \text{HF}_{ev}(H,\bar{\Lambda}_{\mathbb{F}_p})\to\text{HF}_{ev}(H,\bar{\Lambda}_{\mathbb{F}_p})[l([\bar{e}_i]_p)].
\end{align*}
Denote the image of the morphism  by $Im([\bar{e}_i]_p*)$. 
As in \cite[Lemma 28]{shelukhin2022hofer}, we have the following:
\begin{lem}\label{interleave}
    The persistence modules $\hf_{ev}(H,\bar{\Lambda}_{\mathbb{F}_p})$ and $\displaystyle\bigoplus_{i=1}^mIm([\bar{e}_i]_p*)$ are $l(E_p)$-interleaved where $$l(E_p)=\displaystyle\max_{1\le i\le m}\left\{l([\bar{e}_i]_p)\right\}.$$ In particular, their interleaving distance is at most $l(E_p)$.
\end{lem}
%In great generality, one can define the boundary depth $\beta({V})$ of the persistence module ${V}$ as the infimum of all $\lambda\in(0, \infty)$ with the property that, for all $s\in\mathbb{R}$,
%\begin{align*}
%    \ker\left(V^s\to\lim_{\longrightarrow}V^t\right)=\ker\left(V^s\to V^{s+\lambda}\right)
%\end{align*}
%Note the set on the left hand side is the same as the ascending union $$\displaystyle\bigcup_{t\in[s, \infty)}\ker(V^s\to V^t).$$ Depending on $V$, there may be no $\lambda$ with this property, in which case $\beta({V})=\infty$. However, as in \cite{usher2016persistent}, the Floer homology persistence module $\text{HF}(H)$ can easily be checked to have $\beta(\phi_H^1, \bar{\Lambda}_{\mathbb{F}_p})$ equal to the length of the longest finite-length bar in the barcode (or zero if there are no finite-length bars). 

Note also that by Proposition \ref{prop:sum_boundary} the finite collection of persistence modules $\left\{Im([\bar{e}_i]_p*)\right\}$ satifies:
\begin{align*}
    \beta(\bigoplus_jIm([\bar{e}_i]_p*))=\max_{j}\beta\left(Im([\bar{e}_i]_p*)\right).
\end{align*}
Furthermore, Proposition \ref{prop:boundary-depth_stability} and Lemma \ref{interleave} we have the following:
\begin{lem}\label{finitedepth}
    The boundary depth $\beta(\displaystyle\bigoplus_jIm([\bar{e}_i]_p*))$ is finite and 
    \begin{align*}
        |\beta(\phi_H^1, \bar{\Lambda}_{\mathbb{F}_p})-\beta(\bigoplus_{j}Im([\bar{e}_i]_p*))|\le 2l(E_p)
    \end{align*}
\end{lem}
Thus $\beta(\phi_H^1, \bar{\Lambda}_{\mathbb{F}_p})\le \max\beta\left(Im([\bar{e}_i]_p*)\right)+2l(E_p)$. Next, we bound $\beta\left(Im([\bar{e}_i]_p*)\right)$.\\
Let $F$ and $G$ be Hamiltonians such that $\phi_F^1=\phi$ and $\phi_G^1=\psi$. Define $Im\left([\bar{e}_i]_p, F\right)$ as the image of the map
\begin{align*}
    [\bar{e}_i]_p*: \text{HF}(F,\bar{\Lambda}_{\mathbb{F}_p})\to \text{HF}(F,\bar{\Lambda}_{\mathbb{F}_p})[l([\bar{e}_i]_p)]
\end{align*}
and $Im\left([\bar{e}_i]_p, G\right)$ as the image of the map
\begin{align*}
    [\bar{e}_i]_p*: \text{HF}(G,\bar{\Lambda}_{\mathbb{F}_p})\to \text{HF}(G,\bar{\Lambda}_{\mathbb{F}_p})[l([\bar{e}_i]_p)]
\end{align*}
Following the proof of \cite[Proposition 12]{shelukhin2022hofer}, $Im\left([\bar{e}_i]_p, F\right)$ and $Im\left([\bar{e}_i]_p, G\right)$ are $$\left(\dfrac{1}{2}\gamma_{[\bar{e}_i]_p}(G\#\overline{F}, \bar{\Lambda}_{\mathbb{F}_p})+\dfrac{1}{2}\epsilon\right)$$ -interleaved for arbitrarily small $\epsilon$.\\
By Lemma \ref{finitedepth}, $\beta\left(Im\left([\bar{e}_i]_p, F\right)\right)$ and $\beta\left(Im\left([\bar{e}_i]_p, G\right)\right)$ are finite and 
\begin{align*}
    \left|\beta(\phi_F^1, \bar{\Lambda}_{\mathbb{F}_p})-\max_i\beta\left(Im\left([\bar{e}_i]_p, F\right)\right)\right|\le 2l(E_p)\\
    \left|\beta(\phi_G^1, \bar{\Lambda}_{\mathbb{F}_p})-\max_i\beta\left(Im\left([\bar{e}_i]_p, G\right)\right)\right|\le 2l(E_p)
\end{align*}
Then 
\begin{align*}
    \left|\max_i\beta\left(Im\left([\bar{e}_i]_p, F\right)\right)-\max_i\beta\left(Im\left([\bar{e}_i]_p, G\right)\right)\right|\le\gamma_{[\bar{e}_i]_p}(G\#\overline{F}, \bar{\Lambda}_{\mathbb{F}_p})+\epsilon 
\end{align*}
Now the result in Proposition \ref{prop:interleaving} easily follows.
\end{proof}

\begin{thm}
\label{thm:upper_closure}
Suppose that $\qh_{ev}(M, \Lambda_{\mathbb{K}})$ is semisimple, where the ground field $\mathbb{K}$ has characteristic $0$. Then, the boundary depth of each $\psi\in \Ham(M, \omega)$ satisfies $\beta\left(\psi, \bar{\Lambda}_{\mathbb{F}_p}\right)\le C$, where $C$ is independent of $p$. 
\end{thm}
\begin{proof}
The result follows from Proposition \ref{prop:gamma_adapted_bound} and Proposition \ref{prop:interleaving}. From the proof of Proposition \ref{prop:gamma_adapted_bound} and Proposition \ref{BBB}, one can see that the constant $C$ only depends on the idempotents of $\text{QH}_{ev}(M, \Lambda_{\mathbb{K}})$.
\end{proof}

\begin{proof}[Proof of Theorem \ref{thm:upper_bound}]
Since $\beta\left(\psi, \bar{\Lambda}_{\mathbb{F}_p}\right)=\beta\left(\psi, \Lambda_{\mathbb{F}_p}\right)$, the conclusion follows from Theorem \ref{thm:upper_closure}. 
\end{proof}

 \section{The $\Z_p$-equivariant Floer homology}\label{sec: smith}
This section is a combination of \cite{shelukhin2022hofer, shelukhin2021mathbb, seidel2015equivariant, sugimoto2021hofer}. We mainly follow the ideas of Sugimoto in \cite{sugimoto2021hofer} to extend the definition of the $\Z_p$-equivariant pair-of-paints product introduced in \cite{shelukhin2022hofer} to the semipositive setting.

%We note that the definition of $\mathbb{Z}_p$-equivariant Floer homology in \cite[Section~6.4]{sugimoto2021hofer} only requires $\phi$ to have isolated fixed points. However, in what follows, we use \cite[Section~6.4]{sugimoto2021hofer} in the nondegenerate case, and use \cite[Section~7]{shelukhin2022hofer} for the more general case that only assumes isolated fixed points, for the purpose of proving the Smith-type inequality.
\subsection{The $\Z_p$-equivariant Floer homology of $\cf(\phi, \lamzero_{\bF_{p}})^{\otimes p}$}
Suppose $H$ is a nondegenerate Hamiltonian generating $\phi$. We proceed as in \cite[Section~7.1]{shelukhin2022hofer} to construct an $X_\infty$-module that yield $\h(\Z_p,\cf(\phi,\lamzero_{\bF_p})^{\otimes p})$. The modified Floer differential 
\begin{equation*}
	 d_{\Z_p}:\cf(H,\lamzero_{\bF_p})^{\otimes p}\otimes\lamzero_{\bF_p}[[u]]\brat{\theta}\rightarrow\cf(H,\lamzero_{\bF_p})^{\otimes p}\otimes\lamzero_{\bF_p}[[u]]\brat{\theta}
\end{equation*}
is a $\lamzero_{\bF_p}[[u]]\brat{\theta}$ of 
\begin{align*}
	d_{\Z_p}({x})&=d^{(p)}{x}+(1-\tau){x}\otimes\theta\\
	d_{\Z_p}(x\otimes\theta)&=-d^{(p)}{x}\otimes\theta+(1+\tau+\cdots+\tau^{p-1})x\otimes u
\end{align*}

where $x={x}_1\otimes\cdots\otimes{x}_p$, $d^{(p)}$ is the naturally induced differential on $\cf(\phi, \lamzero_{\bF_{p}})^{\otimes p}$ and $$\tau({x}_1\otimes\cdots\otimes{x}_p)=(-1)^{|x_{p}|(|x_1|+\cdots+|x_{p-1}|)}x_p\otimes{x}_1\otimes\cdots\otimes{x}_{p-1}.$$ Then, $d_{\Z_p}$ determines an $X_\infty$-module structure on \[\cf(H,\lamzero_{\bF_p})^{\otimes p}\otimes\lamzero_{\bF_p}\brat{\theta},\] which is also an $\epsilon$-gapped $X_{\infty}$-module as in \cite[Section~6.4]{sugimoto2021hofer}. Finally we define $\text{H}(\mathbb{Z}_p, \text{CF}(\phi, \Lambda_{\mathbb{F}_p}^0)^{\otimes p})$ as the homology of $(\text{C}(\mathbb{Z}_p, \text{CF}(\phi, \Lambda_{\mathbb{F}_p}^0)^{\otimes p}), d_{\Z_p})$.

%Note that $(C_{K,k},d_{\Z_p})$ and $(C_{K,l},d_{\Z_p})$ are chain homotopy equivalent and, hence, $X_\infty$-homotopy equivalent, for any $k$ and $l$. We define $\iota_{K, k\to l}$ as this chain homotopy and $\tau_{K\to K+1, k}$ as the constant morphism. Thus, we have a directed family of $\epsilon$-gapped $X_k$-modules $$\{C_{K, k}, \iota_{K, k\to l}, \tau_{K\to K+1, k}\}.$$ By \cite[Proposition~44]{sugimoto2021hofer}, there is a unique $\epsilon$-gapped $X_\infty$-module $(C, \mathfrak{d})$ up to $\epsilon$-gapped $X_\infty$-homotopy equivalence. By the proof of \cite[Proposition~44, Proposition~42]{sugimoto2021hofer}, one can take $C=C_{L, i}$ for a fixed $(L, i)$. 
\begin{remark}
For the definitions of $\epsilon$-gapped $X_K$-modules and $\epsilon$-gapped $X_K$-morphisms, one can refer to \cite[Section 6.2]{sugimoto2021hofer}. For the definitions of directed families of $\epsilon$-gapped $X_K$-modules and direct families of $\epsilon$-gapped $X_K$-morphisms, one can refer to \cite[Section 6.3]{sugimoto2021hofer}.
\end{remark}

%\begin{remark}
%Since we assume that $\phi$ is nondegenerate, we can take $\{H_k\}$ to be a constant sequence, i.e. $H_k=H$ for any $k$.
%\end{remark}
 
\subsection{The $\Z_p$-equivariant Floer homology of $\cf(\phi^p, \lamzero_{\bF_{p}})$}
\subsubsection{Nondegenerate case}\label{sec:nondeg_equiv}
Suppose $\phi$ is a nondegenerate Hamiltonian diffeomorphism generated by $H$. Let $f$ be a $\Z_p$-invariant Morse function on $S^\infty$, where the $\Z_p$-action on $S^\infty$ is given by scalar multiplication by the $p$-th root of unity. For each degree $k\in\mathbb{N}$, there are $p$ critical points denoted by $Z_k^m$, $m\in\{0, 1, \cdots, p-1\}$. The critical points contained in $S^{2k-1}$ are $\{Z_j^m\}$ with $j\in\{0, 1, \cdots, 2k+1\}$ and $m\in\{0, 1,\cdots, p-1\}$. Fix an almost complex structure $J$ and recall that $H^{(p)}$ is a generator of $\phi^p$. We now proceed as in \cite{sugimoto2021hofer} to construct the relevant $\epsilon$-gapped $X_K$-modules. %Choose $\epsilon>0$ to be less than the minimal symplectic area of a $J$-holomorphic sphere and the minimal energy of a non-constant Floer cylinder with respect to $H^{(p)}$. %Let $\{G_{(K,i)}\}$ be a family of Hamiltonian functions such that $G_{K,i}$ is sufficiently close to $H^{(p)}$, $G_{(K,i)}\rightarrow H^{(p)}$ in the $C^{\infty}$-topology as $i$ tends to infinity, and $(G_{(K,i)},J)$ is a regular pair. 
Furthermore, let $\cH^{(K,i)}$ be a family of Hamiltonian functions parametrized by $(w,t)\in S^{2K+1}\times S^1$ satisfying:
\begin{enumerate}[label=(\roman*)]
	\item For all $w$ in a small neighbourhood of $\{Z_{i}^m\}$, $\cH^{(K,i)}_{w,t}(x)=H^{(p)}(t-m/p,x)$,
	\item For all $m\in\Z_p$ and $w\in S^\infty$,  $\cH^{(K,i)}_{m\cdot w,t}=\cH^{(K,i)}_{w,t-m/p}$,
	\item $\cH^{(K,i)}$ is invariant under shift by $\tau$, i.e.  $\cH^{(K,i)}_{\tau(w),t}=\cH^{(K,i)}_{w,t}$.
\end{enumerate}
Let $x,y$ be $1$-periodic orbits of $H^{(p)}$, $m\in\Z_p$, $\lambda\geq 0$, $\alpha\in\{0,1\}$ and $0\leq l\leq 2K+1$. Consider the following perturbed Cauchy-Riemann equation
\begin{equation*}
\begin{cases}
\overbar{\partial}_Ju+X_{\cH^{(K,i)}_{t, w(s)}}(u)^{0, 1}=0\\
\partial_sw+\nabla f(w)=0
\end{cases}
\end{equation*}
subject to
\begin{equation*}
\begin{cases}
\lim_{s\to-\infty}(u(s, t), w(s))=(x(t), Z_\alpha^0)\\
\lim_{s\to+\infty}(u(s, t), w(s))=(y(t-m/p), Z_l^m).
%\int_{u}\om+\int^1_0 H^{(p)}(t,{x}(t))-H^{(p)}(t,{y}(t))dt =0.
\end{cases}
\end{equation*}
Counting solutions to the above perturbed Cauchy-Riemann equation allows one to define a map
\begin{equation*}
	d^{(K,i)}_{\alpha,l}:\cf(H^{(p)}, \lamzero_{\bF_{p}})\rightarrow\cf(H^{(p)}, \lamzero_{\bF_{p}}),
\end{equation*}
which induces an $\epsilon$-gapped $X_K$-module $(\cf(H^{(p)}, \lamzero_{\bF_{p}})\otimes\lamzero_{\bF_p}\brat{\theta}, \{\delta^{(K,i)}_l\}_{l=0}^K)$, where
\begin{align*}
	\delta^{(K,i)}_l({x}\otimes 1)&=d^{(K,i)}_{0,2l}({x})\otimes1 + d^{(K,i)}_{0,2l+1}({x})\otimes\theta\\
	\delta^{(K,i)}_l({x}\otimes\theta)&=d^{(K,i)}_{1,2l}({x})\otimes1 + d^{(K,i)}_{1,2l+1}({x})\otimes\theta.
\end{align*}
In particular, $\delta^{(K,i)}\circ\delta^{(K,i)}=0\mod (u^{K+1})$, where
\begin{equation*}
	\delta^{(K,i)}=\delta^{(K,i)}_0+\delta^{(K,i)}_1\otimes u + \cdots+\delta^{(K,i)}_K\otimes u^K.
\end{equation*}
We next define $\epsilon$-gapped $X_K$-morphisms \[\iota_{(K,i\rightarrow j)}:(\cf(H^{(p)}, \lamzero_{\bF_{p}})\otimes\lamzero_{\bF_p}\brat{\theta},\delta^{(K,i)})\rightarrow(\cf(H^{(p)}, \lamzero_{\bF_{p}})\otimes\lamzero_{\bF_p}\brat{\theta},\delta^{(K,j)})\] as follows. Consider a family $\cH^{(K,i\rightarrow j)}_{s,w,t}$ of Hamiltonian functions connecting $\cH^{(K,i)}_{w,t}$ to $\cH^{(K,j)}_{w,t}$, which is parametrized by $\R\times S^{2K+1}\times S^{1}$ and satisfies:
\begin{enumerate}[label=(\roman*)]
	\item For $s\ll0$, $\cG^{(K,i\rightarrow j)}_{s,w,t}(x)=\cH^{(K,i)}_{w,t}(x)$, and for $s\gg0$, $\cH^{(K,j)}_{s,w,t}=\cH^{(K,j)}_{w,t}$.
	\item For all $m\in \Z_p$ and $w\in S^\infty$,  $\cH^{(K,i\rightarrow j)}_{s,m\cdot w,t}=\cH^{(K,i\rightarrow j)}_{s,w,t-m/p}$,
	\item $\cH$ is invariant under shift by $\tau$, i.e.  $\cH^{(K,i\rightarrow j)}_{s,\tau(w),t}=\cH^{(K,i\rightarrow j)}_{s,w,t}$.
\end{enumerate}
Let $x,y$ be $1$-periodic orbits of $H^{(p)}$, $m\in\Z_p$, $\lambda\geq 0$, $\alpha\in\{0,1\}$ and $0\leq l\leq 2K+1$. Define $\iota_{\alpha,l,(K,i\rightarrow j)}$ by counting solutions to the following perturbed Cauchy-Riemann equation
\begin{equation*}
\begin{cases}
\overbar{\partial}_Ju+X_{\cH^{(K,i\rightarrow j)}_{s,w(s), t}}(u)^{0, 1}=0\\
\partial_sw+\nabla f(w)=0
\end{cases}
\end{equation*}
subject to
\begin{equation*}
\begin{cases}
\lim_{s\to-\infty}(u(s, t), w(s))=(x(t), Z_\alpha^0)\\
\lim_{s\to+\infty}(u(s, t), w(s))=(y(t-m/p), Z_l^m).
%\int_{u}\om+\int^1_0 H^{(p)}(t,{x}(t))-H^{(p)}(t,{y}(t))dt =\lambda.
\end{cases}
\end{equation*}
Set,
\begin{align*}
	\iota_{(K,i\rightarrow j),l}({x}\otimes 1)&=\iota_{0,2l,(K,i\rightarrow j)}({x})\otimes1 + \iota_{0,2l+1,(K,i\rightarrow j)}({x})\otimes\theta\\
	\iota_{(K,i\rightarrow j),l}({x}\otimes\theta)&=\iota_{1,2l,(K,i\rightarrow j)}({x})\otimes1 + \iota_{1,2l+1,(K,i\rightarrow j)}({x})\otimes\theta.
\end{align*}
It follows from \cite[Section 6]{sugimoto2021hofer} that $\iota_{(K,j\rightarrow k)}\circ\iota_{(K,i\rightarrow j)}$ is $\epsilon$-gapped homotopic to $\iota_{(K,i\rightarrow k)}$. In a similar fashion, it is possible to define $\epsilon$-gapped $X_K$-morphisms  \[\tau_{(K\rightarrow K+1,i)}:(\cf(H^{(p)}, \lamzero_{\bF_{p}})\otimes\lamzero_{\bF_p}\brat{\theta},\delta^{(K,i)})\rightarrow(\cf(H^{(p)}, \lamzero_{\bF_{p}})\otimes\lamzero_{\bF_p}\brat{\theta},\delta^{(K+1,i)}).\] We note that $i$ has to be sufficiently large so that $\tau_{(K\rightarrow K+1,i)}$ is defined over $\lamzero_{\bF_p}$. It turns out that \[\bigg\{\cf(H^{(p)}, \lamzero_{\bF_{p}})\otimes\lamzero_{\bF_p}\brat{\theta},\,\,\iota_{(K,i\rightarrow j)},\,\,\tau_{(K\rightarrow K+1,i)}\bigg\}\] is a directed family of $X_K$-modules. By \cite[Proposition 44]{sugimoto2021hofer}, there is a unique $\epsilon$-gapped $X_\infty$-module $(D, \mathfrak{l})$ up to $\epsilon$-gapped $X_\infty$-homotopy equivalence. Finally we define the chain complex $\text{C}(\mathbb{Z}_p, \text{CF}(\phi^p, \Lambda_{\mathbb{F}_p}^0))$ as $D$ and $\text{H}(\mathbb{Z}_p, \text{CF}(\phi^p, \Lambda_{\mathbb{F}_p}^0))$ as the homology of $(D, \mathfrak{l})$.

\subsubsection{degenerate case}\label{sec:deg_equiv}
Assume that $\phi$ has finitely many isolated fixed points. Now we follow \cite[Section 7.4.1]{shelukhin2022hofer} to define a canonical $\Lambda_{\mathbb{F}_p}^0$-complex. For $H_1$ a sufficiently $C^2$-small nondegenerate perturbation of $H$, then as in Section \ref{sec:nondeg_equiv}, we have an $\epsilon$-gapped $X_\infty$-module $$(\text{CF} (H_1^{(p)}, \Lambda_{\mathbb{F}_p}^0 )\otimes\Lambda_{\mathbb{F}_p}^0\langle\theta\rangle, \{\delta_l\}_{l\ge 0}).$$ Define $\delta_{\infty};=\displaystyle\sum_{l=0}^\infty\delta_l\otimes u^l$. Then $$ (\text{CF} (H_1^{(p)}, \Lambda_{\mathbb{F}_p}^0)\otimes\Lambda_{\mathbb{F}_p}^0[u^{-1}, u]]\langle\theta\rangle, \delta_{\infty})$$ is a chain complex. By the $\epsilon$-gapped condition, $\delta_l=\delta_{l, loc}+T^{\epsilon}\delta_{l, \epsilon}$. Thus $\delta_{\infty}$ is of the form 
\begin{align*}
    \delta_{\infty}=\delta_{loc}+T^\epsilon\delta_{\epsilon}
\end{align*}
where $\delta_{loc}=\displaystyle\sum_{l=0}^\infty\delta_{l, loc}\otimes u^l$ is the differnetial on $\text{CF}^{loc} (H_1^{(p)}, \Lambda_{\mathbb{F}_p}^0)\otimes\Lambda_{\mathbb{F}_p}^0[u^{-1}, u]]\langle\theta\rangle$ with $$\text{CF}^{loc}(H_1^{(p)}, \Lambda_{\mathbb{F}_p}^0):=\displaystyle\bigoplus_{x\in\text{Fix}(\phi^p)}\text{CF}(H_1^{(p)}, x).$$ Recall 
\begin{align*}
    \delta_l({x}\otimes 1)=d_{0, 2l}({x})\otimes 1+d_{0, 2l+1}({x})\otimes\theta\\
    \delta_l({x}\otimes\theta)=d_{1, 2l}({x})\otimes 1+d_{1, 2l+1}({x})\otimes\theta\\
\end{align*}
Then 
\begin{align*}
    \delta_{\infty}({x}\otimes 1) &=d_F({x})\otimes 1+d_{0, 1}({x})\otimes\theta+ud_{0, 2}({x})\otimes 1+ud_{0,3}({x})\otimes\theta+\cdots\\
    \delta_{\infty}({x}\otimes\theta) &=d_F({x})\otimes\theta+ud_{1,2}({x})\otimes 1+ud_{1,3}({x})\otimes\theta+\cdots
\end{align*}
where $d_F$ is the Floer differential on $\text{CF}(H_1^{(p)}, \Lambda_{\mathbb{F}_p}^0)$. Then define
\begin{align*}
    d_0({x}\otimes 1) &:=d_F({x})\otimes 1\\
    d_0({x}\otimes\theta) &:=d_F({x})\otimes\theta\\
    d_1({x}\otimes 1) &:=d_{0, 1}({x})\otimes\theta\\
    d_1({x}\otimes\theta) &:=ud_{1,2}({x})\otimes 1\\
\end{align*}
Now $d_0+d_1$ defines the differential $d_{cone}$ on
\begin{align*}
    cone(d_{0,1}: \text{CF}(H_1^{(p)}, \Lambda_{\mathbb{F}_p}^0)\to \text{CF}(H_1^{(p)}, \Lambda_{\mathbb{F}_p}^0))\cong\text{CF}(H_1^{(p)}, \Lambda_{\mathbb{F}_p}^0)\otimes 1\oplus\text{CF}(H_1^{(p)}, \Lambda_{\mathbb{F}_p}^0)\otimes\theta
\end{align*}
with 
\begin{align*}
    d_{cone}({x}_0, {x}_1)=\left(d_F({x}_0), -d_F({x}_1)+d_{0, 1}({x}_0)\right)
\end{align*}
First we construct a canonical $\Lambda_{\mathbb{F}_p}^0$-complex for the cone as follows. By the crossing energy argument, write $d_F=d_{loc}+T^{\epsilon_0}D_d$ where $D_d$ is a $\Lambda_{\mathbb{F}_p}^0$-morphism. Let $X$ be the part of a $\Lambda_{\mathbb{F}_p}^0$-basis of $\text{CF}^{loc}(H_1^{(p)}, \Lambda_{\mathbb{F}_p}^0)$ giving the free part of the homology. The homological perturbation formula gave us a differential $\overline{d}$ on $X$, and $(X, \overline{d})$ is $\delta_0$-quasi-equivalent to $(\text{CF}(H_1^{(p)}, \Lambda_{\mathbb{F}_p}^0), d_F)$. Furthermore $(X\oplus X, \overline{d}\oplus-\overline{d})$ is $\delta_0$-quasi-equivalent to $(\text{CF}(H_1^{(p)}, \Lambda_{\mathbb{F}_p}^0)\oplus \text{CF}(H_1^{(p)}, \Lambda_{\mathbb{F}_p}^0), d_F\oplus-d_F)$. 

Applying the homological perturbation formula again to $$(cone(d_{0,1}), d_{cone}=(d_F\oplus-d_F)+\widetilde{d}_{0,1})$$ and $(X\oplus X, \overline{d}\oplus-\overline{d})$ where $\widetilde{d}_{0, 1}({x}_0, {x}_1)=(0, d_{0, 1}({x}_0))$, we get a differential $\overline{d}_{cone}$ on $X\oplus X$, and $(X\oplus X, \overline{d}_{cone})$ is $2\delta_0$-quasi-equivalent to $(cone(d_{0,1}), d_{cone})$. 

On $cone(d_{0, 1})\otimes \Lambda_{\mathcal{K}}^0$ with $\mathcal{K}:=\mathbb{F}_p[u^{-1}, u]]$, $\delta_{\infty}=d_{cone}+uE_{loc}+T^\epsilon uE$. Applying the homological perturbation formula to $(cone(d_{0,1})\otimes\Lambda_{\mathcal{K}}^0, \delta_{\infty}=d_{cone}+uE_{loc}+T^\epsilon uE)$ and $((X\oplus X)\otimes\Lambda_{\mathcal{K}}^0, \overline{d}_{cone})$, we get a differential $\overline{\delta}_{\infty}$ on $(X\oplus X)\otimes\Lambda_{\mathcal{K}}^0$, and $((X\oplus X)\otimes\Lambda_{\mathcal{K}}^0, \overline{\delta}_\infty)$ is $3\delta_1$-quasi-equivalent to $(cone(d_{0,1})\otimes\Lambda_{\mathcal{K}}^0, \delta_{\infty})$ with $\delta_1:=(4n+2)\delta_0$. Finally we define $\text{C}(\mathbb{Z}_p, \text{CF}(\phi^p, \Lambda_{\mathbb{F}_p}^0))$ as $(X\oplus X)\otimes\Lambda_{\mathcal{K}}^0$ with differential $\overline{\delta}_\infty$, and it is $3\delta_1$-quasi-equivalent to $\text{C}(\mathbb{Z}_p, \text{CF}(\phi_1^p, \Lambda_{\mathbb{F}_p}^0))$.

\subsection{The $\Z_p$-equivariant pair-of-pants product}
Suppose $\phi$ is a nondegenerate Hamiltonian diffeomorphism generated by $H$. The $\Z_p$-equivariant pair-of-pants product is an $\epsilon$-gapped $X_\infty$-morphism induced by the $\epsilon$-gapped $X_K$-morphism
\begin{equation*}
	\mathfrak{f}^{(K,i)}:\cf(H,\lamzero_{\bF_p})^{\otimes p}\otimes\lamzero_{\bF_p}\brat{\theta}\rightarrow\cf(H^{(p)}, \lamzero_{\bF_{p}})\otimes\lamzero_{\bF_p}\brat{\theta}
\end{equation*}
between the two directed families of $\epsilon$-gapped morphisms that were discussed above. Therefore, we obtain a map 
\begin{equation*}
	\h(\Z_p,\cf(\phi,\lamzero_{\bF_p})^{\otimes p})\ra\h(\Z_p,\cf(\phi^p,\lamzero_{\bF_p}))
\end{equation*}
as desired. We proceed to recall the construction of $\mathfrak{f}^{(K,i)}$. Set
\begin{equation*}
	\Sigma_p=\Bigg(\bigsqcup_{0\leq k\leq p-1}\R\times[k,k+1]\Bigg)/\sim,
\end{equation*}
where the equivalence relation is given by
\begin{enumerate}[label=(\roman*)]
	\item For all $0\leq k\leq p-1$, $(s,k)\in[0,\infty)\times[k-1,k]\sim(s,k)\in[0,\infty)\times[k,k+1]$,\\
	\item for all $0\leq k\leq p-1$, $(s,k)\in(-\infty,0]\times[k,k+1]\sim(s,k+1)\in(-\infty,0]\times[k,k+1]$,\\
	\item and $(s,0)\in[0,\infty)\times[0,1]\sim(s,p)\in[0,\infty)\times[p-1,p]$.
\end{enumerate}
Determine a complex structure near the point $[(0,0)]\in\Sigma_p$, which is in the image of the following coordinate chart $w:B(1/2)\subset\C\rightarrow\Sigma_p$ defined as
\begin{equation*}
w(z)=
\begin{cases}
	[z^p], \qquad &0\leq\arg(z)\leq\pi/p,\\
	[z^p+k\sqrt{-1}],  &(2k-1)\pi/p\leq\arg(z)\leq(2k+1)\pi/p,\\
	[z^p+p\sqrt{-1}],  &(2p-1)\pi/p\leq\arg(z)\leq2\pi.
\end{cases}
\end{equation*}
Assume that $H(t,x)=0$ near $t=0$ and let $\cK^{(K,i)}$ be a family of Hamiltonians parametrized by $(w,z)\in S^{2K+1}\times\Sigma_p$ satisfying the following:
\begin{enumerate}[label=(\roman*)]
	\item For all $z=[s,t]\in\Sigma_p$, $\cK^{(K,i)}_{w,z}(x)=H([t],x)$ when $s\ll0$, and $\cK^{(K,i)}_{w,z}(x)=(1/p)H^{(p)}(x, t/p)$, when $s\gg0$,
	\item For all $z=[s,t]\in\Sigma_p$, $\cK^{(K,i)}_{w,z}\xrightarrow{i\rightarrow\infty}H(t,x)$,
	\item For all $m\in \Z_p$ and $w\in S^{2K+1}$,  $\cK^{(K,i)}_{m\cdot w,t}=\cK^{(K,i)}_{w,z+mi}$,
	\item $\cK$ is invariant under shift by $\tau$, i.e.  $\cK^{(K,i)}_{\tau(w),t}=\cK^{(K,i)}_{w,z}$.
\end{enumerate}
For $1$-periodic orbits $\{x_k\}_{k=1}^{p}$ and $y$ of $H$ and $H^{(p)}$, respectively, $m\in\Z_p$, $\alpha\in\{0,1\}$, and $0\leq l\leq 2K+1$, consider the following equation:
\begin{equation*}
\begin{cases}
\overbar{\partial}_Ju+X_{\cK^{(K,i)}_{w(s), [s,t]}}(u)^{0, 1}=0\\
\partial_sw+\nabla f(w)=0
\end{cases}
\end{equation*}
subject to
\begin{equation*}
\begin{cases}
\lim_{s\to-\infty}(u([s, t]), w(s))=(x_i(t), Z_\alpha^0)\\
\lim_{s\to+\infty}(u([s, t]), w(s))=(y(t-m/p), Z_l^m).
%\int_{{u}}\om+\int^1_0 H^{(p)}(t,{y}(t))dt-\sum_{i}H^{(p)}(t,{x_i}(t))dt =\lambda.
\end{cases}
\end{equation*}
Counting solutions to this equation, yields the map
\begin{equation*}
	f^{(K,i)}_{\alpha,l}:(\cf(H,\lamzero_{\bF_p})^{\otimes p}\otimes\lamzero_{\bF_p}\brat{\theta},d_{\Z_p})\rightarrow(\cf(H^{(p)},\lamzero_{\bF_{p}})\otimes\lamzero_{\bF_p}\brat{\theta},\delta^{(K,i)}).
\end{equation*}
Finally, one defines $\mathfrak{f}^{(K,i)}=\{f^{(K,i)}_{l}\}^K_{l=0}$ where
\begin{align*}
    f_l^{(K, i)}(({x}_1\otimes\cdots\otimes {x}_p)\otimes 1) &=f_{0, 2l}^{(K, i)}({x}_1\otimes\cdots\otimes {x}_p)\otimes 1+f_{0, 2l+1}^{(K, i)}({x}_1\otimes\cdots\otimes {x}_p)\otimes\theta\\
    f_l^{(K, i)}(({x}_1\otimes\cdots\otimes {x}_p)\otimes\theta) &=f_{1, 2l}^{(K, i)}({x}_1\otimes\cdots\otimes {x}_p)\otimes 1+f_{1, 2l+1}^{(K, i)}({x}_1\otimes\cdots\otimes {x}_p)\otimes\theta
\end{align*}

It follows from \cite[Lemma 52]{sugimoto2021hofer} that $\{\mathfrak{f}^{(K, i)}\}$ determines a morphism between the $\epsilon$-gapped directed family $$\{\text{CF}(H, \Lambda_{\mathbb{F}_p}^0)^{\otimes p}\otimes\Lambda_{\mathbb{F}_p}^0\langle\theta\rangle\}$$ and the $\epsilon$-directed family $$\{\text{CF}(H^{(p)}, \Lambda_{\mathbb{F}_p}^0)\otimes\Lambda_{\mathbb{F}_p}^0\langle\theta\rangle\}.$$ Then by \cite[Proposition 44]{sugimoto2021hofer}, there is an $\epsilon$-gapped $X_{\infty}$-morphism between the $\epsilon$-gapped $X_\infty$-modules $\text{C}(\mathbb{Z}_p, \text{CF}(\phi, \Lambda_{\mathbb{F}_p}^0)^{\otimes p})$ and $\text{C}(\mathbb{Z}_p, \text{CF}(\phi^p, \Lambda_{\mathbb{F}_p}^0))$. Finally by \cite[Lemma 16]{sugimoto2021hofer}, it induces an isomorphism between the homologies $\text{H}(\mathbb{Z}_p, \text{CF}(\phi, \Lambda_{\mathbb{F}_p}^0)^{\otimes p})$ and $\text{H}(\mathbb{Z}_p, \text{CF}(\phi^p, \Lambda_{\mathbb{F}_p}^0))$.

\begin{remark}
 The proof of \cite[Lemma 16]{sugimoto2021hofer} for the semipositive case is given on \cite[Page 80]{sugimoto2021hofer}.
\end{remark}

\subsection{The Smith type inequality for total bar-lengths}
With the equivariant pair-of-pants product defined, we can prove the following theorem in the same way as in \cite{shelukhin2022hofer}; we sketch the argument for the reader's convenience. 

\begin{thm}\label{thm: smith}
Let $\phi\in\Ham(M, \omega)$ be a Hamiltonian diffeomorphism of a closed semipositive symplectic manifold $(M, \omega)$. Suppose that $\fix(\phi^p)$ is finite. Then
\begin{center}
    $p\cdot\beta_{\mathrm{tot}}(\phi, \Lambda_{\bF_p})\le\beta_{\mathrm{tot}}(\phi^p, \Lambda_{\bF_p})$.
\end{center}
\end{thm}
\begin{proof}
{\textbf{Case 1: $\phi$ is nondegenerate}}
First, note that $\hf(\phi, \lamzero_{\bF_{p}})\otimes_{\lamzero_{\bF_p}}\lamzero_{\cK}\brat{\theta}$ is isomorphic to $\rH(\Z_p, \cf(\phi, \lamzero_{\bF_{p}}))$ since the $\Z_p$-action on $\cf(\phi, \lamzero_{\bF_{p}})$ is trivial. Denote by \[\beta_1,\dots,\beta_s\] the bar-lengths of $\h(\Z_p,\cf(\phi,\lamzero_{\bF_p}))$. Then, $s=2K(\phi,\bF_p)$ and $\beta_{2i-1}=\beta_{2i}=\beta_{i}(\phi,\Lambda_{\bF_p})$ since
\begin{equation*}
	\hf(\phi,\lamzero_{\bF_p})\otimes_{\lamzero_{\bF_p}}\lamzero_{\cK}\brat{\theta}\cong (\hf(\phi,\lamzero_{\bF_p})\otimes_{\lamzero_{\bF_p}}\lamzero_{\cK})\oplus(\hf(\phi,\lamzero_{\bF_p})\otimes_{\lamzero_{\bF_p}}\lamzero_{\cK}).
\end{equation*}
Secondly, as in \cite[Section 7.2]{shelukhin2022hofer}, the map \[x\in\cf(\phi,\lamzero_{\bF_p})\mapsto x\otimes\cdots\otimes x\in\cf(\phi,\lamzero_{\bF_p})^{\otimes p}\] induces a quasi-Frobenius isomorphism \[r_{p}^{*}\h(\Z_p,\cf(\phi,\lamzero_{\bF_p}))\ra\h(\Z_p,\cf(\phi,\lamzero_{\bF_p})^{\otimes p}),\] where $r_p:\lamzero_{\cK}\ra\lamzero_{\cK}$ is the homomorphism defined by $T\mapsto T^{\frac{1}{p}}$. By adjusting by $r_p$, the isomorphism is defined over $\lamzero_{\cK}$. Denote the bar-lengths of $\h(\Z_p,\cf(\phi,\lamzero_{\bF_p})^{\otimes p})$ by $\beta_{p,1},\cdots,\beta_{p,s}$, which are also the bar-lengths of $r_{p}^{*}\h(\Z_p,\cf(\phi,\lamzero_{\bF_p}))$. Then, $\beta_{p,i}=p\cdot\beta_{i}$ by the map $r_{p}^{*}$ and $s=2K(\phi,\bF_p)$. Thus, $\beta_{p,2i-1}=\beta_{p,2i}=p\cdot\beta_{i}(\phi,\Lambda_{\bF_p})$. Next, by the equivariant pair-of-pants product, there is an isomoprhim \[\h(\Z_p,\cf(\phi,\lamzero_{\bF_p})^{\otimes p})\ra\h(\Z_p,\cf(\phi^p,\lamzero_{\bF_p}))\] over $\lamzero_{\cK}$. Denote by $\hat\beta_{1},\cdots,\hat\beta_{s}$ the bar-lengths of $\h(\Z_p,\cf(\phi^p,\lamzero_{\bF_p}))$. Then, $\hat\beta_{2i-1}=\hat\beta_{2i}=p\cdot\hat\beta_{i}(\phi,\Lambda_{\bF_p})$ as in \cite[Corollary 23]{shelukhin2022hofer}. Finally, by \cite[Proposition 17 and Lemma 18]{shelukhin2022hofer} we have that $\hat\beta_{tot}=\sum\hat\beta_i\leq2\cdot\beta_{tot}(\phi^p,\Lambda_{\bF_p})$, in particular, $2p\cdot\beta_{tot}(\phi,\Lambda_{\bF_p})\leq2\cdot\beta_{tot}(\phi^p,\Lambda_{\bF_p})$. When the fixed points are degenerate, one can argues using local equivariant Floer homology as in \cite[Section 7.4]{shelukhin2022hofer}.

\textbf{Case 2: $\phi$ is degenerate}
Assume that $\phi=\phi_H^1$ has finitely many isolated fixed points. Let $H_1$ be a $C^2$-small nondegenerate perturbation of $H$ and $\phi_1=\phi_{H_1}^1$. By Theorem~\ref{thm: perturbation}, $\text{CF}(\phi_1, \Lambda_{\mathbb{F}_p}^0)$ is $\delta_0$-quasi-equivalent to the canonical $\Lambda^0$-complex $\text{CF}(\phi, \Lambda_{\mathbb{F}_p}^0)$, and the number of bar-lengths of $\text{CF}(\phi, \Lambda_{\mathbb{F}_p}^0)$ is bounded above by a number $k(\phi):=K(\phi,{\bF_p})$ independent of the perturbation $\phi_1$. Furthermore, the bar-lengths of $\text{CF}(\phi_1, \Lambda_{\mathbb{F}_p}^0)$ and $\text{CF}(\phi, \Lambda_{\mathbb{F}_p}^0)$ are $2\delta_0$-close. We emphasize that the number of finite bars of $\text{CF}(\phi_1, \Lambda_{\mathbb{F}_p}^0)$ length smaller than $2\delta_0$, a priori, depends on the perturbation $\phi_1$. Nonetheless, for any positive integer $m$, the sum of the last $m$ bar-lengths of the bar-length spectrums of $\text{CF}(\phi_1, \Lambda_{\mathbb{F}_p}^0)$ and $\text{CF}(\phi, \Lambda_{\mathbb{F}_p}^0)$ differ by at most $2m\delta_0$. Here, if $m>k(\phi)$ we are setting the initial $m-k(\phi)$ bar-lengths of  $\text{CF}(\phi, \Lambda_{\mathbb{F}_p}^0)$ to be $0$. 

It follows that the doubled $p$-scaled bar-lengths of $\text{CF}(\phi, \Lambda_{\mathbb{F}_p}^0)$ is $p\cdot 2\delta_0$-close to those of $\text{C}(\mathbb{Z}_p, \text{CF}(\phi_1, \Lambda_{\mathbb{F}_p}^0)^{\otimes p})$ and, hence, for any positive integer $m$, the sum of the last $m$ bar-lengths of their respective bar-length spectrums differs by at most $2pm\delta_0$. 

By Section \ref{sec:deg_equiv}, $\text{C}(\mathbb{Z}_p, \text{CF}(\phi^p, \Lambda_{\mathbb{F}_p}^0))$ is $3\delta_1$-quasi-equivalent to $\text{C}(\mathbb{Z}_p, \text{CF}(\phi_1^p, \Lambda_{\mathbb{F}_p}^0))$. Since $\phi_1$ is nondegenerate, by the equivariant pair-of-pants product, $\text{C}(\mathbb{Z}_p, \text{CF}(\phi_1, \Lambda_{\mathbb{F}_p}^0)^{\otimes p})$ is $0$-quasi-equivalent to $\text{C}(\mathbb{Z}_p, \text{CF}(\phi_1^p, \Lambda_{\mathbb{F}_p}^0))$. Thus, the doubled $p$-scaled bar-lengths of $\text{CF}(\phi, \Lambda_{\mathbb{F}_p}^0)$ is $2(p\delta_0+3\delta_1)$-close to those of $\text{C}(\mathbb{Z}_p, \text{CF}(\phi^p, \Lambda_{\mathbb{F}_p}^0))$. Note that the numbers of bar-lengths of the canonical complexes $\text{CF}(\phi, \Lambda_{\mathbb{F}_p}^0)$ and $\text{C}(\mathbb{Z}_p, \text{CF}(\phi^p, \Lambda_{\mathbb{F}_p}^0))$ are bounded above by a number $k'(\phi)$ independent of the choice of the perturbation $\phi_1$. Thus, in this case, the $2(p\delta_0+3\delta_1)$-quasi-equivalence implies that 
\begin{align*}
        2p\cdot\beta_{tot}(\phi, \Lambda_{\mathbb{F}_p})-k'(\phi)\cdot 2(p\delta_0+3\delta_1)\le\beta_{tot}(\text{C}(\mathbb{Z}_p, \text{CF}(\phi^p, \Lambda_{\mathbb{F}_p}^0)))
\end{align*}
Now by \cite[Proposition 17 and Lemma 18]{shelukhin2022hofer}, we obtain
\begin{align*}
        \beta_{tot}(\text{C}(\mathbb{Z}_p, \text{CF}(\phi^p, \Lambda_{\mathbb{F}_p}^0)))\le 2\beta_{tot}(\text{CF}(\phi^p, \Lambda_{\mathbb{F}_p}^0), T^{\delta_1}d_F)
\end{align*}
Note that the number of the bar-lengths of $\text{CF}(\phi^p, \Lambda_{\mathbb{F}_p}^0)$ is bounded above by $k(\phi^p):=K(\phi^p,{\bF_p})$ independent of the perturbation $\phi_1^p$. Then
\begin{align*}
        2\beta_{tot}(\text{CF}(\phi^p, \Lambda_{\mathbb{F}_p}^0), T^{\delta_1}d_F)\le 2\beta_{tot}(\phi^p, \Lambda_{\mathbb{F}_p})+4k(\phi^p)\delta_1
\end{align*}
Thus
\begin{align*}
        2p\cdot\beta_{tot}(\phi, \Lambda_{\mathbb{F}_p})\le 2\beta_{tot}(\phi^p, \Lambda_{\mathbb{F}_p})+4k(\phi^p)\delta_1+k'(\phi)\cdot2(p\delta_0+3\delta_1)
\end{align*}
Note that $\delta_0$ tends to $0$ as $\phi_1$ tends to $\phi$ and $\delta_1=(4n+2)\delta_0$. We have 
\begin{align*}
        2p\cdot\beta_{tot}(\phi, \Lambda_{\mathbb{F}_p})\le 2\beta_{tot}(\phi^p, \Lambda_{\mathbb{F}_p})
\end{align*}
\end{proof}

%\begin{remark}
%Note that we do not need the $\mathbb{Z}_p$-equivariant pair-of-pants product for degenerate $\phi$ to prove Theorem \ref{thm: smith}.
%\end{remark}

\section{Proof of Theorem \ref{main}}
\label{sec:proof_main}
Once we have Theorem \ref{thm:upper_bound} and Theorem \ref{thm: smith} under the semipositive setting, the proof of Theorem \ref{main} is almost identical to that of \cite[Theorem A]{shelukhin2022hofer}.

\textbf{Case 1: $\text{char}\left(\mathbb{K}\right)=0$} 
%When $\phi=\phi_H$ is nondegenerate, consider the Floer chain complex $\left(\text{CF}\left(H, \Lambda_{\mathbb{K}}\right), d\right)$ and let 
%\begin{align*}
%N\left(\phi, \mathbb{K}\right) &=\dim_{\Lambda_{\mathbb{K}}}\text{CF}\left(H, \Lambda_{\mathbb{K}}\right),\\ 
%B\left(\phi, \mathbb{K}\right) &=\dim_{\Lambda_{\mathbb{K}}}\text{HF}\left(H, \Lambda_{\mathbb{K}}\right),\\
%K\left(\phi, \mathbb{K}\right) &=\dim_{\Lambda_{\mathbb{K}}}\text{Im}(d). 
%\end{align*}

Consider the canonical chain complex $\left(\text{CF}\left(\phi_H, \Lambda_{\mathbb{K}}\right), d\right)$ described in Theorem \ref{thm: perturbation} and let
\begin{align*}
    N\left(\phi, \mathbb{K}\right) &=\dim_{\Lambda_{\mathbb{K}}}\text{CF}\left(\phi, \Lambda_{\mathbb{K}}\right)=\sum_{x\in\text{Fix}(\phi)}\dim_{\mathbb{K}}\text{HF}^{loc}\left(\phi, x\right)\\
    B\left(\phi, \mathbb{K}\right) &=\dim_{\Lambda_{\mathbb{K}}}\text{H}\left(\text{CF}\left(\phi_H, \Lambda_{\mathbb{K}}\right), d\right)\\
    K\left(\phi, \mathbb{K}\right) &=\dim_{\Lambda_{\mathbb{K}}}\text{Im}(d). 
\end{align*}

As in \cite[Lemma 19]{shelukhin2022hofer}, we have 
\begin{align*}
    N\left(\phi, \mathbb{K}\right)=B\left(\phi, \mathbb{K}\right)+2K\left(\phi, \mathbb{K}\right).
\end{align*}

Recall that we assume $N\left(\phi, \mathbb{K}\right)>\dim_{\mathbb{K}}\text{H}_*\left(M; \mathbb{K}\right)$ in Theorem \ref{main}. When $\phi$ is nondegenerate, we have 
\begin{align*}
    \beta_{\text{tot}}\left(\phi, \Lambda_{\mathbb{K}}\right)>0
\end{align*}
where $\beta_{\text{tot}}$ is defined as in Section \ref{sec:SVD}.
By Lemma \ref{lem: barcode_coeff}, we have
\begin{align*}
    \beta_{\text{tot}}\left(\phi, \Lambda_{\mathbb{F}_p}\right)>0
\end{align*}
for sufficiently large $p$.

When $\phi$ is degenerate, we set $$\beta_{\text{tot}}(\phi,\Lambda_{\mathbb{K}})=\displaystyle\sum_{j=1}^{K(\phi, \mathbb{K})}\beta_j(\phi, \Lambda_{\mathbb{K}}).$$ Here, $\beta_j(\phi, \Lambda_{\mathbb{K}})$ is defined in Theorem \ref{thm: perturbation} as the limit of $\beta^\prime_j(\phi, \Lambda_{\mathbb{K}})$ where $\beta^\prime_j(\phi, \Lambda_{\mathbb{K}})$ are the bar-length of the canonical complex $\text{CF}(\phi, \Lambda_{\mathbb{K}})$. Thus we still have 
\begin{align*}
    \beta_{\text{tot}}(\phi, \Lambda_{\mathbb{K}})>0
\end{align*}
by the assumption of Theorem \ref{main}. For a nondegenerate perturbation $\phi_1$ of $\phi$, and a sufficiently large prime $p$, $\beta_{\text{tot}}(\phi_1, \Lambda_{\mathbb{F}_p})=\beta_{\text{tot}}(\phi_1, \Lambda_{\mathbb{K}})$ and $\beta_j(\phi, \Lambda_{\mathbb{K}})$ (resp. $\beta_j(\phi, \Lambda_{\mathbb{F}_p})$) is close enough to $\beta_{K^\prime+j}(\phi_1, \Lambda_{\mathbb{K}})$ (resp. $\beta_{K^\prime+j}(\phi_1, \Lambda_{\mathbb{F}_p})$). Thus,
\begin{align*}
    \beta_{\text{tot}}(\phi, \Lambda_{\mathbb{F}_p})\geq  \beta_{\text{tot}}(\phi, \Lambda_{\mathbb{K}})/2>0
\end{align*}

Now, by Theorem \ref{thm: smith},
\begin{align}\label{1}
    0<p\cdot\beta_{\text{tot}}\left(\phi, \Lambda_{\mathbb{F}_p}\right)\le\beta_{\text{tot}}\left(\phi^p, \Lambda_{\mathbb{F}_p}\right)
\end{align}
for sufficiently large $p$. By the definition of boundary depth, 
\begin{align}\label{2}
\beta_{\text{tot}}\left(\phi^p, \Lambda_{\mathbb{F}_p}\right)\le K(\phi^p, \mathbb{F}_p)\cdot \beta(\phi, \Lambda_{\mathbb{F}_p})
\end{align}
By Theorem \ref{thm:upper_bound},
\begin{align}\label{3}
    \beta(\phi, \Lambda_{\mathbb{F}_p})\le C
\end{align}
where $C$ is independent of $p$ for sufficiently large $p$.
Combining Equations \ref{1}, \ref{2}, and \ref{3}, we have
\begin{align*}
K(\phi^p, \mathbb{F}_p)\ge \dfrac{p\cdot\beta_{\text{tot}}(\phi, \Lambda_{\mathbb{F}_p})}{C}>0
\end{align*}

Let $p_0$ be a sufficiently large prime such that for all primes $p>p_0$ we have: $(1)$ the bar-length spectrum of $\phi_1$ over $\Lambda_{\bF_p}$ and over $\Lambda_{\bK}$ coincide as in Lemma \ref{lem: barcode_coeff}; $(2)$ the dimension of $\h_*(M;\bF_p)$ is equal to that of $\h_*(M;\bK)$ and $ N(\phi, \mathbb{K})=N(\phi, \mathbb{F}_p)$ as in Remark \ref{rmk: bar_field_ext}; $(3)$ the iteration $\phi^p$ of $\phi$ is admissible in sense of Theorem \ref{A}; and $(4)$ such that \eqref{3} is satisfied.

Next, assume by contradiction that there are no simple contractible fixed points of $\phi^p$ for a prime $p>p_0$. Then, we have $N(\phi^p, \mathbb{K})=N(\phi, \mathbb{K})$ for any coefficient field $\bK$ since $\phi^p$ is an admissible iteration.

Thus, we have 
\begin{align*}
    N(\phi, \mathbb{K})=N(\phi, \mathbb{F}_p)= N(\phi^p, \mathbb{F}_p) &=B(\phi^p, \mathbb{F}_p)+2K(\phi^p, \mathbb{F}_p)\\
    &\ge B(\phi^p, \mathbb{F}_p)+\dfrac{2p\beta_{\text{tot}}(\phi, \Lambda_{\mathbb{F}_p})}{C}\\
    &\ge B(\phi, \mathbb{K})+\dfrac{p\beta_{\text{tot}}(\phi, \Lambda_{\mathbb{K}})}{C}
\end{align*}
This is a contradiction whenever $$p> C':=\frac{C(N(\phi, \mathbb{K})-B(\phi, \mathbb{K}))}{\beta_{\text{tot}}(\phi, \Lambda_{\mathbb{K}})}.$$ Therefore, $\phi^p$ must have a simple contractible fixed point whenever $p>\max\{p_0,C'\}$.

\textbf{Case 2: $\text{char}(\mathbb{K})=p$}

\begin{claim}
 Assume $\qh_{ev}(M; \Lambda_{\mathbb{K}})$ is semisimple. Then $\qh_{ev}(M; \Lambda_{\mathbb{F}_p})$ is semisimple.
\end{claim}
\begin{proof}
  We have that $$\text{QH}_{ev}(M; \Lambda_{\mathbb{K}})=\text{QH}_{ev}(M; \Lambda_{\mathbb{F}_p})\otimes_{\Lambda_{\mathbb{F}_p}}\Lambda_{\mathbb{K}}.$$ Therefore, if $\text{QH}_{ev}(M; \Lambda_{\mathbb{K}})$ has a nilpotent ideal then $\text{QH}_{ev}(M; \Lambda_{\mathbb{F}_p})$ has a nilpotent ideal, which would contradict the assumption that $\text{QH}_{ev}(M; \Lambda_{\mathbb{K}})$ is semisimple. Thus, the Jacobson radical of $\text{QH}_{ev}(M; \Lambda_{\mathbb{F}_p})$ is zero, i.e. $\text{QH}_{ev}(M; \Lambda_{\mathbb{F}_p})$ is semisimple.
\end{proof}

Assume to reach a contradiction that there exists $k_0$ such that $\text{Fix}(\phi^{p^k})=\text{Fix}(\phi^{p^{k_0}})$ for all $k\ge k_0$ since otherwise $\phi$ has infinitely many contractible periodic points.

Set $\psi=\phi^{p^{k_0}}$. There is $k_1$ such that $\psi^{p^k}=\phi^{p^{k_0+k}}$ is an admissible iteration of $\psi$ for all $k\ge k_1$. Then by Theorem \ref{A}, $N(\psi^{p^{k}}, \mathbb{F}_p)=N(\psi, \mathbb{F}_p)$.

\begin{lem}
    There is an upper bound $C$, independent of $k$, on the boundary depth $\beta(\phi^{p^{k}}, \Lambda_{\mathbb{F}_p})$.
\end{lem}
\begin{proof}
    Let $E=\{e_i\}$ be a collection of idempotents in $\text{QH}_{ev}(M; \Lambda_{\mathbb{F}_p})$. Then by taking $\delta=\displaystyle\max_{i}l(e_i)$ as in the proof of Proposition \ref{prop:gamma_adapted_bound}, we have $\gamma_E\left(\phi^{p^{k}}\right)\le D$ where $D$ is independent of $k$. Then, by Proposition \ref{prop:interleaving} we can get an upper bound on the boundary depth which is independent of $k$.
\end{proof}
\begin{remark}
    In this case $C$ depends on $p$, which is different from the case where $\text{char}(\mathbb{K})=0$.
\end{remark}

Assuming that $N(\phi, \mathbb{K})>\dim_{\mathbb{K}}\text{H}_*(M; \mathbb{K})$ in Theorem \ref{main}, we have $\beta_{\text{tot}}(\phi, \Lambda_{\mathbb{F}_p})>0$ since $\bK$ is a field extension of $\bF_p$. Then, for all $k\geq k_1$, 
\begin{align*}
    0<p^{k_0+k}\beta_{\text{tot}}(\phi, \Lambda_{\mathbb{F}_p}) &\le \beta_{\text{tot}}\left(\psi^{p^{k}}, \Lambda_{\mathbb{F}_p}\right)\\
    & \le K\left(\psi^{p^{k}}, \mathbb{F}_p\right)\cdot\beta\left(\psi^{p^{k}}, \Lambda_{\mathbb{F}_p}\right)\\
    & \le K\left(\psi^{p^{k}}, \mathbb{F}_p\right)\cdot C.
\end{align*}
Thus,
\begin{align*}
    K\left(\psi^{p^{k}}, \mathbb{F}_p\right)
    \ge\dfrac{p^{k_0+k}\beta_{\text{tot}}(\phi, \Lambda_{\mathbb{F}_p})}{C}.
\end{align*}
Now,
\begin{align*}
    N\left(\psi, \mathbb{F}_p\right)=N\left(\psi^{p^{k}}, \mathbb{F}_p\right) &=B\left(\psi^{p^{k}}, \mathbb{F}_p\right)+2K\left(\psi^{p^{k}}, \mathbb{F}_p\right)\\
    &\ge B\left(\psi, \mathbb{F}_p\right)+2\dfrac{p^{k_0+k}\beta_{\text{tot}}(\phi, \Lambda_{\mathbb{F}_p})}{C}
\end{align*}
which is a contradiction for sufficiently large $k$.\qed

\bibliographystyle{alpha}
\bibliography{BibliographyHZSP}
\end{document}